\documentclass[12pt,a4paper,twoside,reqno]{amsart}

\usepackage{amsmath, amssymb, amsthm, enumerate, graphicx}

\newcommand{\IR}{\mathbb{R}}
\newcommand{\IC}{\mathbb{C}}
\newcommand{\IH}{\mathbb{H}}

\newcommand{\IO}{\mathbb{O}}
\newcommand{\IK}{\mathbb{K}}
\newcommand{\WW}{\mathcal{W}}
\newcommand{\CC}{\mathcal{C}}
\newcommand{\af}{\mathfrak{a}}
\newcommand{\pp}{\mathfrak{p}}
\newcommand{\nn}{\mathfrak{n}}

\newcommand{\ov}[1]{\overline{#1}}
\newcommand{\un}[1]{\underline{#1}}
\DeclareMathOperator{\Ric}{Ric}
\DeclareMathOperator{\Int}{Int}
\DeclareMathOperator{\tr}{tr}
\DeclareMathOperator{\id}{id}
\DeclareMathOperator{\dist}{dist}

\DeclareMathOperator{\Sym}{Sym}
\DeclareMathOperator{\vol}{vol}

\DeclareMathOperator{\supp}{supp}
\DeclareMathOperator{\pr}{pr}
\DeclareMathOperator{\im}{Im}

\DeclareMathOperator{\End}{End}
\DeclareMathOperator{\Stab}{Stab}
\DeclareMathOperator{\Rm}{Rm}

\DeclareMathOperator{\spann}{span}
\DeclareMathOperator{\ctanh}{ctanh}

\DeclareMathOperator{\DIV}{div}
\DeclareMathOperator{\proj}{proj}
\DeclareMathOperator{\sh}{sh}
\DeclareMathOperator{\ch}{ch}
\DeclareMathOperator{\newtanh}{th}
\DeclareMathOperator{\cth}{cth}
\DeclareMathOperator{\Ad}{Ad}
\renewcommand{\sinh}{\sh}
\renewcommand{\cosh}{\ch}
\renewcommand{\tanh}{\newtanh}
\renewcommand{\ctanh}{\cth}
\newcommand{\dotcup}{\ensuremath{\mathaccent\cdot\cup}}
\newcommand{\EMPTY}[1]{}
\newcommand{\TT}[3]{\Big( \begin{matrix} #1 \; #2 \\ #3 \end{matrix} \Big)}
\newcommand{\tTT}[3]{\big( \begin{smallmatrix} #1 \; #2 \\ #3 \end{smallmatrix} \big)}

\newtheorem{Theorem}{Theorem}[section]
\newtheorem{Question}[Theorem]{Question}
\newtheorem{Lemma}[Theorem]{Lemma}
\newtheorem{Corollary}[Theorem]{Corollary}
\newtheorem{Proposition}[Theorem]{Proposition}
\newtheorem{Definition}[Theorem]{Definition}
\numberwithin{equation}{section}

\title[Stability of symmetric spaces under Ricci flow]{Stability of symmetric spaces of \\ noncompact type under Ricci flow}
\author{Richard H Bamler}
\address{Stanford University, Department of Mathematics, Stanford, CA 94305}
\email{rbamler@math.stanford.edu}
\date{\today}

\begin{document}

\begin{abstract}
In this paper we establish stability results for symmetric spaces of noncompact type under Ricci flow, i.e. we show that any small perturbation of the symmetric metric is flown back to the original metric under an appropriately rescaled Ricci flow.

It will be important for us which smallness assumptions we have to impose on the initial perturbation.
We will find that as long as the symmetric space does not contain any hyperbolic or complex hyperbolic factor, we don't have to assume any decay on the perturbation.
Furthermore, in the hyperbolic and complex hyperbolic case, we show stability under a very weak assumption on the initial perturbation.
This will generalize a result obtained by Schulze, Schn\"urer and Simon (\cite{SSS}) in the hyperbolic case.

The proofs of these results make use of an improved $L^1$-decay estimate for the heat kernel in vector bundles over symmetric spaces, which is of independent interest.
\end{abstract}

\maketitle
\tableofcontents

\section{Introduction}
\subsection{Stability of symmetric spaces}
Consider a locally symmetric space $(M, \ov g)$, i.e. a Riemannian manifold which locally has a reflection symmetry at every point (for more details see section \ref{sec:symsp}).
By the de Rham Decomposition Theorem, its universal cover $\tilde M$ can be expressed as a product $M_1 \times \ldots \times M_m$ of irreducible symmetric spaces.
All $M_i$ are Einstein.
If all Einstein constants $\lambda_i$ are negative, then $M$ is said to be of \emph{noncompact type}.
Furthermore, if all $\lambda_i$ are equal to some $\lambda < 0$, then $(M, \ov g)$ is Einstein itself with Einstein constant $\lambda$.
Hence it is a fixed point of the rescaled Ricci flow equation
\begin{equation} \partial_t g_t = - 2 \Ric_{g_t} + 2 \lambda g_t. \label{eq:RF} \end{equation}

In this paper, we will prove stability results for the metric $\ov g$, i.e. we will show that every sufficiently small perturbation $g_0 = \ov g + h$ flows back to $\ov g$ under (\ref{eq:RF}) as $t \to \infty$.
Surprisingly, for most symmetric spaces we don't have to impose any spatial decay assumption on the perturbation $h$.

\begin{Theorem} \label{Thm:mainA}
Let $(M, \ov g)$ be a locally symmetric space of noncompact type which is Einstein of Einstein constant $\lambda < 0$ and assume that the de Rham decomposition of $\tilde M$ contains no factors which are homothetic to $\IH^n, (n \geq 2)$ or $\IC \IH^{2n}, (n \geq 1)$.
Then there is an $\varepsilon > 0$ depending only on $\tilde M$ such that if
\[ (1-\varepsilon) \ov g < g_0 < (1+\varepsilon) \ov g, \]
and if $(g_t)$ evolves by (\ref{eq:RF}), then $g_t$ exists for all time $t$ and as $t \to \infty$ we have convergence $g_t \longrightarrow \ov g$ in the pointed Cheeger-Gromov sense, i.e. there is a family of diffeomorphisms $\Psi_t$ of $M$ such that $\Psi_t^* g_t \longrightarrow \ov g$ and $\Psi_t \to \Psi_\infty$ in the smooth sense on every compact subset of $M$.
\end{Theorem}

In the hyperbolic or complex hyperbolic case, we have to impose stronger assumptions on the perturbation.

\begin{Theorem} \label{Thm:mainB}
Let $(M, \ov g)$ be either $\IH^n$ for $n \geq 3$ or $\IC \IH^{2n}$ for $n \geq 2$, choose a basepoint $x_0 \in M$ and let $r = d(\cdot, x_0)$ denote the radial distance function. \\
There is an $\varepsilon_1 > 0$ and for every $q < \infty$ an $\varepsilon_2 = \varepsilon_2(q) > 0$ such that the following holds:
If $g _0 = \ov g + h$ and $h = h_1 + h_2$ satisfies
\[ |h_1| < \frac{\varepsilon_1}{r + 1} \qquad \text{and} \qquad \sup_M |h_2| +  \bigg( \int_M |h_2|^q dx \bigg)^{1/q} < \varepsilon_2, \]
then Ricci flow (\ref{eq:RF}) exists for all time and we have convergence $g_t \longrightarrow \ov g$ in the pointed Cheeger-Gromov sense.
\end{Theorem}

In the case $M = \IH^n$, $n \geq 4$ Schulze, Schn\"urer and Simon (\cite{SSS}) have shown stability for every perturbation $h$ for which $\Vert h \Vert_{L^\infty(M)}$ is bounded by a small constant depending on $\Vert h \Vert_{L^2(M)}$.
This result is implied by Theorem \ref{Thm:mainB} by the interpolation inequality.
Li and Yin (\cite{LY}) have shown a stability result for $M = \IH^n$, $n \geq 3$ when the Riemannian curvature approaches the hyperbolic curvature like $\varepsilon_1(\delta) e^{-\delta r}$.

One drawback of the decay assumption of Theorem \ref{Thm:mainB} is that we cannot generalize the stability to quotients of $M$ under a group action which does not fix the distance function $r$.
Results of this kind have to be proven separately.
For example, for compact quotients of hyperbolic space, a stability result was obtained by Ye in \cite{Ye} and for finite volume quotients (i.e. if there are cusps) by the author in \cite{BamCusps}.

Theorems \ref{Thm:mainA} and \ref{Thm:mainB} have the following immediate consequences:
\begin{Corollary} \label{Cor:mainA}
Let $(M, \ov g)$ be a locally symmetric space of noncompact type which is Einstein of Einstein constant $\lambda < 0$ and assume that the de Rham decomposition of $\tilde M$ contains no factors which are homothetic to $\IH^n, (n \geq 2)$ or $\IC \IH^{2n}, (n \geq 1)$.
Then there is an $\varepsilon > 0$ depending only on $\tilde M$ such the following holds:
If $g$ is an Einstein metric on $M$ with Einstein constant $\lambda$ and
\[ (1-\varepsilon) \ov g < g < (1+\varepsilon) \ov g, \]
then $g$ is isometric to $\ov g$.
\end{Corollary}

\begin{Corollary}
Let $(M, \ov g)$ be either $\IH^n$ for $n \geq 3$ or $\IC \IH^{2n}$ for $n \geq 2$, choose a basepoint $x_0 \in M$ and let $r = d(\cdot, x_0)$ denote the radial distance function. \\
There is an $\varepsilon_1 > 0$ and for every $q < \infty$ an $\varepsilon_2 = \varepsilon_2(q) > 0$ such that the following holds:
If $g$ is an Einstein metric on $M$ of the same Einstein constant as $\ov g$ and $g = \ov g + h_1 + h_2$ with
\[ |h_1| < \frac{\varepsilon_1}{r + 1} \qquad \text{and} \qquad \sup_M |h_2| +  \bigg( \int_M |h_2|^q dx \bigg)^{1/q} < \varepsilon_2, \]
then $g$ is isometric to $\ov g$.
\end{Corollary}

By results of Graham-Lee (\cite{GL}) and Biquard (\cite{Biq}), the spaces $\IH^n, (n \geq 4)$ and $\IC \IH^{2n}, (n \geq 2)$ admit deformations $g$ which are Einstein of the same Einstein constant, are not isometric to $\ov g$ and satisfy
\[ (1-\varepsilon) \ov g < g_0 < (1+\varepsilon) \ov g. \]
Hence for those spaces we cannot expect a result which is as strong as that of Theorem \ref{Thm:mainA} or Corollary \ref{Cor:mainA}.
However, the following questions remain:

\begin{Question}
Do we always have longtime existence of the Ricci flow for any such initial metrics and under what assumptions do we have convergence to an Einstein metric?
\end{Question}

\begin{Question}
Does Theorem \ref{Thm:mainA} also hold for $\IH^3$?
Note that every nearby Einstein metric is hyperbolic.
\end{Question}

\subsection{Heat kernel estimates in twisted vector bundles}
Theorems \ref{Thm:mainA} and \ref{Thm:mainB} rely on a careful analysis of heat kernels in homogeneous vector bundles over the symmetric space in question.
Our main interest here will lie in the bundle of symmetric bilinear forms, the bundle that perturbations of $\ov{g}$ live in, but our analysis does not limit to this vector bundle.
The results of our analysis are not related to Ricci flow and are of independent interest.
They provide a characterization of the $L^1$-decay behavior of heat kernels in homogeneous vector bundles over symmetric spaces of noncompact type.
The asserted decay rates are optimal in many cases.
Note that Carron (\cite{Car}) has obtained strong pointwise estimates on such heat kernels.
However, in the settings analyzed by the author, these estimates do not imply the optimal $L^1$-decay rate.

We will now summarize our heat kernel estimates (see subsection \ref{subsec:heatkerIntro} for more details).
In the following statements, we consider a simply-connected symmetric space $(M, \ov{g})$ of noncompact type and a homogeneous vector bundle $E$ over $M$.
Note that there is a canonical connection $\nabla^E$ on $E$.
Let $(k_t)_{t > 0} \in C^\infty (M; E) \otimes E^*_{p_0}$ be the heat kernel for the connection Laplacian $\triangle = - \nabla^{E*} \nabla^E = \sum_{i=1}^n (\nabla^E)^2_{v_i, v_i}$  ($(v_i)_{i = 1, \ldots, n}$ denotes a local orthonormal frame) centered at some point $p_0 \in M$, i.e.
\[ \partial_t k_t = \triangle k_t \qquad \text{and} \qquad k_t \xrightarrow{t \to 0} \delta_{p_0}  \id_{E_{p_0}}. \]
In subsection \ref{subsec:heatkerIntro} we will describe a recipe for computing constants $\lambda_L, \lambda_B \geq 0$ (if $M$ has rank $1$) or constants $\lambda_\CC, \lambda_0, \lambda_1 \geq 0$ (in the general rank case), based on the geometry of $E$ and $M$, such that the following Theorems hold:

\begin{Theorem}[Rank 1 case, cf Theorems \ref{Thm:dectrianglerank1}, \ref{Thm:Asymptriangle}] \label{Thm:restatementrank1}
Assume that $M$ has rank $1$.
Then there are constants $c > 0, C< \infty$ such that:

If $\lambda_B > \lambda_L$, then the exponential decay rate is exactly $- \lambda_L$, i.e.
\[ c e^{-\lambda_L t} < \Vert k_t \Vert_{L^1(M)} < C e^{-\lambda_L t} \qquad \text{for all} \qquad t > 0. \]
If $\lambda_B < \lambda_L$, then the exponential decay rate lies between $-\lambda_L$ and $-\lambda_B$, i.e.
\[ c e^{-\lambda_L t} < \Vert k_t \Vert_{L^1(M)} < C e^{-\lambda_B t} \qquad \text{for all} \qquad t > 0. \]
If $\lambda_B = \lambda_L$, the upper bound still holds with $\lambda_B$ replaced by any $\lambda < \lambda_B$ (where $C$ depends on $\lambda$).
More precisely, we have
\[ c e^{-\lambda_L t} < \Vert k_t \Vert_{L^1(M)} < C (\log(t+2))^{1/2} (t+2)^{a/2} e^{-\lambda_L t} \qquad \text{for all} \qquad t > 0 \]
where $a$ can be determined via the root system of $M$.

Moreover, we have the pointwise estimate
\[ |k_t(p)| < \frac{C}{\vol B_r(p_0)} e^{-\min \{ \lambda_L, \lambda_B \} t} \qquad \text{where} \qquad r = d (p_0, p). \]
\end{Theorem}

Note that the last assertion is neither stronger nor weaker than the first.

\begin{Theorem}[General rank case, cf Theorem \ref{Thm:dectrianglegenrank}] \label{Thm:restatementgenrank}
Assume that $M$ has arbitrary rank.
Then there are constants $c > 0, C < \infty$ such that:

If $\lambda_1 > \lambda_\CC$, then the exponential decay rate is exactly $-\lambda_0 = - \lambda_\CC$, i.e.
\[ c e^{-\lambda_0 t} < \Vert k_t \Vert_{L^1(M)} < C e^{-\lambda_0 t} \qquad \text{for all} \qquad t > 0. \]
If $\lambda_1 \leq \lambda_\CC$, then the upper bound still holds with $\lambda_0$ replaced by any $\lambda < \lambda_0$ (where $C$ depends on $\lambda$).
More precisely, there is an $A < \infty$ such that.
\[ c e^{-\lambda_\CC t} < \Vert k_t \Vert_{L^1(M)} < C (t+2)^A e^{-\lambda_0 t} \qquad \text{for all} \qquad t > 0. \]
\end{Theorem}

As an immediate corollary we obtain
\begin{Corollary} \label{Cor:dectriangle}
Assume that $M$ has arbitrary rank.
Consider a solution $(s_t)_{t \geq 0} \in C^\infty(M; E)$ to the heat eqation $\partial_t s_t = \triangle s_t$ which is bounded on compact time intervals.
Then if $\lambda < \lambda_0$ or $\lambda \leq \lambda_0$ and $\lambda_1 \not= \lambda_\CC$, we have
\[ \Vert s_t \Vert_{L^\infty(M)} \leq C e^{-\lambda t} \Vert s_0 \Vert_{L^\infty(M)} \]
and in the case $\lambda = \lambda_1 = \lambda_\CC$ we get
\[ \Vert s_t \Vert_{L^\infty(M)} \leq C (t+2)^A  e^{-\lambda t} \Vert s_0 \Vert_{L^\infty(M)}. \]
\end{Corollary}
Observe that we did not impose any spatial decay or compact support assumptions on $s_0$.

Furthermore, we can use Theorem \ref{Thm:dectrianglegenrank} to find an $L^1$-estimate on the associated Green's kernel which leads to an $L^\infty$-estimate of the Poisson equation:
\begin{Corollary} \label{Cor:Greentriangle}
Assume that $M$ has arbitrary rank.
Let $\lambda < \lambda_0$ and consider the Green's kernel $g \in C^\infty(M \setminus \{ p_0 \};E) \otimes E_0^*$ of the operator $-\triangle - \lambda$ centered in $p_0$.
Then $\Lambda = \Vert g \Vert_{L^1} < \infty$.

As a consequence, we obtain the estimate
\begin{equation} \label{eq:triangles}
 \Lambda \Vert \triangle s + \lambda s \Vert_{L^\infty} \geq  \Vert s \Vert_{L^\infty}
\end{equation}
for all bounded sections $s \in C^\infty(M;E)$.
In particular, $-\triangle - \lambda$ does not have $L^\infty$-bounded kernel elements.
\end{Corollary}

Recall that the constants $\lambda_L, \lambda_B$ or $\lambda_\CC, \lambda_0, \lambda_1$ are non-negative and depend on the homogeneous vector bundle $E$.
More specifically, they arise from a local computation in the symmetric space $M$ and the vector bundle $E$ and depend on the curvature of $M$ and $E$.
Our results are interesting for those vector bundles $E$ for which $\lambda_L, \lambda_B$ or $\lambda_\CC, \lambda_0, \lambda_1$ are even positive, which can only happen if the curvature of $E$ does not vanish.
It will be an essential aspect of our proofs to understand and use the curvature of $E$ to our advantage.
In fact, our results would not imply a positive decay rate if the curvature of $E$ vanished, i.e. if $E$ was flat.
For example, in this case the quantity $\Vert k_t \Vert_{L^1}$ would be constant in time, so Theorem \ref{Thm:dectrianglegenrank} could not imply a positive decay rate.
As for Corollary \ref{Cor:dectriangle}, observe that a simple application of the maximum principle already gives us $\Vert s_t \Vert_{L^\infty} \leq \Vert s_0 \Vert_{L^\infty}$ for any vector bundle $E$, which is sharp in the case in which $E$ is flat since constant solutions exist and are stationary.
On the other hand, if we consider the case in which $E$ has non-vanishing curvature, then Theorem \ref{Thm:dectrianglegenrank} and Corollary \ref{Cor:dectriangle} may yield a better decay rate, which, however, only becomes noticeable for large $t$.
Corollary \ref{Cor:Greentriangle} illustrates the effect of the curvature of $E$ in the most demonstrative way: If $E$ were flat, then any constant section $s \in C^\infty(M;E)$ would contradict inequality (\ref{eq:triangles}) already for $\lambda = 0$.
In the non-flat case, however, it may happen that no constant section exist. 
So the curvature of $E$ forces every section indirectly to have non-zero Laplacian.

\subsection{Outline of the paper}
The main ingredients of the proofs Theorems \ref{Thm:mainA} and \ref{Thm:mainB} are the heat kernel estimates in twisted vector bundles over symmetric spaces.
We will apply these estimates to the vector bundle $\Sym_2 T^*M$ whose sections are perturbations $h_t$ of the metric $\ov{g}$.
If $h_t$ is small, then the Ricci flow equation expressed in terms of $h_t$ can be approximated by the linearized Ricci deTurck flow equation.
This linearized deTurck flow equation is a heat equation with an extra zeroth order term, which just generates an additional exponential decay or growth rate.
Our goal will then be to use Theorems \ref{Thm:restatementrank1} and \ref{Thm:restatementgenrank} to estimate the exponential decay rate of the $L^1$-norm of the heat kernel associated to the linearized Ricci deTurck flow equation.
If this rate is positive, then this implies an exponential decay of the $L^\infty$-norm of any bounded solution of the linearized Ricci deTurck flow equation similarly as in Corollary \ref{Cor:dectriangle}.
In this case the stability of the nonlinear equation follows easily.

Next, we will compute the constants $\lambda_\CC, \lambda_0, \lambda_1$ and find that the obstruction against exponential decay of the linearized equation comes from so-called cusp deformations (see subsection \ref{subsec:nullspace}).
Those deformations correspond to the ``trivial Einstein deformations'' in \cite[sec 2.3]{BamDehn} and can be seen as algebraic deformations of cusp cross-sections.
Cusp deformations created also the major analytic issues in \cite{BamCusps}.
It will turn out that cusp deformations only exist for the spaces $\IH^n, (n \geq 3)$ and $\IC \IH^{2n}, (n \geq 2)$.
Theorem \ref{Thm:mainA} and Theorem \ref{Thm:mainB} for type $h_2$ perturbations will then follow immediately from this heat kernel estimate.
In order to allow type $h_1$ perturbations, we will use a trick from the geometry of negatively curved spaces.

The paper is organized as follows: In section \ref{sec:AnalyticalPrelim}, we discuss the Ricci flow and Ricci deTurck flow equation and give a short overview over all analytical tools needed in this paper.
Section \ref{sec:symsp} contains a brief introduction into the geometry of symmetric spaces.
In section \ref{sec:heatkernel}, we prove more abstract bounds on heat kernels in homogeneous vector bundles over symmetric spaces.
These bounds involve certain constants, which we will then estimate for our particular purpose in section \ref{sec:Einstop}.
Finally, section \ref{sec:proofs} contains the proofs of Theorems \ref{Thm:mainA} and \ref{Thm:mainB}.

\subsection{Acknowledgments} 
I would like to thank my advisor Gang Tian for his constant support.
Moreover, I am grateful to Hans-Joachim Hein, Robert Kremser, John Lott, Peter Sarnak and Anna Wienhard for many helpful discussions.

\section{Analytical preliminaries} \label{sec:AnalyticalPrelim}
\subsection{Ricci deTurck flow} \label{subsec:RdTflow}
In order to establish the desired stability results, we will analyze Ricci deTurck flow.
This flow is a modification of Ricci flow via a continuous family of diffeomorphisms.

Recall that the rescaled Ricci flow equation reads
\begin{equation} 
\dot{g}^{RF}_t = - 2 \Ric_{g^{RF}_t} + 2 \lambda g^{RF}_t. \label{eq:nRF}
\end{equation}
In order to define the Ricci deTurck flow, we need to make use of a distinguished background metric $\ov{g}$ which we will always choose to be the given symmetric metric on $M$.
Define the divergence operator
\[ \DIV_{\ov{g}} : C^\infty(M; \Sym_2 T^*M) \longrightarrow C^\infty(M; T M), \quad h \longmapsto - \sum_i (\ov\nabla_{\ov{e}_i} h(\ov{e}_i, \cdot))^{\ov{\#}} \]
where we sum over a local $\ov{g}$-orthonormal frame field $(\ov{e}_i)$ and the musical operator is also taken with respect to $\ov{g}$.
Set
\[ X_{\ov g}(h) = \DIV_{\ov{g}} h + \tfrac12 \ov\nabla \tr_{\ov{g}} h. \]
Then the Ricci deTurck flow equation reads
\begin{equation} \label{eq:RdTflow}
 \dot{g}^{DT}_t = - 2 \Ric_{g^{DT}_t} + 2\lambda g^{DT}_t - \mathcal{L}_{X_{\ov{g}}(g^{DT}_t)} g^{DT}_t.
\end{equation}
The advantage of Ricci deTurck flow over Ricci flow is that its linearization at $g_t = \ov{g}$ is strongly elliptic.
This fact has been used by deTurck to give a simplified proof for the short-time existence of Ricci flow (\cite{DeT}).
In fact, if we express equation (\ref{eq:RdTflow}) in terms of the perturbation $h_t = g_t^{DT} - \ov g$, we obtain (for this and the following computations compare with \cite[sec 2.2]{BamCusps})
\begin{equation} \label{eq:nonlinRdT} \partial_t h_t + L h_t = Q_t  \end{equation}
where $L$ is called \emph{Einstein operator} with
\[ (L h)_{ab} = - \triangle h_{ab} - 2 \ov{g}^{uv} \ov{g}^{pq} \ov{R}_{aupb} h_{vq} \]
and $Q_t$ only contains terms of higher order:
\begin{alignat*}{1}
Q_{ab} = &- g^{uv} g^{pq} ( \nabla_u h_{pa} \nabla_v h_{qb} - \nabla_p h_{ua} \nabla_v h_{qb} + \tfrac12 \nabla_a h_{up} \nabla_b h_{vq}) \\
&  - g^{uv}g^{pq} ( - \nabla_u h_{vp} + \tfrac12 \nabla_p h_{uv} ) (\nabla_a h_{qb} + \nabla_b h_{qa} - \nabla_q h_{ab}) \\
& - \ov{g}^{uv} \ov{g}^{pq} ( - \nabla_p h_{qv} + \tfrac12 \nabla_v h_{pq} ) \nabla_u h_{ab} \\ 
& - \ov{g}^{uv} \ov{g}^{pq} ( - \nabla^2_{bp} h_{qv} + \tfrac12 \nabla^2_{bv} h_{pq}) h_{au} 
 - \ov{g}^{uv} \ov{g}^{pq} ( - \nabla^2_{ap} h_{qv} + \tfrac12 \nabla^2_{av} h_{pq}) h_{bu} \\
& - ( g^{uv} - \ov{g}^{uv}) ( \nabla^2_{ua} h_{bv} + \nabla^2_{ub} h_{av} - \nabla^2_{uv} h_{ab} - \nabla^2_{ab} h_{uv} ).
\end{alignat*}
Hence if $|h| < 0.1$, we can estimate $|Q| \leq C ( |\nabla h|^2 + |h| |\nabla^2 h| )$.
We will also sometimes make use of the identity
\[ Q_t = R_t + \nabla^* S_t, \]
where
\begin{alignat*}{1}
R_{ab} = & - g^{uv} g^{pq} ( \nabla_u h_{pa} \nabla_v h_{qb} - \nabla_p h_{ua} \nabla_v h_{qb} + \tfrac12 \nabla_a h_{up} \nabla_b h_{vq}) \\
& - g^{uv} g^{pq} ( - \nabla_u h_{vp} + \tfrac12 \nabla_p h_{uv}) (\nabla_a h_{qb} + \nabla_b h_{qa} - \nabla_q h_{ab} ) \\
& + \ov{g}^{uv} \ov{g}^{pq} ( - 2 \nabla_p h_{qv} + \tfrac12 \nabla_v h_{pq} ) ( \nabla_a h_{bu} + \nabla_b h_{au} - \nabla_u h_{ab} ) \\
& + \ov{g}^{uv} \ov{g}^{pq} \nabla_a h_{pu} \nabla_b h_{qv}
\end{alignat*}
and $( \nabla^* S)_{ab} = - \ov{g}^{kl} \nabla_k S_{lab}$ with
\begin{alignat*}{1}
S_{lab} =& \ov{g}^{uv} \ov{g}^{pq} ( - \nabla_p h_{qv} + \tfrac12 \nabla_v h_{pq} ) (\ov{g}_{lb}  h_{au} +\ov{g}_{la} h_{bu} ) \\
& + \ov{g}_{lp} ( g^{pv} - \ov{g}^{pv} ) ( \nabla_a h_{bv} + \nabla_b h_{av} - \nabla_v h_{ab} ) - \ov{g}_{la} ( g^{uv} - \ov{g}^{uv} ) \nabla_b h_{uv}.
\end{alignat*}
Observe that as long as $|h| < 0.1$ we have
\[ |R| \leq C |\nabla h|^2 \qquad \text{and} \qquad |S| \leq C |h| |\nabla h|. \]

The following Proposition expresses the equivalence of Ricci deTurck flow and Ricci flow.
\begin{Proposition} \label{Prop:DTisRF}
 Let $(g_t^{DT})_{t \in [0,T)}$ be a smooth solution to the Ricci deTurck flow equation (\ref{eq:RdTflow}) and assume that $|g_t^{DT}-\ov{g}| < \varepsilon_0$ everywhere for some universal constant $\varepsilon_0 > 0$, which only depends on $(M, \ov{g})$.
 Define the time dependent vector field $X_t = X_{\ov g} (g_t^{DT})$.
 Then $X_t$ has a flow $(\Psi_t)_{t \in [0,T)}$, i.e. there is a family of diffeomorphisms $\Psi_t : M \to M$ such that
\[ \dot\Psi_t = X_t \circ \Psi_t \qquad \text{and} \qquad \Psi_0 = \id_M, \]
and $g_t = \Psi^*_t g_t^{DT}$ solves the normalized Ricci flow equation (\ref{eq:nRF}).
\end{Proposition}
\begin{proof}
For the existence of the flow $(\Psi_t)$ observe that we have $|X_t| \leq C (t^{-1/2} + 1)$ by Corollary \ref{Cor:Shi} below.
The fact that $g_t$ satisfies the normalized Ricci flow equation follows directly from (\ref{eq:RdTflow}).
\end{proof}

Hence, in order to establish Theorems \ref{Thm:mainA} and \ref{Thm:mainB}, it suffices to prove the stability for Ricci deTurck flow instead of Ricci flow.
As we will see later, the main work will go into establishing the stability of the \emph{linearized Ricci deTurck flow} equation
\begin{equation} \label{eq:linRdT} \partial_t h_t + L h_t = 0 \end{equation}

\subsection{A priori derivative estimates}
We recall an a priori derivative estimate for linear or a certain type of nonlinear parabolic equations.
If $\Omega \subset \IR^n \times \IR$ denotes some parabolic neighborhood in space-time (e.g. $\Omega = B_r(0) \times [0,T]$), then we denote by $C^{2m;m}(\Omega)$ the space of scalar or vector valued functions on $\Omega$ which are $i$ times differentiable in spatial direction and $j$ times differentiable in time direction whenever $i + 2j \leq 2m$.
For $\alpha \in (0,\frac12)$, the corresponding H\"older space is denoted by $C^{2m, 2\alpha; m, \alpha}(\Omega)$.

In order to present our results in a scaling invariant way, we use the following weights to define the H\"older norm on $C^{2m, 2\alpha; m, \alpha}(\Omega)$:
Assume 
\begin{equation*} r = \min \{ r' \; : \; \text{$\Omega \subset B_{r'}(p) \times [t-(r')^2, t]$ for some $p$, $t$} \} < \infty. \end{equation*}
Then set
\[ \Vert u \Vert_{C^{2m, 2\alpha; m, \alpha}(\Omega)} = \sum_{|\iota|+2k \leq 2m} r^{|\iota|+2k} (\Vert D^\iota \partial_t^k u \Vert_{C^0} + r^{2\alpha} [ D^\iota \partial_t^k u ]_{2\alpha,\alpha} ), \]
where $\iota$ runs over products of spatial derivatives.

Set $B_r = B_r(0) \subset \IR^n$.

\begin{Proposition} \label{Prop:Shi}
 Let $r > 0$ and consider the parabolic neighborhoods $\Omega = B_r \times [-r^2,0]$ and $\Omega' = B_{2r} \times [-4r^2,0]$.
 
 Assume that $u \in C^{2;1}(\Omega')$ satisfies the equation
 \begin{multline*} (\partial_t - L) u = R[u] = r^{-2} f_1 (r^{-1} x, u) \cdot u + r^{-1} f_2(r^{-1} x, u) \cdot \nabla u \\ + f_3(r^{-1} x, u) \cdot \nabla u \otimes \nabla u + f_4(r^{-1} x, u) \cdot u \otimes \nabla^2 u,  \end{multline*}
 where $f_1, \ldots, f_4$ are smooth functions in $x$ and $u$ such that $f_2, f_3, f_4$  can be paired with the tensors $u$, $\nabla u$, $\nabla u \otimes \nabla u$, $u \otimes \nabla^2 u$.
 Assume that the linear operator $L$ has the form
 \[
  L u = a_{ij}(x) \partial_{ij}^2 u + b_i(x) \partial_i u + c(x) u.
 \]
 
 Now assume that we have the following bounds for $m \geq 1$, $\alpha \in (0,\frac12)$:
 \[ \begin{split} \frac{1}{\Lambda} < a_{ij} < \Lambda, \quad \Vert a_{ij} \Vert_{C^{2m-2, 2\alpha; m-1, \alpha}(\Omega')} < \Lambda, \\ \quad \Vert b_{i} \Vert_{C^{2m-2, 2\alpha; m-1, \alpha}(\Omega')} < r^{-1} \Lambda, \quad \Vert c \Vert_{C^{2m-2, 2\alpha; m-1, \alpha}(\Omega')} < r^{-2} \Lambda. 
 \end{split}
 \]
 
 Then there are constants $\varepsilon_m > 0$ and $C_m < \infty$ depending only on $\Lambda$, $\alpha$, $n$, $m$ and the $f_i$ such that if
 \[ H = \Vert u \Vert_{L^\infty(\Omega')}  < \varepsilon_m, \]
 then
 \[ \Vert u \Vert_{C^{2m, 2\alpha; m, \alpha}(\Omega)} < C_m H . \]
\end{Proposition}

For a proof see e.g. \cite[Proposition 2.5]{BamCusps}.

We will frequently make use of the following consequence of Proposition \ref{Prop:Shi}.
\begin{Corollary} \label{Cor:Shi}
 Let $0 < \tau < 1$ and assume that $(h_t)_{t \in [0,\tau)}$ satisfies either the Ricci deTurck flow equation (\ref{eq:nonlinRdT}) or the linearized Ricci deTurck flow equation (\ref{eq:linRdT}) on a domain $D' \subset M$, where $(M, \ov{g})$ is an arbitrary complete Riemannian manifold.
Let moreover $D \subset \Int D'$ be a compact domain.
 
Then for any $m$, there exist constants $\varepsilon_m > 0, C_m < \infty$ depending only on $m$, $n$ and bounds on the curvature tensor of $M$ as well as its derivatives, such that if
\[ H = \Vert h \Vert_{L^\infty(D' \times [0,\tau))} < \varepsilon_m, \]
then
\[ \Vert \nabla^m h_t \Vert_{L^\infty(D)} < C_m t^{-m/2} H \qquad \text{for all $t \in [0,\tau)$}. \]
\end{Corollary}
Observe that $\varepsilon_m, C_m$ are in particular independent of the injectivity radius of $M$.
\begin{proof}
 At each point $p \in D$ pass over to a local cover and consider the domains $\Omega = B_r(p) \times [3r^2,4r^2] \subset B_{2r}(p) \times [0,4r^2] = \Omega'$ for $0 < r < \frac12 \tau^{1/2}$.
 Proposition \ref{Prop:Shi} then yields the desired result.
\end{proof}

\subsection{Short-time existence}
From (\ref{eq:nonlinRdT}), we see that the Ricci deTurck flow equation is strongly parabolic if $h_t$ is small enough.
We will quote a general short-time existence result which follows by a standard inverse function theorem argument.
For more details see \cite[Theorem 3.2]{Shi}, \cite[Chapter VII, Theorem 7.1]{LS}, \cite[sec 4]{SSS1} and \cite[sec 3.7]{Bam-phd}.

\begin{Proposition}[Short-time existence] \label{Prop:shortex}
 Let $(M,\ov g)$ be a complete Riemannian manifold.
 Assume that its curvature tensor is globally bounded in the $C^{0,\alpha}$-sense.
 Then there are $\varepsilon_{s.e.}, \tau_{s.e.} > 0$, $C_{s.e.} < \infty$ which only depend on $M$ and $\ov g$ such that the following holds: \\
Let $g_0$ be a smooth metric on $M$.
If 
\[ \Vert g_0 - \ov g \Vert_{L^{\infty}(M)} < \varepsilon_{s.e.}, \]
then there is a unique $L^\infty$-bounded smooth solution $(g_t) \in C^{\infty}(M \times [0,\tau_{s.e.}])$ to the Ricci deTurck flow equation (\ref{eq:RdTflow}) with initial metric $g_0$.
Moreover, we have the bound
\[ 
\Vert g_t - \ov g \Vert_{L^\infty(M \times [0,\tau_{s.e.}])} \leq C_{s.e.}  \Vert g_0 - \ov g \Vert_{L^\infty(M)} .
\]
\end{Proposition}

\subsection{Short-time estimates for the heat kernel}
For small times, we can estimate the heat kernel using a result by Cheng, Li and Yau (\cite{CLY}):
\begin{Proposition} \label{Prop:CLY}
Let $M^n$ be a complete Riemannian manifold of uniformly bounded curvature, $E$ a vector bundle over $M$ and $p_0 \in M$.
Then for every $T < \infty$ and $\delta > 0$ there are constants $C_m = C_m(M,E,p_0, T, \delta)$ such that the following holds:\\
Let $(k_t)_{0<t<T} \in C^\infty(M;E) \otimes E_{p_0}^*$ be the heat kernel in $p_0$, i.e.
\[ \partial_t k_t = \triangle k_t \qquad \text{and} \qquad k_t \xrightarrow{t \to 0} \delta_{p_0}  \id_{E_{p_0}}. \]
Then we have the estimates
\[ |\nabla^m k_t|(p) \leq C_m t^{-(n+m)/2} \exp \Big( - \frac{r^2}{(4+\delta) t} \Big), \]
where $r = d(p_0,p)$ and $0 < t < T$.
\end{Proposition}
\begin{proof}
Observe that by Kato's inequality we have $\partial_t |k_t| \leq \triangle |k_t|$ and hence the scalar heat kernel on $M$ bounds $|k_t|$.
The bounds on the derivatives follow with the help of Proposition \ref{Prop:Shi}.
\end{proof}

\section{The geometry of symmetric spaces} \label{sec:symsp}
\subsection{Introduction}
We give a short introduction to the geometry of symmetric spaces.
More detailed expositions can be found e.g. in \cite{Hel}, \cite{Ebe}, \cite{Bal}, \ldots

Let $(M, \ov g)$ be a Riemannian manifold and $p \in M$.
We call an isometry $\Phi : M \to M$ with $\Phi(p) = p$, a \emph{reflection} at $p$, if $d \Phi_p = - \id_{T_pM}$.
$M$ is called a \emph{(globally) symmetric space}, if it admits a reflection at every point.
If $M$ admits a local reflection at every point, then $M$ is called a $\emph{locally symmetric space}$.
Every locally symmetric space is the quotient of a simply connected symmetric space by a properly discontinuous group action and vice versa.
If $M$ is globally symmetric, then its isometry group acts transitively on the points of $M$.
Moreover, the curvature tensor is parallel, $\nabla \Rm \equiv 0$, on every locally symmetric space.

Assume now that $M$ is simply connected.
We call $M$ \emph{irreducible}, if it does not split as a product $M = M' \times M''$.
In this case, $M$ is automatically Einstein.
If the Einstein constant is zero, then $M$ is isometric to Euclidean space.
If it is positive, then $M$ has non-negative sectional curvature and is compact and $M$ is said to be of \emph{compact type} and if it is negative, then $M$ has non-positive sectional curvature and is diffeomorphic to $\IR^n$ and $M$ is said to be of \emph{noncompact type}.
By the de Rham Decomposition Theorem, every simply connected symmetric space splits uniquely as the product
\[ M = M_1 \times \ldots \times M_m \]
of irreducible factors $M_i$.
Generally, we say that a locally symmetric space is of compact (resp. noncompact) type, if all factors in the de Rham decomposition of its universal cover are of compact (resp. noncompact) type.

Assume still that $M$ is simply connected and choose a basepoint $p_0 \in M$.
Denote by $G$ the connected component of its isometry group and by $K < G$ the isotropy group at $p_0$, i.e. the subgroup of isometries fixing $p_0$.
Then $M = G/K$ where we identify $p_0$ with $1 \cdot K$.

For a list of all irreducible symmetric spaces, see \cite[p. 200]{Bes}.

\subsection{The infinitesimal structure} \label{subsec:infstruc}
Let $M = G/K$ be a simply connected symmetric space of noncompact type.
We will now discuss its infinitesimal structure.
Let $\mathfrak{g}$ be the Lie algebra of $G$.
The elements of $\mathfrak{g}$ correspond to Killing fields on $M$.
There is an involutory isomorphism $\sigma : \mathfrak{g} \to \mathfrak{g}$ which corresponds to the reflection at the basepoint $p_0$.
This isomorphism induces the splitting $\mathfrak{g} = \mathfrak{p} \oplus \mathfrak{k}$ into $-1$ and $1$ eigenspaces, where $\mathfrak{k}$ is the Lie algebra of $K$.
Moreover, we see that
\begin{equation} \label{eq:ppinketc} [\mathfrak{p}, \mathfrak{p}] \subset \mathfrak{k}, \qquad [\mathfrak{k}, \mathfrak{p}] \subset \mathfrak{p}, \qquad [\mathfrak{k}, \mathfrak{k}] \subset \mathfrak{k}. \end{equation}
For $v, w \in \mathfrak{g}$ we define the \emph{Killing form} by
\[ \langle v, w \rangle = \tr [v, [w, \cdot]].\]
Due to (\ref{eq:ppinketc}) the splitting $\mathfrak{g} = \mathfrak{p} \oplus \mathfrak{k}$ is orthogonal with respect to the Killing form, which is positive definite on $\mathfrak{p}$ and negative definite on $\mathfrak{k}$.

Let $\af \subset \mathfrak{p}$ be a \emph{maximal abelian subalgebra}, i.e. an abelian subalgebra that is not contained in a bigger abelian subalgebra in $\mathfrak{p}$.
The dimension $r = \dim \af$ is called the \emph{rank} of $M$.
Obviously, to every $v \in \mathfrak{p}$ there is a maximal abelian subalgebra containing $v$ and that is contained in $\mathfrak{p}$.
All such algebras are conjugate under the adjoint action of $K$ (cf. \cite[Chapter V, Lemma 6.3]{Hel}).
Hence, the rank of $M$ is well defined.

Now consider the infinitesimal adjoint action $[v, \cdot] : \mathfrak{g} \to \mathfrak{g}$ of any $v \in \af$ on $\mathfrak{g}$.
Since it is antisymmetric with respect to the Killing form and interchanges $\mathfrak{p}$ and $\mathfrak{k}$, we can diagonalize $[v, \cdot]$ with real eigenvalues.
Moreover, since $\af$ is abelian, we can find a simultaneous eigenspace decomposition
\[ \mathfrak{g} = \bigoplus_{\alpha \in \Delta} \mathfrak{g}_\alpha \]
where $\Delta \subset \af^*$ is called the \emph{root system} and
\[ [v,x_\alpha] = \alpha(v) x_\alpha \]
for any $v \in \af$ and $x_\alpha \in \mathfrak{g}_\alpha$.
The subspaces $\mathfrak{g}_\alpha$ are pairwise orthogonal with respect to the Killing form and for all $\alpha \in \Delta \setminus \{ 0 \}$ the subspace $\mathfrak{g}_\alpha$ is isotropic.

The existence of the involution $\sigma$ implies $-\Delta = \Delta$ and the involution $\sigma$ maps $\mathfrak{g}_\alpha$ to $\mathfrak{g}_{-\alpha}$.
So if we set $\mathfrak{p}_{\alpha} = (\mathfrak{g}_{\alpha} \oplus \mathfrak{g}_{-\alpha}) \cap \mathfrak{p}$ and  $\mathfrak{k}_{\alpha} = (\mathfrak{g}_{\alpha} \oplus \mathfrak{g}_{-\alpha}) \cap \mathfrak{k}$, we have $\mathfrak{g}_\alpha \oplus \mathfrak{g}_{-\alpha} = \mathfrak{p}_\alpha \oplus \mathfrak{k}_\alpha$.
Let $v_0 \in \af$ be an arbitrary vector such that $\alpha(v_0) \not= 0$ for all nonzero $\alpha \in \Delta$ and define the set of positive roots by $\Delta^+ = \{ \alpha \in \Delta \; : \; \alpha(v_0) > 0 \}$.
Then we have the following \emph{root space decomposition}
\begin{equation*}
\begin{split}
 \mathfrak{g} &= \af \oplus \bigoplus_{\alpha \in \Delta^+} (\mathfrak{g}_\alpha \oplus \mathfrak{g}_{-\alpha}) \oplus \mathfrak{k}_0 \\
&= \mathfrak{p} \oplus \mathfrak{k} =  \bigg( \af \oplus \bigoplus_{\alpha \in \Delta^+} \mathfrak{p}_\alpha \bigg) \oplus \bigg( \bigoplus_{\alpha \in \Delta^+} \mathfrak{k}_\alpha \oplus \mathfrak{k}_0 \bigg). 
\end{split}
\end{equation*}
These splittings are orthogonal with respect to the Killing form.
The subspace $\mathfrak{k}_0$ is a Lie algebra.
Its geometric meaning will be described below.

Using the Jacobi identity, we can conclude that for any two $\alpha, \beta \in \Delta$, we have $[\mathfrak{g}_\alpha, \mathfrak{g}_\beta] \subset \mathfrak{g}_{\alpha + \beta}$.
Hence $\mathfrak{n} = \mathfrak{n}^+ = \bigoplus_{\alpha \in \Delta^+} \mathfrak{g}_{\alpha}$ and $\mathfrak{n}^- = \bigoplus_{\alpha \in \Delta^+} \mathfrak{g}_{-\alpha}$ are nilpotent Lie algebras with $\sigma(\mathfrak{n}^+) = \mathfrak{n}^-$.
The spaces $\mathfrak{n}^+$ and $\mathfrak{n}^-$ are isotropic with respect to the Killing form, but on $\mathfrak{n} \oplus \mathfrak{n}^-$
\[ (\cdot, \cdot) = - \langle \cdot, \sigma \cdot \rangle \]
is a positive definite scalar product.
Let $\alpha_1, \ldots, \alpha_{n-r}$ be the roots of $\Delta^+$ occuring with the appropriate multiplicities and let $x_1, \ldots, x_{n-r}$ be an orthonormal basis of $\mathfrak{n}^+$ with respect to $( \cdot, \cdot)$ such that $x_i \in \mathfrak{g}_{\alpha_i}$.
Then $[x_i, x_j] \in \mathfrak{g}_{\alpha_i + \alpha_j}$.
Set $y_i = \sigma x_i \in \mathfrak{g}_{-\alpha_i} \subset \mathfrak{n}^-$.
So $\langle x_i, y_j \rangle = - \delta_{ij}$ and $[x_i, y_j] \in \mathfrak{g}_{\alpha_i - \alpha_j}$ and $[y_i, y_j] \in \mathfrak{g}_{-\alpha_i - \alpha_j}$.
We set
\[ p_i = \frac1{\sqrt{2}} ( x_i - y_i), \qquad k_i = \frac1{\sqrt{2}} (x_i + y_i). \]
Hence, $p_1, \ldots, p_{n-r}$ form an orthonormal basis of the orthogonal complement $\af^\perp$ of $\af$ in $\mathfrak{p}$ and $k_1, \ldots, k_{n-r}$ are a negative orthonormal basis of the orthogonal complement of $\mathfrak{k}_0$ in $\mathfrak{k}$.
We also choose an orthonormal basis $v_1, \ldots, v_r$ of $\af$ with respect to $\langle \cdot, \cdot \rangle$.

Observe that $\sigma [x_i, y_i] = [y_i, x_i] = - [x_i, y_i]$, hence $[x_i, y_i] \in \mathfrak{p}$.
Moreover, $[x_i, y_i] \in \mathfrak{g}_0$, so $[x_i, y_i] \in \af$.
Since for any $v \in \af$, we have
\[ \langle [x_i, y_i], v \rangle = - \langle [x_i, v], y_i \rangle = \alpha_i(v) \langle x_i, y_i \rangle = - \alpha_i(v), \]
we obtain 
\begin{equation} [x_i, y_i] = - \alpha_i^\#. \label{eq:xiyi} \end{equation}

Finally, we apply our knowledge on the infinitesimal structure to find out more about the global geometry of $M$.
The subgroup $A = \exp (\af) < G$ corresponding to $\af$ is abelian and isomorphic to $\IR^r$.
The orbit $\mathcal{F} = A . p_0$ is a geodesic submanifold of $M$ isometric to $\IR^r$ and is called a \emph{maximal flat} of $M$.
The subgroup $K_0 = \exp (\mathfrak{k}_0) < K$ corresponding to $\mathfrak{k}_0$ is the point stabilizer of the flat $\mathcal{F}$.
Observe that there are symmetric spaces with trivial $K_0$, such as $SL(n)/SO(n)$, however many symmetric spaces, e.g. hyperbolic space $\IH^n$ ($n \geq 3$), have nontrivial $K_0$.
The stabilizer (not the point stabilizer) $\Stab_K(\mathcal{F})$ of the flat $\mathcal{F}$ however consists of several components of $K_0$.
Forming the quotient $W = \Stab_K(\mathcal{F}) / K_0$ yields a discrete group, called the \emph{Weyl group}.
It follows that the orbit $K. p$ of every point $p \in M$ under the isotropy group $K$ intersects $\mathcal{F}$ in a nonempty set which is invariant under $W$.
Moreover, $\mathcal{F}$ can be decomposed into fundamental domains for the action of $W$, which are called \emph{Weyl chambers}, and $W$ is generated by reflections along the walls of an arbitrary Weyl chamber.
Finally, consider the subgroups $N$ resp. $N^-$ corresponding to $\mathfrak{n}$ resp. $\mathfrak{n}^-$.
The product subgroups $P = AN$ and $P^- = A N^-$ are called \emph{Borel subgroups}.
They act simply transitively on $M$ and stabilize a Weyl chamber at infinity in the geodesic compactification (see \cite[2.17.20]{Ebe}).

\subsection{Homogeneous vector bundles over symmetric spaces}
Let $M = G/K$ as before.
We can regard $M$ as the base of a right $K$-principal bundle $\pi : G \to M$.
Given any representation $\rho : K \to GL(E)$ (where $E$ is a real vector space), we can form the \emph{associated vector bundle}  $G \times_{\rho} E = (G \times E)/\sim$ where
\[ (g g', e) \sim (g, \rho(g') e). \]
We will denote this associated vector bundle, the vector space as well as the representation simply by $E$ and we will also say that $E$ is a \emph{homogeneous vector bundle}.
We remark that the pullback $\pi^* E$ is the trivial bundle $G \times E$.

A \emph{principal connection} on $G$ is a $\mathfrak{k}$-valued $1$-form $\theta \in \Omega^1(G; \mathfrak{k})$ satisfying the following two properties (compare e.g. \cite[chapter 2]{Roe}):
\begin{enumerate}[(i)]
\item Equivariance: For any $v \in TG$, $k \in K$ and right translate $v.k$, we have $\theta(v.k) = Ad(k^{-1})\theta(v)$.
Here $Ad : K \to GL(\mathfrak{k})$ is the adjoint representation with $Ad_* : \mathfrak{k} \to gl(\mathfrak{k}), u \mapsto (w \mapsto [u, w])$.
\item Being a projection: For $u \in \mathfrak{k}$ denote by $R_u$ the vector field $R_u : G \to TG$ generated by the infinitesimal right action $g \mapsto g.u$.
Then we impose $\theta(R_u) = u$.
\end{enumerate}
A principal connection induces a connection on every homogeneous vector bundle $E$:
Let $f \in C^\infty(M; E)$ be a section of $E$ and consider its pullback $\tilde{f} = \pi^* f$ as a function $G \to E$.
Then we set for any $v \in T_p M$
\[ \nabla^E_v f = (d \tilde{f} (v') + \rho_* \theta(v') \tilde{f})/\sim \]
where $v' \in T_{p'} G$ is any vector projecting to $v$, i.e. $\pi(p') = p$ and $d \pi (v') = v$.

There is a canonical principal connection $\theta$ on $G$ with which we will always work from now on:
Identify all tangent spaces of $G$ with $\mathfrak{g} = \pp \oplus \mathfrak{k}$ by the left $G$-action and define $\theta$ everywhere to be the projection $\mathfrak{g} \to \mathfrak{k}$.
This is the only connection that is invariant by the left $G$ action and the reflection at the basepoint.
Note that if we consider the adjoint representation $\Ad : K \to GL(\pp)$ and its associated vector bundle, the tangent bundle $E = TM$, then the induced connection is the Levi-Civita connection.

\subsection{Killing fields and Lie derivatives}
Consider a homogeneous vector bundle $E$ over a symmetric space $M = G/K$ corresponding to a representation $\rho : K \to GL(E)$.
Moreover, let $\theta$ be the principal connection on $\pi : G \to M$ from the last subsection.

For each $x \in \mathfrak{g}$ there is a Killing field $X = \frac{d}{dt}|_{t=0} \exp(tx) \in C^\infty(M; TM)$ and a right-invariant vector field $\tilde{X} \in C^\infty(G; TG)$ with $\tilde{X}(1) = x$.
Then $d\pi (\tilde{X}) = X$.
Consider now a section $f \in C^\infty(M;E)$ and the corresponding function $\tilde{f} = \pi^* f : G \to E$.
We define $\tilde{f}' : G \to E$ as the derivative on $G$ in the direction $\tilde{X}$
\[ \tilde{f}' = d \tilde{f} (\tilde{X}). \]
Since $\tilde{f}' (gg') = \rho((g')^{-1}) \tilde{f}' (g)$, we find that $\tilde{f}' = \pi^* f'$ for some section $f' \in C^\infty(M; E)$.
We call $f'$ the \emph{Lie derivative} of $f$ with respect to $X$ or $x$ and write
\[ f' = \mathcal{L}_X f = \mathcal{L}_x f. \]
Then for $x, y \in \mathfrak{g}$, we have (observe that since the vector fields $\tilde{X}, \tilde{Y}$ are right-invariant, $[\tilde{X}, \tilde{Y}]$ corresponds to $-[x,y]$)
\begin{equation} \label{eq:LLcommut}
\mathcal{L}_x \mathcal{L}_y f - \mathcal{L}_y \mathcal{L}_x f = - \mathcal{L}_{[x,y]} f.
\end{equation}
We now relate the Lie derivative $\mathcal{L}_X$ to the covariant derivative $\nabla_X$.
At any point $p \in M$, we can decompose $X = X_0 + X_1$ where $X_0$ and $X_1$ are Killing fields such that for the corresponding right-invariant vector fields $\tilde{X}_0, \tilde{X}_1 \in C^\infty(G; TG)$, we have $\theta(\tilde{X}_0) = 0$ and $\theta(\tilde{X})=\tilde{X}_1$ on $\pi^{-1}(p)$.
Then $X_1(p) = 0$ and at $p$
\[ \mathcal{L}_{X_0} f = \nabla_{X_0} f \qquad \text{and} \qquad \mathcal{L}_{X_1} f = \rho_* \theta(\tilde{X}) f. \]
This implies
\begin{equation} \label{eq:LandLC}
\mathcal{L}_X f = \nabla_X f + \rho_* \theta(\tilde{X}) f.
\end{equation}

Finally, we compute the Riemannian curvature of $M$ at $p_0$.
Let $x,y,z \in \pp$ and denote by $X,Y,Z$ the corresponding Killing fields.
Then $\theta(\tilde{X}) = \theta(\tilde{Y}) = \theta(\tilde{Z}) = 0$ and hence by (\ref{eq:LandLC}) applied to $E = TM$, we must have $\nabla X = \nabla Y = \nabla Z$ at $p_0$.
Hence
\[ -[[X,Y],Z] =  \nabla^2_{Z,X} Y - \nabla^2_{Z, Y} X   \\ =  R(Z,Y)X - R(Z, X)Y = R(X,Y)Z.
\]
So expressed on $\pp$
\begin{equation} \label{eq:curvature}
R(x,y)z = - [[x,y],z].
\end{equation}

\subsection{Cross-sections of symmetric spaces} \label{subsec:crosssec}
Consider the splitting $\mathfrak{g} = \pp \oplus \mathfrak{k}$, fix a maximal abelian subalgebra $\af \subset \pp$ and consider the set $\Delta^+ \subset \af^*$ of positive roots of $\mathfrak{g}$.
We call a root $\alpha \in \Delta^+$ \emph{simple} if there is no decomposition $\alpha = \alpha_1 + \alpha_2$ with $\alpha_1, \alpha_2 \in \Delta^+$.
We know (cf \cite[2.9.5]{Ebe}) that the set $\mathcal{B}^+ = \{ \beta_1, \ldots, \beta_r \}$ of simple roots forms a basis of the vector space $\af^*$ and that every $\alpha \in \Delta^+$ can be expressed as a linear combination $\sum_{i=1}^r k_i \beta_i$ of the simple roots with nonnegative integer coefficients $k_i$.

We define the \emph{positive Weyl chamber}
\[ \CC = \{ v \in \af \;\; : \;\; \alpha(v) \geq 0 \; \text{for all $\alpha \in \Delta^+$} \} = \{ v \in \af \;\; : \;\; \beta(v) \geq 0 \; \text{for all $\beta \in \mathcal{B}^+$} \}. \]
$\CC$ has the structure of a polytope and for every splitting $\mathcal{B}^+ = \ov{\mathcal{B}}^+  \dotcup \un{\mathcal{B}}^+$ we can consider the corresponding \emph{wall}
\[ \WW = \CC \cap \{ v \in \af \;\; : \;\; \beta(v) = 0 \; \text{for all $\beta \in \ov{\mathcal{B}}^+$} \}. \]
So for $\ov{\mathcal{B}}^+ = \emptyset$, we get $\WW = \CC$ and for $\ov{\mathcal{B}}^+ = \mathcal{B}^+$, we get $\WW = \{ 0 \}$.
From now on, given a wall $\WW \subset \CC$, we will denote the corresponding splitting sets by $\ov{\mathcal{B}}^+_\WW$ and $\un{\mathcal{B}}^+_\WW$.
Every wall $\WW \subset \CC$ has a boundary $\partial\WW$, which consists of walls $\WW' \in \partial \WW$ that are smaller than $\WW$ by one dimension.
Those walls $\WW'$ correspond to the splitting sets $\ov{\mathcal{B}}^+_{\WW'} = \ov{\mathcal{B}}^+_\WW \cup \{ \beta \}$ for $\beta \in \un{\mathcal{B}}^+_\WW$.

For every wall $\WW \subset \CC$, let $\ov{\Delta}^+_\WW \subset \Delta^+$ be the set of roots that can be represented by linear combinations of the simple roots $\ov{\mathcal{B}}^+_\WW$  and let $\un{\Delta}^+_\WW = \Delta^+ \setminus \ov{\Delta}^+_\WW$.
Setting $\ov{\af}_\WW = \spann (\ov{\mathcal{B}}^+_\WW)^\#$, we moreover obtain an orthogonal splitting $\af = \ov{\af}_\WW \oplus \un{\af}_\WW$.
Now set $\ov{\pp}_\WW = \ov{\af}_\WW \oplus \bigoplus_{\ov{\alpha} \in \ov{\Delta}^+_\WW} \pp_{\ov{\alpha}}$,  $\ov{\mathfrak{k}}_\WW = [\ov{\pp}_\WW , \ov{\pp}_\WW]$ and $\ov{\mathfrak{g}}_\WW = \ov{\pp}_\WW  \oplus \ov{\mathfrak{k}}_\WW$.
Note that $\ov{\mathfrak{g}}_\WW$ and $\ov{\mathfrak{k}}_\WW$ are Lie algebras themselves.
Denote by $\ov{G}_\WW$ and $\ov{K}_\WW$ the corresponding Lie groups.
Then $\ov{M}_\WW = \ov{G}_\WW / \ov{K}_\WW$ is a symmetric space, which we will call a \emph{cross-section} of $M$.
Hence $\ov{M}_{\{0\}} = M$ and $\ov{M}_\CC = \{ \text{pt} \}$.
We remark that not every symmetric space $M = G'/K'$ with $G' < G$ and $K' < K$ arises by this construction, e.g. $\IH^2$ is not a cross-section of $\IH^3$.

We need to discuss a few more properties of $\ov{M}_\WW$:
Its Weyl group $\ov{W}_\WW$, acting on $\ov{\af}_\WW$, is generated by reflections along the walls corresponding to the roots $\ov{\Delta}^+_\WW$.
Hence $\ov{W}_\WW$ is a subgroup of the Weyl group $W$ of $M$ and its fixed point set in $\af$ is exactly $\un{\af}_\WW$.
Next, consider the nilpotent Lie algebras  
\[ \ov{\nn}_\WW = \bigoplus_{\ov{\alpha} \in \ov{\Delta}^+_\WW} \mathfrak{g}_{\ov{\alpha}} \qquad \text{and} \qquad  \un{\nn}_\WW = \bigoplus_{\un{\alpha} \in \un{\Delta}^+_\WW} \mathfrak{g}_{\un{\alpha}}. \]
Let $\ov{N}_\WW, \un{N}_\WW < G$ be the corresponding Lie groups, let $\ov{A}_\WW, \un{A}_\WW < G$ be the Lie groups corresponding to $\ov{\af}_\WW$ resp. $\un{\af}_\WW$ and set $\ov{P}_\WW = \ov{A}_\WW \ov{N}_\WW$ and $\un{P}_\WW = \un{A}_\WW \un{N}_\WW$.
Observe that $N = \ov{N}_\WW \un{N}_\WW$ and $P = \ov{P}_\WW \un{P}_\WW$.
We will now show that $\ov{G}_\WW$ normalizes $\un{P}_\WW$:
To do this, it suffices to establish $[\ov{\pp}_\WW, \un{\af}_\WW \oplus \un{\nn}_\WW] \subset \un{\af}_\WW \oplus \un{\nn}_\WW$.
Obviously, $[\ov{\pp}_\WW, \un{\af}_\WW] = 0$.
For the second summand observe that for $\ov{\alpha} \in \ov{\Delta}^+_\WW$ and $\un{\beta} \in \un{\Delta}^+_\WW$, we have $[\pp_{\ov{\alpha}}, \mathfrak{g}_{\un{\beta}}] \subset \mathfrak{g}_{\ov{\alpha}+\un{\beta}} \oplus \mathfrak{g}_{-\ov{\alpha} + \un{\beta}}$.
Now if $-\ov{\alpha} + \un{\beta}$ were not positive, $\ov{\alpha} - \un{\beta}$ would be, but expressing this root as a linear combination of the roots in $\mathcal{B}^+$ would lead to a negative coefficient in front of one of the roots of $\un{\mathcal{B}}^+_\WW$.
This shows that $[\pp_{\ov{\alpha}}, \mathfrak{g}_{\un{\beta}}] \subset \un{\nn}_\WW$ and hence the claim.

Consider now a homogeneous vector bundle $E$ over $M$.
It corresponds to a representation $\rho: K \to GL(E)$ on a vector space which we also denote by $E$.
Restriction to $\ov{K}_\WW$ yields a respresentation $\rho_\WW : \ov{K}_\WW \to GL(E)$.
We will denote the associated homogeneous vector bundle over $\ov{M}_\WW$ by $E_\WW$.
Let now $f \in C^\infty(\ov{M}_\WW; E_\WW)$ be a section.
It corresponds to a smooth map $\tilde{f} : \ov{G}_\WW \to E$ such that $\tilde{f}(gk) = \rho_\WW(k^{-1}) \tilde{f}(g)$ for all $g \in \ov{G}_\WW, k \in \ov{K}_\WW$.
Using the fact that $P$ and $\ov{P}_\WW$ operate simply transitively on $M$ resp. $\ov{M}_\WW$, and the identity $P = \ov{P}_\WW \un{P}_\WW$, we conclude that there is a unique smooth extension $\tilde{\hat{f}} : G \to E$ of $\tilde{f}$ such that the following holds: $\tilde{\hat{f}}(gk) = \rho(k^{-1})\tilde{\hat{f}}(g)$ for all $g \in G, k \in K$ and $\tilde{\hat{f}}(hg) =\tilde{\hat{f}}(g)$ for all $h \in \un{P}_\WW, g \in G$.
Hence $\tilde{\hat{f}}$ corresponds to a smooth $\un{P}_\WW$-invariant section $\widehat{f} \in C^\infty(M;E)$ which we call the \emph{lift of $f$}.
Since the isometry group $\ov{G}_\WW$ of $\ov{M}_\WW$ normalizes $\un{P}_\WW$, the construction of the lift is equivariant under $\ov{G}_\WW$.

\section{The heat kernel in homogeneous vector bundles} \label{sec:heatkernel}
\subsection{Statement of the results} \label{subsec:heatkerIntro}
In this section, we will prove a general decay result about the heat kernel in homogeneous vector bundles over symmetric spaces.
Let $M = G/K$ be a simply-connected symmetric space of noncompact type and consider a homogeneous vector bundle $E$ over $M$.
Choose a basepoint $p_0 \in M$, set $E_0 = E_{p_0}$ and consider the heat kernel $(k_t)_{t > 0} \in C^\infty (M; E) \otimes E^*_0$ for the connection Laplacian $\triangle = - \nabla^{E^*} \nabla^E = \sum_{i=1}^n (\nabla^E )^2_{v_i, v_i}$ (for a local orthonormal frame $(v_i)_{i = 1, \ldots, n}$) in $E$ centered in $p_0$, i.e. for all $e \in E_0$
\[ \partial_t k_t e = \triangle k_t e \qquad \text{and} \qquad k_t e \xrightarrow{t \to 0} \delta_{p_0} e. \]
In the following, we will explain how to compute a constant $\lambda_0 = \lambda_{M,E}$, depending on the space $M$ and the bundle $E$, which controls the exponential $L^1$-decay rate of $k_t$, i.e. for which $\Vert k_t \Vert_{L^1(M)} < C e^{-\lambda_0 t}$ for all $t > 0$ and some $C < \infty$ or for which, in certain cases, we have the slightly weaker bound $\Vert k_t \Vert_{L^1(M)} < C (\log (t+2) )^{1/2} (t+2)^{a/2} e^{-\lambda_0 t}$ for all $t > 0$, some $C < \infty$ and some $a < \infty$, depending only on $M$.
In many cases, the bound on the decay rate is already the exact decay rate, by which we mean that we even have $c e^{-\lambda_0 t} < \Vert k_t \Vert_{L^1(M)} < C e^{-\lambda_0 t}$ for all $t >0$ and some $c > 0, C<\infty$.

The constant $\lambda_0 = \lambda_{M,E}$ is defined to be the minimum over certain constants, each corresponding to a cross-section of the symmetric space $M$.
In order to give an idea about the concept behind these constants, we will first discuss the case in which $M$ has rank $1$.
In this case $\lambda_0 = \min \{ \lambda_L, \lambda_B \}$, where $\lambda_L$ and $\lambda_B$ are defined as follows:
\begin{description}
\item[The constant $\mathbf{\lambda_L}$] Consider a Borel subgroup $P = A N < G$ (see subsection \ref{subsec:infstruc}), i.e. $P$ acts simply transitively on $M$ and fixes a point at infinity.
Let $V_{par} \subset C^\infty(M; E)$ be the vector space of $P$-invariant sections (we will later call those sections \emph{parabolically invariant}).
Evaluation at $p_0$ induces an isomorphism $V_{par} \cong E_0$.
Observe that for every $f \in V_{par}$, its (connection) Laplacian $\triangle f = - \nabla^*\nabla f$ is also contained in $V_{par}$ and hence we can define the operator $S_{par} = -\triangle : V_{par} \to V_{par}$.
As we will see in the next subsection, $S_{par}$ is self-adjoint and using the isomorphism $V_{par} \cong E_0$, we will compute that $S_{par} (e) = - \sum_{i=1}^{n-1} k_i. k_i. e$.
Now define $\lambda_L$ to be the smallest eigenvalue of $S_{par}$.
We will see that always $\lambda_L \geq 0$.
\item[The constant $\mathbf{\lambda_B}$] Here we consider all Bochner formulas for sections in $E$, i.e. expressions
\begin{equation} \label{eq:BochnerlambdaB}
 -\triangle = D^*D + \lambda
\end{equation}
for some linear first order operator $D : C^\infty(M;E) \to C^\infty(M;E')$ and its formal adjoint $D^* : C^\infty(M; E') \to C^\infty(M; E)$.
Let $\lambda_B$ be the maximum of all such $\lambda$.
Obviously, $\lambda_B \geq 0$, since we always have the trivial Bochner formula $-\triangle = \nabla^* \nabla$.
The constant $\lambda_B$ bounds the $L^2$-decay of $k_t$, i.e. $\Vert k_t \Vert_{L^2(M)} \leq C e^{-\lambda_B t}$ for all $t>1$ and some $C < \infty$.

Note that in the rank $1$ case, we could also define $\lambda_B$ to be the supremum over all $\lambda$ for which we have $\Vert k_t \Vert_{L^2(M)} \leq C e^{-\lambda t}$ for all $t>1$ and some $C < \infty$.
This might improve the constant $\lambda_0$ and lead to a stronger result.
However, it would make the computation of $\lambda_B$ unnecessarily complicated for our purposes and it is also not clear to the author how to carry this concept over to the higher rank case.
\end{description}
The main theorem of this section in the rank $1$ case now reads
\begin{Theorem}[Rank $1$ case] \label{Thm:dectrianglerank1}
 Let $M$ be of rank $1$ and let $\lambda_L$ and $\lambda_B$ be defined as above.
 Then there are constants $c>0$, $C< \infty$ such that:
 
 If $\lambda_B > \lambda_L$, then
 \[ c e^{-\lambda_L t} < \Vert k_t \Vert_{L^1(M)} < C e^{-\lambda_L t} \qquad \text{for all} \qquad t > 0. \]
 If $\lambda_B < \lambda_L$, then we have at least
 \[ c e^{-\lambda_L t} < \Vert k_t \Vert_{L^1(M)} < C e^{-\lambda_B t} \qquad \text{for all} \qquad t > 0. \]
 Finally, if $\lambda_B = \lambda_L$, the upper bound still holds with $\lambda_B$ replaced by any $\lambda < \lambda_B$ (where $C$ depends on $\lambda$).
 More precisely, we have
 \[ c e^{-\lambda_L t} < \Vert k_t \Vert_{L^1(M)} < C (\log (t+2))^{1/2} (t+2)^{a/2} e^{-\lambda_L t} \]
where $a = \max \{ (\sum_{i=1}^{n-r} |\alpha_i| ) / \min_i |\alpha_i|, 2 \}$.
\end{Theorem}

The next result is in the same spirit as Theorem \ref{Thm:dectrianglerank1}, but it gives a pointwise bound on the heat kernel.
\begin{Theorem} \label{Thm:Asymptriangle}
Let $M$ be of rank $1$ and let $\lambda_L$ and $\lambda_B$ be defined as above.
There is a constant $C < \infty$ such that for $\lambda_0 = \min \{ \lambda_L, \lambda_B \}$ 
\[ |k_t(p)| \leq \frac{C}{\vol B_r(p_0)} e^{- \lambda_0 t}  \qquad \text{where} \qquad r = d(p_0,p).\]
\end{Theorem}

We give two examples that illustrate the possible constellations of $\lambda_L$ and $\lambda_B$: \\[2mm]
\textbf{Example A.\;}
Consider $M = \IH^n$ and $E = T^*$, the bundle of $1$-forms.
In this case, we have $\lambda_L = 1$ and $\lambda_B = n-1$.
So for $n > 2$, we have an exact exponential decay with rate $- \lambda_0 = -1$.
For $n=2$, we can show that $e^t \Vert k_t \Vert_{L^1(M)}$ does not stay bounded for large $t$: 

Assume the opposite and consider the disc model of $\IH^2$ embedded in $\IR^2$, i.e. $g_{\IH^2} = (1-r^2)^{-2} (dx^2 + dy^2)$, and choose $p_0 = 0$.
Then for $f = dx \in C^\infty(M;E)$, we have $\triangle f + f = 0$ and hence $\int_M \langle f, e^t k_t \rangle$ is constant in time and nonzero.
But since $|f| = 1-r^2$ and since the $L^1$-norm of $e^t k_t$ is assumed to stay bounded, we conclude that the supremum of $e^t | k_t |$ over a sufficiently large ball around $p_0$ has to stay bounded from below as $t \to \infty$.

On the other hand, $e^t k_t$ is bounded in $L^\infty$ along with all its derivatives for the following reason:
By Cauchy-Schwarz and the convolution property of $k_t$, we conclude that for $t  > 2$ we have $|\nabla^m k_t| = |\nabla^m k_{1} * k_{t-1}| \leq \Vert k_{t-1} \Vert_{L^2(M)} \leq C e^{-t}$.
Observe moreover that $\frac12 \partial_t \Vert (e^t k)t \Vert_{L^2(M)}^2 = - \Vert d ( e^t k_t ) \Vert_{L^2(M)}^2 - \Vert d^* (e^t k_t ) \Vert_{L^2(M)}^2$.

By Arzela-Ascoli we find a subsequence $e^{t_i} k_{t_i}$ which converges to some nonzero $k_\infty \in C^\infty(M;E)$.
This $k_\infty$ must be bounded in $L^2$ and $L^1$ and by the right choice of the $t_i$, we can guarantee that $d k_\infty = 0$ and $d^* k_\infty = 0$.
Since $k_\infty$ is also spherical (see the next subsection), it is possible to conclude that $k_\infty$ must be a nonzero multiple of $f$, but $f$ is unbounded in $L^1$. \\[2mm]
\textbf{Example B.\;}
Consider the case $M = \IH^2$ and $E = \Sym_2^0 T^*$, the space of quadratic differentials.
Then $\lambda_L = 4$ and $\lambda_B = 2$, so the exponential decay rate lies between $-4$ and $-2$.
Since $f = dx dy \in C^2(M; E)$ is a bounded section satisfying $\triangle f + 2 f = 0$, we conclude that $\int_M \langle f, e^{2t} k_t \rangle$ is constant in time and nonzero.
Hence, $\Vert e^{2t} k_t \Vert_{L^1(M)}$ has to stay bounded from below and the exponential decay rate is exactly $-2$. \\

We will now discuss the case in which $M$ has general rank.
As explained in subsection \ref{subsec:crosssec}, for every wall $\WW \subset \CC$ of the positive Weyl chamber $\CC$, there is a cross-sectional symmetric space $\ov{M}_\WW$ of $M$.
For example, $\ov{M}_{\{ 0 \}} = M$ and $\ov{M}_\CC$ is just a point.
The vector bundle $E$ restricts to a homogeneous vector bundle $E_\WW$ over $\ov{M}_\WW$ and to every section $f \in C^\infty( \ov{M}_\WW ; E_\WW)$ we find a lift $\widehat{f} \in C^\infty(M; E)$, which is invariant under the parabolic subgroup $\un{P}_\WW$.
Obviously, then also $\triangle \widehat{f}$ is invariant under $\un{P}_\WW$ and hence $\triangle \widehat{f} = \widehat{f}'$ for some $f' \in C^\infty(\ov{M}_\WW;E_\WW)$.
There is a linear (zero-order) bundle endomorphism $S_\WW : C^\infty(\ov{M}_\WW ; E_\WW) \to C^\infty(\ov{M}_\WW ; E_\WW)$ which satisfies
\begin{equation} \label{eq:triangleWW}
 f' = \triangle_{\WW} f = \ov{\triangle} f - S_\WW f,
\end{equation}
where $\ov{\triangle}$ denotes the Laplacian on $\ov{M}_\WW$.
In the next subsection, we will see that $S_\WW$ is self-adjoint at that at $p_0$, we have $S_\WW(e) = - \sum_{\un{\alpha} \in \un{\Delta}^+_\WW} k_{\un{\alpha}} . k_{\un{\alpha}} . e$.
We remark that in the case $\WW = \{ 0 \}$, we have $\triangle_\WW = \ov{\triangle} = \triangle$ and $S_\WW = 0$.
In the case $\WW = \CC$, we have $\triangle_\WW = - S_\WW = - S_{par}$.
Now consider all possible Bochner formulas
\begin{equation} \label{eq:BochnetriangleWW}
 - \triangle_\WW = D^* D + \lambda
\end{equation}
on $\ov{M}_\WW$ and let $\lambda_\WW$ be the maximum of all such $\lambda$.
We have $\lambda_\WW \geq 0$, since there is always the trivial Bochner formula $-\triangle_\WW = \nabla^* \nabla + S_\WW$.
Observe that in the two extreme cases we get $\lambda_{\{ 0 \}} = \lambda_B$ and $\lambda_\CC = \lambda_L$.
We finally set $\lambda_0 = \lambda_{M, E} = \min_{\WW \subset \CC} \lambda_\WW$ and $\lambda_1 = \min_{\WW \subsetneqq \CC} \lambda_\WW$

We can now state the main theorem of this section in its full generality:
\begin{Theorem}[General rank case] \label{Thm:dectrianglegenrank}
Let $M$ be a simply-connected symmetric space of noncompact type and let the constants $(\lambda_\WW)_{\WW \subset \CC}$, $\lambda_0$ and $\lambda_1$ be defined as above.
Then there are constants $c>0$, $C< \infty$ such that:

If $\lambda_1 > \lambda_\CC$, then the exponential decay rate is exactly $-\lambda_0 = - \lambda_\CC$, i.e.
\[ c e^{-\lambda_0 t} < \Vert k_t \Vert_{L^1(M)} < C e^{-\lambda_0 t} \qquad \text{for all} \qquad t > 0. \]
If $\lambda_1 \leq \lambda_\CC$, then the upper bound still holds with $\lambda_0$ replaced by any $\lambda < \lambda_0$ (where $C$ depends on $\lambda$).
More precisely, there is an $A < \infty$ such that.
\[ c e^{-\lambda_\CC t} < \Vert k_t \Vert_{L^1(M)} < C (t+2)^A e^{-\lambda_0 t} \qquad \text{for all} \qquad t > 0. \]
\end{Theorem}

A few remarks on the proofs of Theorems \ref{Thm:dectrianglerank1}, \ref{Thm:Asymptriangle} and \ref{Thm:dectrianglegenrank}:
Obviously, Theorem \ref{Thm:dectrianglegenrank} implies Theorem \ref{Thm:dectrianglerank1} in the main case $\lambda_L > \lambda_B$ (which is the one needed here).
Despite of this fact, we first carry out the proof of Theorem \ref{Thm:dectrianglerank1} in subsection \ref{subsec:L1decrank1}, since it is much simpler than the proof of Theorem \ref{Thm:dectrianglegenrank}, which is presented in subsection \ref{subsec:L1decgenrank}.
Subsection \ref{subsec:radialsections} contains a preparatory discussion on spherical models, which will be used to describe the heat kernel $k_t$.
In subsection \ref{subsec:firstbounds}, we discuss some basic bounds that will be needed in both the rank $1$ as well as the general rank case.

\subsection{Spherical sections in symmetric spaces} \label{subsec:radialsections}
Consider a homogeneous vector bundle $E$ over $M$ coming from a representation $\rho : K \to GL(E)$.
We will analyze two classes of sections of $E$, namely \emph{spherical} and \emph{parabolically invariant} ones.
Later, we will generalize the discussion of the parabolically invariant sections to $\un{P}_\WW$-invariant sections, what will then allow us to compute the endomorphisms $S_\WW$.
The last part will only be needed for the proof of Theorem \ref{Thm:dectrianglegenrank} and can be skipped for the rank $1$ case.

We first introduce spherical sections.
Let $p_0 \in M$ be a basepoint and $K$ its stabilizer.
Then $K$ naturally acts on the space of sections $C^\infty(M; E)$ of $E$ and on the fiber $E_0 = E_{p_0}$ over $p_0$.
Hence it also acts on $C^\infty(M; E) \otimes E_0^*$.
We will now consider elements of this space.

\begin{Definition}
A section $f \in C^\infty(M; E) \otimes E_0^*$ is called \emph{spherical} if it is invariant under the action of $K$.
\end{Definition}
(Compare also with the notion of spherical functions as introduced in \cite[Chapter VII.8]{Kna}.)

Obviously, the Laplacian $\triangle f$ of a spherical section is also spherical and $(\triangle f)(e) = \triangle (f(e))$ for all $e \in E_0$.

Let now $f$ be a spherical section, consider a maximal abelian subalgebra $\af \subset \pp$ and the flat $\mathcal{F} = \exp(\af).p_0$ (cf subsection \ref{subsec:infstruc}).
Recall that the orbit of any point $p \in M$ under the action of $K$ intersects $\mathcal{F}$ in a nonempty set, which is invariant under the Weyl group $W$.
So $f$ is already determined by its restriction to $\mathcal{F}$.
Since the curvature along $\mathcal{F}$ vanishes, the vector bundle $E$ restricted to $\mathcal{F}$ becomes trivial and we can view the restriction of $f$ to $\mathcal{F}$ as a function
\[ F : \af \longrightarrow \End (E_0), \qquad v \longmapsto f ( \exp(v). p_0). \]
We conclude that $F$ actually only takes values in the subspace of $K_0$-equivariant endomorphisms $\End_{K_0}(E_0)$ of $E_0$ and is equivariant under the Weyl group, i.e. for every $w \in W$ and $v \in \af$, we have $F(w.v) = w. F(v) = w \circ F(v) \circ w^{-1}$.
\begin{Definition}
The smooth function $F : \af \to \End_{K_0}(E_0)$ is called a \emph{spherical model} of $f$.
\end{Definition}

We will now compute the Laplacian $f' = \triangle f$ of a spherical section $f$ in terms of its spherical model $F$, i.e. for each $v \in \af$ we will compute $F'(v)$,  where $F'$ is the spherical model of $f'$.
Recall the vectors $v_i, k_i, p_i, x_i, y_i \in \mathfrak{g}$ as defined in subsection \ref{subsec:infstruc}.
The conjugates $k'_i = \Ad(\exp(v)) k_i$, $p'_i = \Ad(\exp(v)) p_i$ correspond to rotations and translations at $p = \exp(v). p_0$.
Hence
\begin{equation} \label{eq:fprimevippi}
 f'(p) = \sum_{i=1}^r \mathcal{L}_{v_i} \mathcal{L}_{v_i} f (p) + \sum_{i=1}^{n-r} \mathcal{L}_{p'_i} \mathcal{L}_{p'_i} f (p).
\end{equation}
Note that, for example, $\mathcal{L}_{p'_i} f(p)$ denotes the Lie derivative of the section $f \in C^\infty (M; E) \otimes E^*_0$ at $p$ corresponding to $p'_i \in \mathfrak{g}$, where $p'_i$ acts trivially on the second factor $E^*_0$ of the tensor product.
Similarly, if we write $\mathcal{L}_{k_i} f$ later on, then we assume that $k_i$ acts trivially on $E^*_0$, even though there is a non-trivial action.
So for any $e \in E_p$ we have $(\mathcal{L}_{k_i} f)(e) = \mathcal{L}_{k_i} (f(e))$ at $p$.

We will now rewrite the previous equation (\ref{eq:fprimevippi}).
For this, observe that since
\[ \Ad(\exp(v))x_i = \exp( \alpha_i(v)) x_i \qquad \text{and} \qquad \Ad(\exp(v))y_i = \exp( - \alpha_i(v)) y_i, \]
we have
\[ p'_i = -\frac1{\sinh \alpha_i(v)} k_i + \frac1{\tanh \alpha_i(v)} k'_i. \]
So since by (\ref{eq:xiyi}) $[k_i, k'_i] = \sinh(\alpha_i(v)) \alpha_i^\#$, we conclude by (\ref{eq:LLcommut})
\begin{multline*}
 \mathcal{L}_{p'_i} \mathcal{L}_{p'_i} f = \frac1{\sinh^2 \alpha_i(v)} \mathcal{L}_{k_i} \mathcal{L}_{k_i} f + \frac{\cosh^2 \alpha_i(v)}{\sinh^2 \alpha_i(v)} \mathcal{L}_{k'_i}\mathcal{L}_{k'_i} f 
- \frac{\cosh \alpha_i(v)}{\sinh^2 \alpha_i(v)} \left( \mathcal{L}_{k'_i} \mathcal{L}_{k_i} f + \mathcal{L}_{k_i} \mathcal{L}_{k'_i} f \right) \\
= \frac1{\sinh^2 \alpha_i(v)} \mathcal{L}_{k_i} \mathcal{L}_{k_i} f + \frac{\cosh^2 \alpha_i(v)}{\sinh^2 \alpha_i(v)} \mathcal{L}_{k'_i}\mathcal{L}_{k'_i} f
 - 2\frac{\cosh \alpha_i(v)}{\sinh^2 \alpha_i(v)}  \mathcal{L}_{k'_i} \mathcal{L}_{k_i} f + \ctanh (\alpha_i(v)) \mathcal{L}_{\alpha_i^\#} f .
\end{multline*}
We now use the fact that $f$ is invariant under the action of $K$.
So at $p$ we have $\mathcal{L}_{k_i} f(e) - f(k_i. e) = 0$.
Moreover, since $k'_i = \Ad (\exp(v)) k_i$, the parallel transport of $\mathcal{L}_{k'_i} f(e)$ from $p$ to $p_0$ along $\mathcal{F}$ is equal to $k_i$ applied to the parallel transport of $f(e)$ from $p$ to $p_0$ along $\mathcal{F}$.
In other words, using the trivialization of $E$ over $\mathcal{F}$, we have $\mathcal{L}_{k'_i} f (e) = k_i.f(e)$.
Hence, we obtain
\begin{multline*}
 F'(v)(e) = \triangle F(v)(e) + \sum_{i=1}^{n-r} \left[ \frac1{\sinh^2 \alpha_i(v)} F(v) (k_i.k_i.e) + \frac{\cosh^2 \alpha_i(v)}{\sinh^2 \alpha_i(v)} k_i. k_i . F(v)(e) \right. \\ \left. - 2\frac{\cosh \alpha_i(v)}{\sinh^2 \alpha_i(v)} k_i. (F(v)(k_i.e)) +  \ctanh (\alpha_i(v)) (\partial_{\alpha_i^\#} F)(v)(e)\right].
\end{multline*}

So for $v \to \infty$ in the sense that $\alpha_i(v) \to \infty$ for all $i$, the expression becomes in the limit
\begin{equation}
 F'(v)(e) = \triangle F(v)(e)  + \sum_{i=1}^{n-r} \left[   k_i. k_i . F(v)(e)  +  (\partial_{\alpha_i^\#} F)(v)(e)\right]. \label{eq:fprimeformulalimit}
\end{equation}

Next, let $P = AN < G$ be a Borel subgroup and call a $P$-invariant section $f \in C^\infty(M; E)$ \emph{parabolically invariant}.
Since $P$ acts transitively on $M$, the section $f$ is determined by its value $f(p_0) \in E_0$ at $p_0$.
Observe that with $f$ its Laplacian $\triangle f$ is also parabolically invariant.
We have
\begin{equation} \label{eq:trianglesLL}
 (\triangle f)(p_0) = \sum_{i=1}^r \mathcal{L}_{v_i} \mathcal{L}_{v_i} f (p_0) + \sum_{i=1}^{n-r} \mathcal{L}_{p_i} \mathcal{L}_{p_i} f (p_0).
\end{equation}
Using the fact that $\mathcal{L}_{v_i} f = \mathcal{L}_{x_i} f = 0$ and $p_i = - k_i + \sqrt{2} x_i$, we obtain
\[ (\triangle f)(p_0) = \sum_{i=1}^{n-r} \left[ \mathcal{L}_{k_i} \mathcal{L}_{k_i} f (p_0) - \sqrt{2} \mathcal{L}_{x_i} \mathcal{L}_{k_i} f (p_0) \right]. \]
Since $\mathcal{L}_{x_i} \mathcal{L}_{k_i} f = \mathcal{L}_{k_i} \mathcal{L}_{x_i} f + \frac1{\sqrt{2}} \mathcal{L}_{\alpha^\#_i} f = 0$, we obtain
\begin{equation} \label{eq:parkiki}
 (\triangle f)(p_0) = \sum_{i=1}^{n-r} k_i . k_i . f(p_0).
\end{equation}
Observe that the right hand side is exactly the zero order term in (\ref{eq:fprimeformulalimit}).
Note also that the map $E_0 \to E_0$, $e \mapsto \sum_{i=1}^{n-r} k_i . k_i . e$ is self-adjoint.

Now consider a cross-section $\ov{M}_\WW$ corresponding to a wall $\WW \subset \CC$ and let $\un{P}_\WW = \un{A}_\WW \un{N}_\WW$ be the corresponding parabolic subgroup.
Let $f \in C^\infty(\ov{M}_\WW; E_\WW)$ and consider its ($\un{P}_\WW$-invariant) lift $\widehat{f} \in C^\infty(M; E)$.
Then $\triangle \widehat{f}$ is also $\un{P}_\WW$-invariant and hence there is an $f' \in C^\infty(\ov{M}_\WW; E_\WW)$ such that $\widehat{f}' = \triangle \widehat{f}$.
We will calculate $f'$.
First, recall that the isometry group $\ov{G}_\WW$ of $\ov{M}_\WW$ normalizes $\un{P}_\WW$, so if $f$ is $\un{P}_\WW$-invariant, then so are its translates by the action of $\ov{G}_\WW$.
Hence, it suffices to compute $f'(p_0)$.
Using (\ref{eq:trianglesLL}) and the fact that $\mathcal{L}_{v} \widehat{f} = 0$ for $v \in \un{\af}_\WW$ and $\mathcal{L}_{x_{\un{\alpha}}} \widehat{f} = 0$, we obtain as above
\[ \widehat{f}' (p_0) = (\triangle \widehat{f})(p_0) = \sum_{\ov{i}} \mathcal{L}_{v_{\ov{i}}} \mathcal{L}_{v_{\ov{i}}} \widehat{f} (p_0) + \sum_{\ov{\alpha} \in \ov{\Delta}^+_\WW} \mathcal{L}_{p_{\ov{\alpha}}} \mathcal{L}_{p_{\ov{\alpha}}} \widehat{f} (p_0) + \sum_{\un{\alpha} \in \un{\Delta}^+_\WW} k_{\un{\alpha}} . k_{\un{\alpha}} . \widehat{f} (p_0). \]
So
\[ f'(p_0) = (\ov{\triangle} f) (p_0) + \sum_{\un{\alpha} \in \un{\Delta}^+_\WW} k_{\un{\alpha}} . k_{\un{\alpha}} . f(p_0) \]
and comparing this with (\ref{eq:triangleWW}) yields
\begin{equation} \label{eq:SWW}
 (S_\WW f)(p_0) = -  \sum_{\un{\alpha} \in \un{\Delta}^+_\WW} k_{\un{\alpha}} . k_{\un{\alpha}} . f(p_0).
\end{equation}

\subsection{First bounds on the heat kernel} \label{subsec:firstbounds}
First, observe that the homogeneous vector bundle $E$ arises from a representation $\rho : K \to GL (E)$ of a vector space with the same name $E$.
Since $K$ is compact, we may fix a $\rho$-invariant metric on the vector space $E$.
This metric induces a metric $\langle \cdot, \cdot \rangle$ on the homogenous vector bundle $E$, which is parallel with respect to the connection $\nabla^E$.

Consider now the heat kernel $(k_t)_{t>0} \in C^\infty(M; E) \otimes E_0^*$ based at $p_0$.
Then we have:
\begin{Lemma} \label{Lem:heatkermodel}
 $(k_t)_{t > 0}$ is a spherical section and its spherical model $(K_t)_{t> 0} : \af \to \End_{K_0} E_0$ satisfies
\begin{multline*}
 \partial_t K_t(v)(e) = \triangle K_t(v)(e)  + \sum_{i=1}^{n-r} \left[ \frac1{\sinh^2 \alpha_i(v)} K_t(v) (k_i.k_i.e) + \frac{\cosh^2 \alpha_i(v)}{\sinh^2 \alpha_i(v)} k_i. k_i . K_t(v)(e) \right. \\ \left. - 2\frac{\cosh \alpha_i(v)}{\sinh^2 \alpha_i(v)} k_i. (K_t(v)(k_i.e)) +  \ctanh (\alpha_i(v)) (\partial_{\alpha_i^\#} K_t)(v)(e)\right].
\end{multline*}
Moreover, $K_t(v) \in \End_{K_0} E_0$ is always self-adjoint.
\end{Lemma}

\begin{proof}
By the uniqueness of the heat kernel, $k_t$ must be a spherical section for all $t > 0$.
So the evolution equation for $K_t$ follows from the calculations in the previous subsection.
It remains to show that $K_t(v)$ is a self-adjoint endomorphism.

For this consider first the reflection $\Phi : M \to M$ in $p_0$.
If we make the identifications $M = G/K$ and $p_0 = 1 \cdot K$, then this reflection comes from an automorphism of $G$ that stabilizes $K$.
This automorphism corresponds to an automorphism of the $K$-principal bundle $G \to M= G/K$ and hence induces a reflection $\Phi^E : E \to E$ of the total space of the bundle $\pi : E \to M$ such that $\pi \circ \Phi^E = \Phi \circ \pi$.
Note that $\Phi^E |_{E_0} = - \id_{E_0}$.
Since $\Phi^E$ leaves the connection $\nabla^E$ invariant, we conclude that $k_t$ is equivariant under $\Phi^E$.
In other words $\Phi^E (k_t(p) e) = k_t (\Phi(p)) (-e)$.
In terms of the spherical model $K_t$, this implies
\begin{equation} \label{eq:KvisKminusv}
 K_t (v) = K_t (-v). 
\end{equation}

Next consider the heat kernel $(\tilde{k}_t)_{t > 0} \in C^\infty (M \times M; E \boxtimes E^*)$ with variable basepoint.
This means that
\[ \partial \tilde{k}_t (x,y) = \triangle_x \tilde{k}_t (x,y) \qquad \text{and} \qquad \tilde{k}_t (\cdot, y) \xrightarrow{t \to 0} \delta_{y} \id_{E_y} \]
In particular
\[ k_t (x) = \tilde{k}_t (x, p_0). \]
Fix some points $p_1, p_2 \in M$ and let $e_i \in E_{p_i}$.
For any $t > 0$ we may compute
\begin{multline*}
\frac{d}{ds} \int \big\langle \tilde{k}_s (z, p_1) e_1 , \tilde{k}_{t-s} (z, p_2) e_2 \big\rangle  dz \\
= \int \Big( \big\langle \triangle_z \tilde{k}_s (z, p_1 ) e_1 , \tilde{k}_{t-s} (z, p_2 ) e_2 \big\rangle 
- \big\langle \tilde{k}_s (z, p_1) e_1 , \triangle_z \tilde{k}_{t-s} (z, p_2) e_2 \big\rangle \Big) dz = 0
\end{multline*}
So letting $s$ go to $t$ and $0$ yields
\[ \big\langle \tilde{k}_t (p_2, p_1) e_1, e_2 \big\rangle = \big\langle e_1, \tilde{k}_t (p_1, p_2) e_2 \big\rangle = \big\langle \tilde{k}^*_t (p_1, p_2) e_1, e_2 \big\rangle. \]
Here $\tilde{k}_t^* (p_1, p_2) \in \End (E_{p_1}, E_{p_2})$ denotes the adjoint of $\tilde{k}_t (p_1, p_2) \in \End (E_{p_2}, E_{p_1})$.
In terms of the spherical model, this relation implies that
\[ K_t^* (v) = K_t (-v) = K_t (v). \]
The last equality follows from (\ref{eq:KvisKminusv}).
\end{proof}

Observe that by (\ref{eq:parkiki}) $\lambda_L$ (in the rank $1$ case) or $\lambda_\CC$ (in the general rank case) is the smallest eigenvalue of the endomorphism
\[ E \longrightarrow E, \qquad e \longmapsto - \sum_{i=1}^{n-r} k_i . k_i . e. \]
Denote moreover by $\mu_i$ the largest eigenvalue of the endomorphism
\[ E \longrightarrow E, \qquad e \longmapsto  - k_i . k_i . e \]
and consider the differential operator
\[ - L^\circ = \triangle - \lambda_\CC + \sum_{i=1}^{n-r} 2 \mu_i \frac{\cosh \alpha_i(v) - 1}{\sinh^2 \alpha_i(v)} + \sum_{i=1}^{n-r} \ctanh(\alpha_i(v)) \partial_{\alpha_i^\#}\]
acting on scalar functions on $\af$.
The third term is a smooth, rotationally invariant function on $M$ and hence it is a spherical model coming from a spherical function $\mu \in C^\infty(M; \IR)$.
Using the discussion from subsection \ref{subsec:radialsections} (in the case $E = \IR$), we find that $L^\circ$ corresponds to the differential operator $\triangle - \lambda_\CC + \mu$ acting on spherical functions on $M$.
Let $(k_t^\circ)_{t>0}$ be the heat kernel of $\triangle - \lambda_\CC + \mu$ centered at $p_0$, i.e.
\[ \partial_t k_t^\circ = \triangle k_t^\circ + (- \lambda_\CC + \mu) k_t^\circ \qquad \text{and} \qquad k_t^\circ \xrightarrow{t \to 0} \delta_{p_0}. \]
By uniqueness, $k_t^\circ$ is spherical and its spherical model $K_t^\circ$ satisfies
\[ \partial_t K_t^\circ = - L^\circ K_t^\circ \qquad \]
We can show that $K_t^\circ$ bounds $K_t$:

\begin{Lemma} \label{Lem:KtKtcirc}
 For every $t > 0$ and $v \in \af$ denote by $K_t(v)(\min)$ resp. $K_t(v) (\max)$ the minimal resp. maximal eigenvalue of the endomorphism $K_t(v)$.
 Then
\[ 0 < K_t(v)(\min) \leq K_t(v) (\max) \leq K_t^\circ (v). \]
 Moreover, $K_t(v)(\max)$ is a subsolution to the heat operator $\partial_t + L^\circ$ in the following sense:
 If $(G_t)_{t\geq t_0} \in C^\infty(\af)$ is a solution to the equation $\partial_t G_t = - L^\circ G_t$ and a spherical model at all times, then $K_{t_0} (v)(\max) \leq G_{t_0}$ implies $K_t(v)(\max) \leq G_t$ for all $t \geq t_0$.
\end{Lemma}
\begin{proof}
The proof makes use of the maximum principle.
We will first establish the bound $K_t(v)(\min) > 0$.
If the inequality was not true then, by the local behavior of the heat kernel for small times, we would find some $\varepsilon > 0$ and some first time $t' > 0$ such that there are $v' \in \af$ and $e' \in E_0$ with $|e'|=1$ such that
\[ \langle K_t(v)e,e \rangle \geq - \varepsilon \]
holds for all $t \leq t'$, $v \in \af$ and $e \in E_0$ with $|e| = 1$ with equality for $t=t'$, $v=v'$ and $e = e'$.
This implies $K_{t'}(v')e' = K_{t'} (v') (\min) e' = - \varepsilon e'$ and
\begin{equation} \label{eq:Ktleeq0}
 \langle \partial_t K_{t'}(v')e', e' \rangle \leq 0.
\end{equation}
as well as $\langle \partial_u K_{t'}(v') e', e' \rangle = 0$ for any direction $u \in \af$ and $\langle \triangle K_{t'}(v') e', e' \rangle \geq 0$.
As for the zero order terms we compute (note that due to the invariance of $\langle \cdot, \cdot \rangle$ on $E_0$ under the action of $K$, we have $\langle k_i . e_1, e_2 \rangle + \langle e_1, k_i . e_2 \rangle = 0$ for any $i = 1, \ldots, n-r$ and $e_1, e_2 \in E_0$)
\begin{multline*}
 \langle K_{t'}(v')(k_i. k_i. e' ), e' \rangle = \langle K_{t'}(v')(e'), k_i . k_i . e' \rangle \\ = K_{t'}(v')(\min) \langle k_i . k_i . e', e' \rangle = \varepsilon \langle k_i . e', k_i . e' \rangle
\end{multline*}
and
\[ \langle k_i . (K_{t'}(v') (k_i . e')), e' \rangle = - \langle K_{t'}(v')(k_i . e'), k_i. e' \rangle \leq \varepsilon \langle k_i. e', k_i . e' \rangle . \]
So
\begin{multline*}
\langle \partial_t K_t(v') e', e' \rangle \geq \sum_{i=1}^{n-r} \bigg\langle \frac1{\sinh^2 \alpha_i(v')} K_{t'}(v') (k_i.k_i.e') + \frac{\cosh^2 \alpha_i(v')}{\sinh^2 \alpha_i(v')} k_i. k_i . K_{t'}(v')(e') \\ - 2\frac{\cosh \alpha_i(v')}{\sinh^2 \alpha_i(v')} k_i. (K_{t'}(v')(k_i.e')), e' \bigg\rangle \geq \varepsilon \sum_{i=1}^{n-r} \bigg( \frac{1- \cosh \alpha_i(v')}{\sinh \alpha_i(v')} \bigg)^2 | k_i . e' |^2 > 0,
\end{multline*}
contradicting (\ref{eq:Ktleeq0}).

We will now prove the claim that $K_t(v)(\max)$ is a subsolution to the heat operator $\partial_t + L^\circ$.
The bound $K_t(v)(\max) \leq K_t^\circ(v)$ follows with a little more effort and will not be proven here since we will not need it.

Denote by $A$ a fixed constant that will be determined later.
Let $\varepsilon > 0$ and assume again that there is some first time $t' > t_0$ such that there are some $v' \in \af$, $e' \in E_0$ with $|e' | = 1$ such that
\[  \langle K_t(v)e, e \rangle \leq G_t(v) + \varepsilon e^{At} \]
holds for all $t_0 \leq t \leq t'$, $v \in \af$ and $e \in E_0$ with $|e| = 1$ and equality is true for $t = t'$, $v = v'$ and $e=e'$.
This implies $K_{t'}(v')e' = K_{t'}(v')(\max) e'$ and $K_{t'}(v')(\max) = G_{t'}(v') + \varepsilon e^{At'}$.

Obviously,
\begin{equation} \langle \partial_t K_{t'} (v') e', e' \rangle \geq \partial_t G_{t'}(v') + \varepsilon A e^{At'}. \label{eq:dtK}
\end{equation}
Moreover, for any direction $u \in \af$, we have $\langle \partial_u K_{t'}(v') e', e' \rangle = \partial_u G_{t'}(v')$ and $\langle \triangle K_{t'} (v') e', e' \rangle \leq \triangle G_{t'}(v')$.
So
\begin{multline}
 \bigg\langle \triangle K_{t'}(v')e' + \sum_{i=1}^{n-r} \ctanh(\alpha_i(v')) (\partial_{\alpha_i^{\#}} K_{t'})(v')(e'), e' \bigg\rangle \\
\leq \triangle G_{t'}(v') + \sum_{i=1}^{n-r} \ctanh(\alpha_i(v')) \partial_{\alpha_i^{\#}} G_{t'}(v'). \label{eq:LapK}
\end{multline}
Since $K_{t'}(v')(e') = K_{t'}(v')(\max) e'$ and $K_{t'}(v')(\max) > 0$, we can estimate
\[  \sum_{i=1}^{n-r} \left\langle k_i.k_i. K_{t'} (v') (e'), e' \right\rangle \leq - \lambda_\CC K_{t'}(v')(\max). \]
Moreover,
\begin{multline*} - \langle k_i. (K_{t'}(v')(k_i. e')), e' \rangle =  \langle K_{t'}(v') k_i. e', k_i. e' \rangle \\ 
\leq K_{t'}(v')(\max) \langle k_i. e', k_i . e' \rangle = - K_{t'}(v')(\max) \langle e', k_i.k_i.e' \rangle .
\end{multline*}
This implies (we use $\frac{\cosh^2 \alpha_i(v')}{\sinh^2 \alpha_i(v')} = 1 + \frac{1}{\sinh^2 \alpha_i(v')}$ and $K_{t'} ( v' ) (e') = K_{t'} (v') (\max ) e'$ here)
\begin{multline*}
 \sum_{i=1}^{n-r} \bigg\langle \frac1{\sinh^2 \alpha_i(v')} K_{t'}(v') (k_i.k_i.e') + \frac{\cosh^2 \alpha_i(v')}{\sinh^2 \alpha_i(v')} k_i. k_i . K_{t'}(v')(e') \\ - 2\frac{\cosh \alpha_i(v')}{\sinh^2 \alpha_i(v')} k_i. (K_{t'}(v')(k_i.e')), e' \bigg\rangle \\
 =
 \sum_{i=1}^{n-r} \bigg( \big\langle k_i. k_i . K_{t'} (v') ( e' ), e' \big\rangle + \frac2{\sinh^2 \alpha_i(v')} \big\langle k_i . k_i . K_{t'}(v') (e') , e' \big\rangle  \\
 + 2\frac{\cosh \alpha_i(v')}{\sinh^2 \alpha_i(v')} \big\langle  K_{t'}(v')(k_i.e'), k_i . e' \big\rangle \bigg) \\
\leq \sum_{i=1}^{n-r} \bigg( - \lambda_\CC  + \frac2{\sinh^2 \alpha_i(v')} \big\langle k_i . k_i . e' , e' \big\rangle  
 + 2\frac{\cosh \alpha_i(v')}{\sinh^2 \alpha_i(v')} \big\langle k_i.e', k_i . e' \big\rangle \bigg) K_{t'} (v') (\max ) \\
 = \sum_{i=1}^{n-r} \bigg( - \lambda_\CC  - 2 \frac{\cosh \alpha_i (v') - 1}{\sinh^2 \alpha_i(v')} \big\langle k_i . k_i . e' , e' \big\rangle   \bigg) K_{t'} (v') (\max ) \\
 \leq \bigg( - \lambda_\CC + \sum_{i=1}^{n-r} 2\mu_i \frac{\cosh \alpha_i (v') - 1}{\sinh^2 \alpha_i (v')} \bigg) K_{t'}(v')(\max).
\end{multline*}
So with (\ref{eq:LapK}) we obtain
\begin{multline*} \langle \partial_t K_{t'}(v') e', e' \rangle \leq - L^\circ G_{t'}(v') - \bigg( - \lambda_\CC + \sum_{i=1}^{n-r} 2\mu_i \frac{\cosh \alpha_i(v') - 1}{\sinh^2 \alpha_i(v')} \bigg) G_{t'}(v') \\ 
+ \bigg( - \lambda_\CC + \sum_{i=1}^{n-r} 2\mu_i \frac{\cosh \alpha_i (v') - 1}{\sinh^2 \alpha_i (v')} \bigg) K_{t'}(v')(\max). \end{multline*}
Note that in the case in which $\alpha_i (v) = 0$ for some $i$, this inequality can be deduced in a similar way.
Combining the inequality with (\ref{eq:dtK}) and using the fact that $K_{t'}(v')(\max) = G_{t'}(v') + \varepsilon e^{At'}$, we conclude
\[ \varepsilon A e^{A t'} \leq \varepsilon \bigg( - \lambda_\CC + \sum_{i=1}^{n-r} 2\mu_i \frac{\cosh \alpha_i (v') - 1}{\sinh^2 \alpha_i (v')} \bigg) e^{At'}. \]
For sufficiently large $A$ this yields a contradiction.
\end{proof}

\subsection{The rank 1 case} \label{subsec:L1decrank1}
Assume in this subsection that $M$ has rank $1$.
Then $\af \cong \IR$ and we simply write $\alpha_i = \alpha_i(1)$, i.e. $\alpha_i(r) = \alpha_i r$.
We will now prove Theorems \ref{Thm:dectrianglerank1} and \ref{Thm:Asymptriangle}:
\begin{proof}[Proof of Theorem \ref{Thm:dectrianglerank1}]
For the lower bounds observe that by definition of $\lambda_L$, there is a (parabolically invariant) section $f \in C^\infty(M; E)$ with $\triangle f = - \lambda_L f$.
Hence its convolution with the heat kernel satisfies $f * k_t = e^{-\lambda_L t} f$, and thus the $L^1$-norm of $e^{\lambda_L t} k_t$ must be bounded from below.

In order to establish the upper bounds, set
\[ H_t^{(p)} = \Big( \int_0^\infty \Big| \prod_{i=1}^{n-1} \sinh (\alpha_i r) \Big| \cdot \left( K_t(r)(\max) \right)^p dr \Big)^{1/p} \]
for $p \geq 1$.
Since $\End (E_0)$ is finite dimensional, there are constants $c_p$ and $C_p$ depending on $p$ such that
\[ c_p H_t^{(p)} \leq \Vert k_t \Vert_{L^p(M)} \leq C_p H_t^{(p)}. \]
By definition of $\lambda_B$, there is a Bochner formula of the form $- \triangle = D^* D + \lambda_B$.
A basic application of Stoke's Theorem yields
\[ \frac{d}{dt} \Vert k_t \Vert_{L^2(M)}^2 = - 2 \lambda_B \Vert k_t \Vert_{L^2(M)}^2 - 2 \Vert D k_t \Vert_{L^2(M)}^2 \leq   - 2 \lambda_B \Vert k_t \Vert_{L^2(M)}^2.\]
Hence for $t \geq 1$
\begin{equation} H_t^{(2)} \leq \frac1{c_2} \Vert k_t \Vert_{L^2} \leq C e^{-\lambda_B t}. \label{eq:M2} \end{equation}

Now from Lemma \ref{Lem:KtKtcirc} we know that $K_t(v)(\max)$ is a subsolution to the heat operator $\partial_t + L^\circ$ and hence
\begin{multline*} \frac{d}{dt} H^{(1)}_t + \lambda_L H^{(1)}_t \\
\leq  \int_0^\infty \Big| \prod_{i=1}^{n-1} \sinh (\alpha_i r) \Big| \Big( \partial_r^2 K_t(r)(\max) + \sum_{i=1}^{n-1} \alpha_i \ctanh(\alpha_i r) \partial_r K_t(r)(\max) \Big) dr \\ + \int_0^\infty \Big| \prod_{i=1}^{n-1} \sinh (\alpha_i r) \Big| \cdot \sum_{i=1}^{n-1} 2\mu_i  \frac{\cosh (\alpha_i r) - 1}{\sinh^2 (\alpha_i r)} K_t(r)(\max) dr.
\end{multline*}
By integration by parts or Green's formula applied to the corresponding spherical function, the first integral vanishes.
It remains to bound the second integral.

Let $\delta > 0$.
Choose the $i_0$ for which $\alpha_{i_0}$ is minimal and recall that $a = \max \{ (\sum_{i=1}^{n-1} \alpha_i ) / \alpha_{i_0}, 2 \}$.
Using H\"older's inequality, we can bound the second integral by
\begin{multline*}
 C \bigg( \int_0^\infty \Big| \prod_{i=1}^{n-1} \sinh (\alpha_i r) \Big|  \left( \frac{\cosh (\alpha_{i_0} r ) - 1}{\sinh^2 (\alpha_{i_0} r)} \right)^{a+\delta} dr \bigg)^{\frac1{a+\delta}} \\ \times 
\bigg( \int_0^\infty \Big| \prod_{i=1}^{n-1} \sinh (\alpha_i r) \Big| (K_t(r)(\max))^{1+\frac1{a+\delta-1}} dr \bigg)^{\frac{a+\delta-1}{a+\delta}} \\
\leq C  \delta^{-\frac1{a+\delta}} H_t^{( 1+\frac1{a+\delta-1} )}
\leq C \delta^{-\frac1{a+\delta}} \big( H_t^{(1)} \big)^{\frac{a-2+\delta}{a+\delta}} \big( H_t^{(2)} \big)^{\frac2{a+\delta}}.
\end{multline*}
We can rewrite this inequality as
\[ \frac{d}{d t} \left( e^{\lambda_L t} H_t^{(1)} \right)^{\frac2{a+\delta}} \leq \frac2{a+\delta} C \delta^{-\frac1{a+\delta}} \left( e^{\lambda_L t} H_t^{(2)}  \right)^{\frac2{a+\delta}}. \]
In the cases $\lambda_L < \lambda_B$ and $\lambda_L > \lambda_B$, we can integrate this inequality using (\ref{eq:M2}) to conclude that $e^{\lambda_L t} H_t^{(1)}$ or $e^{\lambda_B t} H_t^{(1)}$ stays bounded.
If $\lambda_L = \lambda_B$, we find
\[ \frac{d}{d t} \left( e^{\lambda_L t} H_t^{(1)} \right)^{\frac2{a+\delta}} \leq C \delta^{-\frac1{a+\delta}}. \]
Hence for $\delta < 1$
\[ e^{\lambda_L t} H_t^{(1)} \leq C \left( \delta^{-\frac1{a+\delta}} t + 2 \right)^{\frac{a+\delta}2} \leq C \delta^{-1/2} (t+2)^{\frac{a+\delta}2}. \]
The theorem follows for $\delta = 1/\log (t+2)$.
\end{proof}

\begin{proof}[Proof of Theorem \ref{Thm:Asymptriangle}]
Note first that for an appropriate constant $V_0$
\[ \vol B_r(p_0) = V_0 \int_0^r \prod_{i=1}^{n-1} \sinh (\alpha_i r') dr' \]
and hence if we set $a = \sum_{i=1}^{n-1} \alpha_i$, we find that $\vol B_r(p_0)$ is asymptotic to $e^{ar}$.
So for $r \geq 1$ we have $c e^{ar} < \vol B_r(p_0) < C e^{ar}$ for large $r$.
For $r \leq 1$, can simply compare with the Euclidean volume growth: $c r^n < \vol B_r(p_0) < C r^n$.

Now consider the heat kernel $k_t$.
For $t < \frac12$, the inequality follows from Proposition \ref{Prop:CLY}.
For $t \geq \frac12$, we can argue as in the proof of Theorem \ref{Thm:dectrianglerank1} and conclude
\[ \Vert k_t \Vert_{L^2(M)} \leq C e^{-\lambda_B t}. \]
By convolution and Cauchy-Schwarz, this gives us an $L^\infty$-bound for $t \geq \frac12$
\begin{equation} \label{eq:Linftyestimate}
 |k_t|(p) = |k_{t/2} * k_{t/2}|(p) \leq C_0 e^{-\lambda_B t} \leq C_0 e^{-\lambda_0 t}.
\end{equation}
Observe that the quantities $k_t(\max)$ and $|k_t|$ are comparable: $C^{-1} |k_t| \leq k_t(\max) \leq |k_t|$.
Now recall that by Lemma \ref{Lem:KtKtcirc}, the spherical model $K_t(\max)$ is a subsolution to the heat operator $\partial_t + L^\circ$ where 
\[ - L^\circ = \partial_r^2 + \sum_{i=1}^{n-1} \alpha_i \ctanh (\alpha_i r) \partial_r - \lambda_L + \sum_{i=1}^{n-1} 2\mu_i \frac{\cosh (\alpha_i r) - 1}{\sinh^2 (\alpha_i r)}. \]

Let $\delta = \frac12 \min_i \alpha_i$ and set $F(r) = e^{-ar} - e^{-(a+\delta)r}$.
Then there is some $r_0$ such that $F'(r) < 0$ for $r \geq r_0$ and we have
\begin{multline*}
 - L^\circ F + \lambda_L F < F''(r) + \sum_{i=1}^{n-1} \alpha_i F'(r) + \sum_{i=1}^{n-1} 2\mu_i \frac{\cosh(\alpha_i r) - 1}{\sinh^2 (\alpha_i r)} F(r) \\
< - \delta (a+\delta) e^{-(a+\delta)r} + C e^{-2 \delta r} (e^{-ar} - e^{-(a+\delta)r}).
\end{multline*}
So, possibly after increasing $r_0$, we can assume that $(-L^\circ F + \lambda_L F)(r) < 0$ for all $r \geq r_0$.
This shows that $e^{-\lambda_L t} F(r)$ and hence also $e^{-\lambda_0 t} F(r)$ is a supsolution for the heat operator $\partial_t + L^\circ$ on the domain $\{ r \geq r_0 \}$.

Now choose $C_1$ so large that $C_1 F(r_0) \geq C_0$.
By (\ref{eq:Linftyestimate}), we have the boundary estimate $K_t(r)(\max)(r_0) \leq C_1 e^{-\lambda_0 t} F(r_0)$ for $t \geq \frac12$.
Moreover, using Proposition \ref{Prop:CLY} we find, after possibly increasing $C_1$, that $K_{\frac12}(r)(\max) \leq C_1 e^{-\lambda_0 \frac12} F(r)$ for $r \geq r_0$.
So, by the maximum principle this implies $K_t(r)(\max)(r) \leq C_1 e^{-\lambda_0 t} F(r)$ for all $t \geq \frac12$ and $r \geq r_0$.
Since $F(r) < e^{-ar}$, this proves the claim.
\end{proof}

\subsection{The $L^1$-decay in the general rank case} \label{subsec:L1decgenrank}
We will now carry out the proof of Theorem \ref{Thm:dectrianglegenrank} for the case in which $M$ is allowed to have rank greater than $1$.
The difficulty here comes from the fact that the functions $\frac{\cosh \alpha_i - 1}{\sinh^2 \alpha_i}$ are only decaying towards one coordinate direction.
We will resolve this issue by controlling certain $L^2L^1$-norms of $K_t$, which allow us to reduce dimensions step by step.
Those norms will correspond to the possible splittings $\af = \ov{\af}_\WW \oplus \un{\af}_\WW$ for walls $\WW \subset \CC$ (see subsection \ref{subsec:crosssec}) and will be controlled using Bochner formulae for the corresponding symmetric spaces $\ov{M}_\WW$.
In the case in which $M$ is a product of rank~$1$ symmetric spaces, this program can be carried out without any problems:
The spaces $\ov{M}_\WW$ are factors of $M$ and the spherical model $K_t$ on $\af$ can be approximated by a spherical model on $\ov{\af}_\WW$ and an $\un{N}_\WW$-invariant section on $\un{\af}_\WW$ as long as we are far enough away from the origin on $\un{\af}_\WW$.
However in the general case, the domain, on which the spherical model resembles this mixed spherical-parabolic model, has a more complicated geometry, due to the lack of orthogonality of the roots.
Therefore, a more careful localization has to be carried out.

For the moment let $\lambda_0$ be an arbitrary constant.
We will need the following results:
\begin{Lemma} \label{Lem:highder}
There are constants $C_m < \infty$ such that the following holds:
Assume that $\Vert k_t \Vert_{L^1(M)} \leq H e^{-\lambda_0 t}$ for $t \in [0,T]$.
Then
\[ \Vert \nabla^m k_t \Vert_{L^1(M)} \leq C_m H e^{-\lambda_0 t} \qquad \text{for} \qquad t \in [1, T]. \]
\end{Lemma}
\begin{proof}
This follows from the fact that $\nabla^m k_t = \nabla^m k_1 * k_{t-1}$ and Young's inequality. 
\end{proof}

\begin{Lemma} \label{Lem:Sobolev}
There are constants $d > 0$ and $C < \infty$ such that for every wall $\WW \subset \CC$ we have the following inequality on the cross-section $\ov{M}_\WW$ of dimension $\ov{n}$:
Let $g \in C^\infty(\ov{M}_\WW; E_\WW) \otimes E^*_0$ be a spherical section.
Then for any $p \in \ov{M}_{\WW}$ we have
\[ |g|(p) \leq C e^{-dr} \Vert g \Vert_{W^{1,\ov{n}}(\ov{M}_\WW)} \]
where $r = d(p_0,p)$ and $W^{1,\ov{n}}$ denotes the Sobolev norm.
\end{Lemma}
\begin{proof}
Since there are only finitely many cross-sections $\ov{M}_\WW$ of $M$, it suffices to show the inequality on $M$.
By Sobolev embedding, the inequality holds  for some uniform $C$ whenever $r \leq 10$.
Assume now $r > 10$ and consider the orbit $O = K. p$.
There are constants $c, d > 0$ that are independent of $p$ such that we can find $N = \lfloor c e^{dr} \rfloor$ points $p_1, \ldots, p_N \in O$ whose pairwise distance is greater than $2$.
Then on each $B_k = B_1(p_k)$ we have by Sobolev embedding
\[ |g|(p) = |g|(p_k) \leq C \Vert g \Vert_{W^{1, n}(B_k)}. \]
Hence
\[ N |g|(x) \leq \sum_{k=1}^N C \Vert g \Vert_{W^{1, n}(B_k)} \leq C \Vert g \Vert_{W^{1,n}(M)}. \]
This yields the desired bound.
\end{proof}

We will need appropriate cutoff functions that specify the regions in which we compare the spherical model $K_t$ with some mixed spherical-parabolic models.
We will use a parameter $\sigma > 10$ to specify the accuracy with which this comparison holds.
For every wall $\WW \subset \CC$ consider the splitting $\af = \ov{\af}_\WW \oplus \un{\af}_\WW$.
Corresponding to $\WW$ and the parameter $\sigma$, we will define cutoff functions $\ov{\eta}_{\sigma}^\WW \in C^\infty(\ov{\af}_{\WW})$ and $\un{\eta}_{\sigma}^\WW \in C^\infty(\un{\af}_\WW)$ such that the support of $\eta^{\WW}_{\sigma} = \ov{\eta}^{\WW}_{\sigma} \un{\eta}^{\WW}_{\sigma} \in C^\infty(\af)$ and the region in which $\eta^{\WW}_{\sigma}$ equals $1$ resemble the wall $\WW$ in a coarse sense.

In order to do this, we first define regions that will help us to characterize the behavior of the $\eta^{\WW}_{\sigma}$.
Let $(a_{\WW})_{\WW \subset \CC}, (b_{\WW})_{\WW \subset \CC}$ be numbers greater than $1$, which we will determine in the next Lemma and define the regions $X^\WW_\sigma, S^\WW_\sigma, R^{\WW}_\sigma \subset \af$ as follows:
\begin{alignat*}{2}
X^\WW_\sigma &= \{ \ov{v} + \un{v} \in \ov{\af}_\WW \oplus \un{\af}_\WW \; : \; |\ov{v}| \leq a_\WW (\sigma-1), \; & \un{\alpha}(\un{v}) &\geq 0 \; \text{for all} \; \un{\alpha} \in \un{\mathcal{B}}^+_\WW \} \\
S^\WW_\sigma &= \{ \ov{v} + \un{v} \in \ov{\af}_\WW \oplus \un{\af}_\WW \; : \; |\ov{v}| \leq a_\WW \sigma, \; & \un{\alpha}(\un{v}) &\geq b_\WW \sigma \; \text{for all} \; \un{\alpha} \in \un{\mathcal{B}}^+_\WW \} \\
R^\WW_\sigma &= \{ \ov{v} + \un{v} \in \ov{\af}_\WW \oplus \un{\af}_\WW \; : \; |\ov{v}| \leq a_\WW (\sigma - 1), \; & \un{\alpha}(\un{v}) &\geq b_\WW (\sigma + 1) \; \text{for all} \; \un{\alpha} \in \un{\mathcal{B}}^+_\WW \}
\end{alignat*}
We will later identify $S^\WW_\sigma$ as containing the support of $\eta^\WW_\sigma$, $R^\WW_\sigma$ as a region in which $\eta^\WW_\sigma$ is constantly equal to $1$ and the regions $X^{\WW'}_\sigma$ for $\WW' \in \partial \WW$ will serve to cover a certain part of the support of $\partial \eta^\WW_\sigma$ (namely $\supp \ov{\eta}^\WW_\sigma \partial \un{\eta}^\WW_\sigma$).
We need the following geometric identities:

\begin{figure}[t]
\caption{The regions $X^\WW_\sigma$, $S^\WW_\sigma$ and $R^\WW_\sigma$.}
\begin{center}
\begin{picture}(0,0)%
\hspace{16mm}\includegraphics[width=11cm]{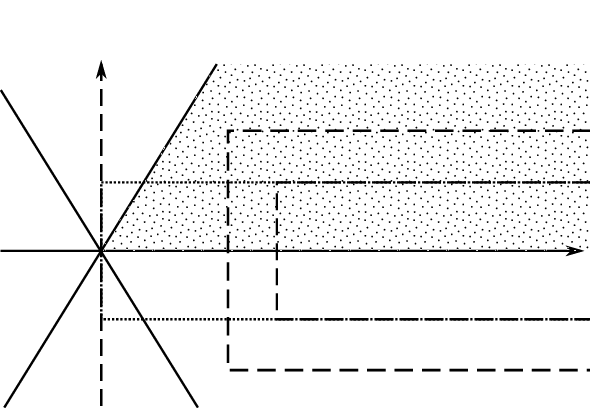}%
\end{picture}%
\setlength{\unitlength}{2863sp}%
\begingroup\makeatletter\ifx\SetFigFont\undefined%
\gdef\SetFigFont#1#2#3#4#5{%
  \reset@font\fontsize{#1}{#2pt}%
  \fontfamily{#3}\fontseries{#4}\fontshape{#5}%
  \selectfont}%
\fi\endgroup%
\begin{picture}(11195,4100)(518,-5050)
\put(3500,-2800){\makebox(0,0)[lb]{\smash{{\SetFigFont{12}{14.4}{\familydefault}{\mddefault}{\updefault}$X^\WW_\sigma$}}}}
\put(4500,-2100){\makebox(0,0)[lb]{\smash{{\SetFigFont{12}{14.4}{\familydefault}{\mddefault}{\updefault}$S^\WW_\sigma$}}}}
\put(5100,-2800){\makebox(0,0)[lb]{\smash{{\SetFigFont{12}{14.4}{\familydefault}{\mddefault}{\updefault}$R^\WW_\sigma$}}}}
\put(2900,-1000){\makebox(0,0)[lb]{\smash{{\SetFigFont{12}{14.4}{\familydefault}{\mddefault}{\updefault}$\ov{\af}_\WW$}}}}
\put(8600,-3000){\makebox(0,0)[lb]{\smash{{\SetFigFont{12}{14.4}{\familydefault}{\mddefault}{\updefault}$\un{\af}_\WW$}}}}
\put(6000,-1300){\makebox(0,0)[lb]{\smash{{\SetFigFont{12}{14.4}{\familydefault}{\mddefault}{\updefault}$\CC$}}}}
\put(6000,-3400){\makebox(0,0)[lb]{\smash{{\SetFigFont{12}{14.4}{\familydefault}{\mddefault}{\updefault}$\WW$}}}}
\end{picture}%
\end{center}
\end{figure}
\begin{Lemma} \label{Lem:regions}
There are choices for $a_\WW, b_\WW > 1$ (which we will henceforth fix) such that for any wall $\WW \subset \CC$ and all $\sigma > 10$:
\begin{enumerate}[(1)]
\item $\un{\alpha}(v) \geq \sigma$ whenever $v \in S^\WW_\sigma$ for all $\un{\alpha} \in \un{\Delta}^+_\WW$.
\item We can cover a certain boundary part of $S^\WW_\sigma$ by $X^{\WW'}_\sigma$ for $\WW' \in \partial \WW$:
\[ \{ \ov{v} + \un{v} \in S^\WW_\sigma \;\; : \;\; \un{\beta}(\un{v}) \leq b_\WW (\sigma + 1) \; \text{for some} \; \un{\beta} \in \un{\mathcal{B}}^+_\WW \} \subset \bigcup_{\WW' \in \partial \WW} X^{\WW'}_\sigma. \]
Recall, that $\partial \WW$ denotes the set of all codimension $1$ walls of $\WW$.
\item For any $f \geq 1$ we have
\[ X^\WW_\sigma \subset R^\WW_{f \sigma} \cup \bigcup_{\WW' \in \partial \WW} X^{\WW'}_{f \sigma}. \]
\end{enumerate}
\end{Lemma}
\begin{proof}
Recall that the walls $\WW$ of $\CC$ stand in one-to-one correspondence with splittings $\mathcal{B}^+ = \ov{\mathcal{B}}^+_\WW \dotcup \un{\mathcal{B}}^+_\WW$ of the basis $\mathcal{B}^+$ and that $\ov{\af}_\WW = \spann (\ov{\mathcal{B}}^+_\WW)^\#$.

For property (1) observe that for $v = \ov{v} + \un{v} \in S^\WW_\sigma$ and $\un{\alpha} \in \un{\Delta}^+_\WW$ we have
\[ \un{\alpha}(v) = \un{\alpha}(\ov{v}) + \un{\alpha}(\un{v}) \geq - C_0 a_\WW \sigma + b_\WW \sigma \]
for some large constant $C_0$.
So property (1) can be ensured if
\begin{equation} \label{eq:aWbW1}
 b_\WW - C_0 a_\WW \geq 1 \qquad \text{for all} \qquad \WW \subset \CC.
\end{equation}

As for property (2) consider $v = \ov{v} + \un{v} \in S^\WW_\sigma$ and assume that $\un{\beta}(\un{v}) \leq b_\WW (\sigma + 1)$ for some $\un{\beta} \in \un{\mathcal{B}}^+_\WW$.
Choose $\un{\gamma} \in \un{\mathcal{B}}^+_\WW$ such that
\[ \frac{\un{\gamma}(\un{v})}{| \proj_{\un{\af}_\WW}(\un{\gamma})|} \leq \frac{\un{\alpha}(\un{v})}{| \proj_{\un{\af}_\WW}(\un{\alpha})|} \qquad \text{for all} \qquad \un{\alpha} \in \un{\mathcal{B}}^+_\WW. \]
Then $\un{\gamma} (\un{v}) \leq C_1 \un{\beta}(\un{v}) \leq C_1 b_\WW (\sigma + 1)$.
Let $\WW' \in \partial \WW$ be the wall for which $\un{\mathcal{B}}^+_{\WW'} = \un{\mathcal{B}}^+_\WW \setminus \{ \un{\gamma} \}$.
Then $\ov{\af}_{\WW'} = \spann ( \ov{\af}_\WW \cup \{ \un{\gamma}^\# \})$.
Hence, if we consider the splitting $v = \ov{v}' + \un{v}' \in \ov{\af}_{\WW'} \oplus \un{\af}_{\WW'}$, we find $|\ov{v}'| \leq  |\ov{v}| + C_2 \un{\gamma}(\un{v}) \leq  a_\WW \sigma + C_2 C_1 b_\WW (\sigma + 1)$.
So, if we choose 
\begin{equation} \label{eq:aWbW2}
a_{\WW'} \geq 2 C_1 a_\WW + 4 C_2 C_1 b_\WW \qquad \text{for all} \qquad \WW' \in \partial \WW, 
\end{equation}
we can ensure that $|\ov{v}'| \leq a_{\WW'} (\sigma-1)$.
In order to conclude that $v \in X^{\WW'}_\sigma$, we still have to show that $\un{\alpha}(\un{v}') \geq 0$ for all $\un{\alpha} \in \un{\mathcal{B}}^+_{\WW'}$.
For this, observe that
\[ \un{v}' = \un{v} - \frac{\un{\gamma}(\un{v})}{| \proj_{\un{\af}_\WW}(\un{\gamma})|^2}  \proj_{\un{\af}_\WW}(\un{\gamma}^\#). \]
So for any $\un{\alpha} \in \un{\mathcal{B}}^+_{\WW'}$
\begin{multline*}
 \un{\alpha} (\un{v}') = \un{\alpha}(\un{v}) - \frac{\un{\gamma}(\un{v})}{| \proj_{\un{\af}_\WW}(\un{\gamma})|^2}  \big\langle \proj_{\un{\af}_\WW}(\un{\alpha}), \proj_{\un{\af}_\WW}(\un{\gamma}) \big\rangle \\
 \geq \un{\alpha}(\un{v}) - \frac{\un{\gamma}(\un{v})}{| \proj_{\un{\af}_\WW}(\un{\gamma})|}  | \proj_{\un{\af}_\WW}(\un{\alpha}) | \geq 0
\end{multline*}

Finally, we analyze property (3):
Let $v \in X^\WW_\sigma \setminus R^\WW_{f \sigma}$.
Then there is a $\un{\beta} \in \un{\mathcal{B}}^+_\WW$ such that $\un{\beta}(\un{v}) \leq b_\WW (f \sigma + 1)$.
As in the previous paragraph, we conclude that property (3) holds whenever (\ref{eq:aWbW2}) is satisfied.

It is now easy to see that we can choose the constants $(a_\WW)_{\WW \subset \CC}$ and $(b_\WW)_{\WW \subset \CC}$ to satisfy (\ref{eq:aWbW1}) and (\ref{eq:aWbW2}).
\end{proof}

In the following Lemma we introduce the cutoff functions $\ov{\eta}^\WW_\sigma$, $\un{\eta}^\WW_\sigma$.
\begin{Lemma} \label{Lem:eta}
We can define cutoff functions $\ov{\eta}^{\WW}_{\sigma} \in C^\infty(\ov{\af}_{\WW})$, $\un{\eta}^{\WW}_{\sigma} \in C^\infty( \un{\af}_{\WW})$ and $\eta^\WW_\sigma = \ov{\eta}^\WW_\sigma \un{\eta}^\WW_\sigma \in C^\infty( \af )$ with the following properties (for $\sigma > 10$):
\begin{enumerate}[(1)]
\item $0 \leq \ov{\eta}^{\WW}_{\sigma}, \un{\eta}^{\WW}_{\sigma} \leq 1$ and $| \partial \ov{\eta}^{\WW}_{\sigma} |, |\partial \un{\eta}^{\WW}_{\sigma} |, | \partial^2 \un{\eta}^{\WW}_{\sigma} | \leq C$ everywhere and independently of $\sigma$ and $\WW$.
Moreover, $\ov{\eta}^{\WW}_{\sigma}$ is invariant under the Weyl group $\ov{W}_\WW$.
\item $\supp \eta^{\WW}_{\sigma} \subset S^\WW_\sigma$ and $\{ \eta^\WW_\sigma = 1 \} \supset R^\WW_\sigma$.
\item On $\supp \eta^{\WW}_{\sigma}$ we have $\un{\alpha} \geq \sigma$ for all $\un{\alpha} \in \un{\Delta}^+_{\WW}$.
\item \[ \supp \ov{\eta}^{\WW}_{\sigma} \partial \un{\eta}^{\WW}_{\sigma}, \; \supp \ov{\eta}^{\WW}_{\sigma} \partial^2 \un{\eta}^{\WW}_{\sigma} \; \subset \; \bigcup_{\WW' \in \partial \WW} X^{\WW'}_\sigma. \]
\end{enumerate}
\end{Lemma}

\begin{proof}
Let $\ov{\eta}^{\WW}_{\sigma} \in C^\infty(\ov{\af}_{\WW})$ be a radially symmetric cutoff function which is $1$ on $\ov{B}_{a_{\WW}(\sigma -1)} (0) \subset \ov{\af}$ and vanishes outside $\ov{B}_{a_{\WW} \sigma}(0)$.
For $\WW = \CC$, we just set $\ov{\eta}^\WW_\sigma = 1$.
In order to define $\un{\eta}^{\WW}_{\sigma} \in C^\infty(\un{\af}_{\WW})$, we choose a cutoff function $\varphi^{\WW}_{\sigma} \in C^\infty(\IR)$ which is $1$ on $[b_{\WW} (\sigma + 1), \infty)$ and vanishes on $(-\infty, b_{\WW} \sigma]$ and we set
\[ \un{\eta}^{\WW}_{\sigma} = \prod_{\un{\alpha} \in \un{\mathcal{B}}^+_{\WW}} \varphi^{\WW}_{\sigma} ( \un{\alpha} \circ \proj_{\un{\af}_{\WW}} ). \]
For $\WW = \{ 0 \}$, we set $\un{\eta}^\WW_\sigma = 1$.
Properties (1) and (2) trivially hold.
Property (3) is just a restatement of Lemma \ref{Lem:regions} (1) and property (4) follows from Lemma \ref{Lem:regions} (2).
\end{proof}

Now consider the heat kernel $(k_t)_{t>0}$ and its spherical model $(K_t)_{t>0}$.
Let $\lambda_0$ still be an arbitrary constant and assume that $\Vert k_t \Vert_{L^1(M)} \leq H e^{-\lambda_0 t}$ for $t \in [0,T]$.
In the following analysis will always assume that $t \in [1,T]$.

Fix a wall $\WW \subset \CC$ and some $\sigma > 10$.
Unless denoted otherwise, we will mostly omit $\WW$ in the index, i.e. $\ov{\eta}_{\sigma} = \ov{\eta}^{\WW}_{\sigma}$ and $\un{\eta}_{\sigma} = \un{\eta}^{\WW}_{\sigma}$.
Consider the splitting $\af = \ov{\af} \oplus \un{\af}$ associated to $\WW$ and define the time-dependent function $G_t : \ov{\af} \to \End_{K_0} E_0$ by\footnote{In the following, whenever we write down an expression of this kind, we want to take the multiplicities of the roots of $\Delta^+$ into account, i.e. a root of higher multiplicity will appear with that multiplicity in the product or sum.}
\[ G_t = \int_{\un{\af}} \un{\eta}_{\sigma} K_t \prod_{\un{\alpha} \in \un{\Delta}^+} e^{\un{\alpha}}. \]
In the case $\WW = \{ 0 \}$ we just have $G_t = K_t$.
Observe that since $K_t(\min) > 0$ by Lemma \ref{Lem:KtKtcirc}, we can use $G_t$ to bound the weighted $L^1$-norm of $K_t$ along $\un{\af}$:
\begin{equation} \label{eq:Gboundsintegral}
 \int_{\un{\af}} \un{\eta}_{\sigma} |K_t| \prod_{\un{\alpha} \in \un{\Delta}^+} e^{\un{\alpha}} \leq C |G_t|.
\end{equation}
By construction, $G_t$ is a spherical model on $\ov{M}_\WW$ (for this note that $\ov{G}_\WW$ stabilizes $\un{A}_\WW$ since $[\ov{\mathfrak{g}}_\WW, \un{\af}_\WW ] = 0$).
Let $g_t \in C^\infty(\ov{M}; E) \otimes E_0^*$ be the associated spherical section.
We can estimate that on $\supp \ov{\eta}_{\sigma} \subset \ov{\af}$ we have (using Lemma \ref{Lem:eta} (3))
\[ |\partial^m G_t | \leq \int_{\un{\af}} \un{\eta}_{\sigma} | \partial^m K_t | \prod_{\un{\alpha}} e^{\un{\alpha}} \leq C \int_{\un{\af}} \un{\eta}_{\sigma} | \partial^m K_t | \Big| \prod_{\un{\alpha}} \sinh \un{\alpha} \Big|. \]
So using Lemma \ref{Lem:highder} and the calculus from subsection \ref{subsec:radialsections}, we can conclude that for $t \in [1,T]$
\begin{equation} \label{eq:nablagestimate}
 \Vert \nabla^m g_t \Vert_{L^1(\supp \ov{\eta}_{\sigma})} \leq C_m \Vert \nabla^m k_t \Vert_{L^1(M)} \leq C_m H e^{-\lambda_0 t}.
\end{equation}
So by Lemma \ref{Lem:Sobolev} we obtain for $t \in [1,T]$, $p \in \supp \ov{\eta}_\sigma \subset \ov{M}$ and $r = \dist(p, p_0)$
\begin{equation} \label{eq:Gestimate}
 |g_t|(p) \leq C e^{-d r} H e^{-\lambda_0 t}.
\end{equation}

We can understand the evolution of $G_t$ using the evolution equation for $K_t$ from Lemma \ref{Lem:heatkermodel}.
To simplify notation we will denote all indices $i$ corresponding to roots $\alpha_i \in \ov{\Delta}^+$ by $\ov{i}$ and we will denote the indices $\un{i}$ similarly.
Moreover, we will use the decomposition $\triangle = \ov{\triangle} + \un{\triangle}$ where $\ov{\triangle}$ denotes the Laplacian on $\ov{\af}_\WW$ and $\un{\triangle}$ the Laplacian on $\un{\af}_\WW$.
\begin{alignat}{1}
\partial_t G_t &= \int_{\un{\af}} \un{\eta}_{\sigma} \bigg[ \ov{\triangle} K_t + \un{\triangle} K_t + \sum_{\ov{\alpha}} \ctanh \ov{\alpha} \; \partial_{\ov{\alpha}^{\#}} K_t + \sum_{\un{\alpha}} (\ctanh \un{\alpha} - 1) \partial_{\un{\alpha}^\#} K_t \displaybreak[1] \notag \\
&\qquad + \sum_i \Big( \frac1{\sinh^2 \alpha_i} K_t . k_i . k_i + \frac{\cosh^2 \alpha_i}{\sinh^2 \alpha_i} k_i . k_i . K_t - 2 \frac{\cosh \alpha_i}{\sinh^2 \alpha_i} k_i . K_t . k_i \Big) \notag \\
& \qquad + \sum_{\un{\alpha}} \partial_{\un{\alpha}^{\#}} K_t \bigg] \prod_{\un{\alpha}} e^{\un{\alpha}}\displaybreak[1] \notag \\
& = \ov{\triangle} G_t + \sum_{\ov{\alpha}} \ctanh \ov{\alpha} \; \partial_{\ov{\alpha}^{\#}} G_t + \sum_{\ov{i}} \Big( \frac1{\sinh^2 \alpha_{\ov{i}}} G_t . k_{\ov{i}} . k_{\ov{i}} + \frac{\cosh^2 \alpha_{\ov{i}}}{\sinh^2 \alpha_{\ov{i}}} k_{\ov{i}} . k_{\ov{i}} . G_t \notag \\
& \qquad - 2 \frac{\cosh \alpha_{\ov{i}}}{\sinh^2 \alpha_{\ov{i}}} k_{\ov{i}} . G_t . k_{\ov{i}} \Big) + \sum_{\un{i}} k_{\un{i}} . k_{\un{i}} . G_t + \%_1 + \%_2 + \%_3 \label{eq:evofG1}
\end{alignat}
where
\begin{alignat*}{1}
\%_1 &= \int_{\un{\af}} \un{\eta}_{\sigma} \Big[ \un{\triangle} K_t + \sum_{\un{\alpha}} \partial_{\un{\alpha}^{\#}} K_t \Big] \prod_{\un{\alpha}} e^{\un{\alpha}} \displaybreak[1] \\
\%_2 &= \int_{\un{\af}} \un{\eta}_{\sigma} \sum_{\un{\alpha}} ( \ctanh \un{\alpha} - 1 ) \partial_{\un{\alpha}}^{\#} K_t \prod_{\un{\alpha}} e^{\un{\alpha}} \displaybreak[1] \\
\%_3 &= \int_{\un{\af}} \un{\eta}_{\sigma} \sum_{\un{i}} \Big( \frac1{\sinh^2 \alpha_{\un{i}}} K_t . k_{\un{i}} . k_{\un{i}} + \Big( \frac{\cosh^2 \alpha_{\un{i}}}{\sinh^2 \alpha_{\un{i}}} - 1\Big) k_{\un{i}} . k_{\un{i}} . K_t \\
&\qquad\qquad\qquad  - 2 \frac{\cosh \alpha_{\un{i}}}{\sinh^2 \alpha_{\un{i}}} k_{\un{i}} . K_t . k_{\un{i}} \Big) \prod_{\un{\alpha}} e^{\un{\alpha}}
\end{alignat*}
Recall that all but the $\%$-terms in (\ref{eq:evofG1}) together just represent the operator $\triangle_{\WW} g$ from (\ref{eq:triangleWW}) in terms of spherical models on $\ov{M}$.
Hence, by the definition of $\lambda_\WW$, there is a first order differential operator $D$ such that
\begin{equation} \label{eq:evolgt}
 \partial_t g_t = - D^* D g_t - \lambda_{\WW} g_t + \%_1 + \%_2 + \%_3.
\end{equation}

Define the time-dependent quantity
\[ B^{\WW}_{\sigma, t} = \int_{\ov{\af}} \ov{\eta}^2_{\sigma} |G_t|^2 \Big| \prod_{\ov{\alpha}} \sinh \ov{\alpha} \Big| = \Vert \ov{\eta}_{\sigma} g_t \Vert^2_{L^2(\ov{M})}. \]
If $\WW = \CC$, we just set $B^{\WW}_{\sigma,t} = |G_t|^2$.
We can compute its time derivative using (\ref{eq:evolgt}):
\begin{multline}
 \tfrac12 \partial_t B^{\WW}_{\sigma, t} = - \Vert \ov{\eta}_{\sigma} D g_t \Vert^2_{L^2(\ov{M})} - \lambda_{\WW} B^{\WW}_{\sigma, t} \\ + \int_{\ov{M}}\ov{\eta}_{\sigma} * \nabla \ov{\eta}_{\sigma} * g_t * \nabla g_t + \int_{\ov{M}} \ov{\eta}^2_{\sigma} \big( \%_1 + \%_2 + \%_3 \big) g_t. \label{eq:evBW1}
\end{multline}
The last two error terms can be estimated using the following Lemma.

\begin{Lemma} \label{Lem:percest}
There are constants $C, A < \infty$ and $c > 0$ such that we have the following estimates:
Assume that $\Vert K_t \Vert_{L^1(M)} \leq H e^{-\lambda_0 t}$ for $t \in [0,T]$.
Then we have for times $[1,T]$:
\begin{alignat*}{1}
\bigg| \int_{\ov{M}} \ov{\eta}^2_{\sigma} \big( \%_2 + \%_3 \big) g_t \bigg| &\leq C e^{- c\sigma} H^2 e^{-2\lambda_0 t} \displaybreak[1] \\
\int_{\ov{M}} \ov{\eta}_{\sigma} |\nabla \ov{\eta}_{\sigma}| | g_t | | \nabla g_t | &\leq C e^{-c \sigma} H^2 e^{-2\lambda_0 t}
\end{alignat*}
Moreover, for every $\varepsilon \in (0,1)$ we have the following estimate:
For $\WW' \subset \WW$ set $f_{\WW'} = \varepsilon^{\dim \WW' - \dim \WW}$.
Then
\[ \bigg| \int_{\ov{M}} \ov{\eta}^2_{\sigma} \%_1 g_t \bigg| \leq \big( B^{\WW}_{\sigma, t} \big)^{1/2} \Vert \ov{\eta}_{\sigma} \%_1 \Vert_{L^2(\ov{M})} \leq C \big( B^{\WW}_{\sigma, t} \big)^{1/2} \sum_{\WW' \subsetneqq \WW} e^{\varepsilon A f_{\WW'} \sigma}\big( B^{\WW'}_{f_{\WW'} \sigma, t} \big)^{1/2}.  \]
\end{Lemma}
\begin{proof}
We start with the first inequality.
Observe that by Lemma \ref{Lem:eta}(3) we know that on $\supp \ov{\eta}_\sigma \subset \ov{\af}$
\[ |\%_2| + |\%_3| \leq C e^{-c \sigma} \int_{\un{\af}} \un{\eta}_{\sigma} \big( |K_t| + |\partial K_t| \big) \prod_{\un{\alpha}} e^{\un{\alpha}}. \]
Hence, by Lemma \ref{Lem:highder} and (\ref{eq:Gestimate}) we conclude (using $e^{-d r} \leq 1$)
\begin{multline*}
\bigg| \int_{\ov{M}} \ov{\eta}^2_{\sigma} \big( \%_2 + \%_2 ) g_t \bigg| 
\leq C e^{-c \sigma} H e^{-\lambda_0 t} \int_{\ov{\af}} \ov{\eta}^2_{\sigma} \bigg( \int_{\un{\af}} \un{\eta}_{\sigma} \big( |K_t| + | \partial K_t | \big) \prod_{\un{\alpha}} e^{\un{\alpha}} \bigg) \Big| \prod_{\ov{\alpha}}  \sinh \ov{\alpha} \Big| \\
\leq C e^{-c\sigma} H e^{-\lambda_0 t} \Big( \Vert k_t \Vert_{L^1(M)} + \Vert \nabla k_t \Vert_{L^1(M)} \Big)
\leq C e^{-c\sigma} H^2 e^{-2\lambda_0 t}
\end{multline*}

The second inequality can be established in an analogous way.
This time, however, we need make use of the $e^{-d r}$-factor in (\ref{eq:Gestimate}) and  we need to make use of (\ref{eq:nablagestimate}) for $m=1$.
We find that for some $c>0$ depending on $d$
\[ \int_{\ov{M}} \ov{\eta}_{\sigma} |\nabla \ov{\eta}_{\sigma} | |g_t| |\nabla g_t|
\leq C e^{-c\sigma} H e^{-\lambda_0 t} \int_{\ov{M}} \ov{\eta}_{\sigma} | \nabla g_t | \leq C e^{-c \sigma} H^2 e^{-2\lambda_0 t}. \]

We now establish the third inequality.
To avoid confusion, we will write out the $\WW$-index again.
Observe first that $\sum_{\un{\alpha}} \un{\alpha}^\#$ is invariant under $\ov{W}_\WW$ and hence it is contained in $\un{\af}_\WW$.
Let now $\ov{v} \in \ov{\af}_\WW$.
Then, by integration by parts
\begin{multline}
 \big( \ov{\eta}^{\WW}_{\sigma} |\%_1| \big) (\ov{v}) \leq \int_{\{ \ov{v} \} \times \un{\af}_\WW} \ov{\eta}^{\WW}_{\sigma} \big( |\partial \un{\eta}^\WW_{\sigma} | + | \partial^2 \un{\eta}^\WW_{\sigma} | \big) |K_t| \prod_{\un{\alpha} \in \un{\Delta}^+_\WW} e^{\un{\alpha}} \\
 \leq C \int_{\{ \ov{v} \} \times \un{\af}_\WW} \ov{\eta}^\WW_\sigma \big( |\partial \un{\eta}^\WW_{\sigma} | + | \partial^2 \un{\eta}^\WW_\sigma | \big) | K_t | \bigg| \prod_{\un{\alpha} \in \un{\Delta}^+_\WW} \sinh \un{\alpha} \bigg|. \label{eq:perc1}
\end{multline}
Since by Lemma \ref{Lem:eta} (4) we know that the support of the integrand is covered by the $X^{\WW'}_\sigma$ for $\WW' \in \partial \WW$, we can bound the right hand side of (\ref{eq:perc1}) by $C \sum_{\WW' \in \partial \WW} Y^{\WW'}_\sigma$ where
\[ Y^{\WW'}_\sigma = \int_{\{\ov{v} \} \times \un{\af}_\WW \cap X^{\WW'}_\sigma} |K_t| \bigg| \prod_{\un{\alpha} \in \un{\Delta}^+_\WW} \sinh \un{\alpha} \bigg|. \]
Now consider any $\WW' \subsetneqq \WW$ (not necessarily of codimension $1$ in $\WW$) and let $f \geq 1$.
Then we can bound $Y^{\WW'}_\sigma$ using Lemma \ref{Lem:regions} (3) for $f = \varepsilon^{-1}$
\[ Y^{\WW'}_\sigma \leq \int_{\{ \ov{v} \} \times \un{\af}_\WW \cap X^{\WW'}_{\sigma}} \eta^{\WW'}_{\varepsilon^{-1} \sigma} |K_t| \bigg| \prod_{\un{\alpha} \in \un{\Delta}^+_\WW} \sinh \un{\alpha} \bigg| + \sum_{\WW'' \in \partial \WW'} Y^{\WW''}_{\varepsilon^{-1} \sigma}. \]
In order to bound the integral, let us first analyze its domain $\{\ov{v} \} \times \un{\af}_\WW \cap X^{\WW'}_{\sigma}$.
Observe that the set $X^{\WW'}_{\sigma}$ can be written as a direct product $X^{\WW'}_{\sigma} = \ov{X}^{\WW'}_{\sigma} \times \un{X}^{\WW'}_{\sigma}$ with respect to the splitting $\ov{\af}_{\WW'} \oplus \un{\af}_{\WW'}$.
Moreover, since $\un{\af}_{\WW'} \subset \un{\af}_\WW$, there is an orthogonal splitting $\un{\af}_\WW = \af_\perp \oplus \un{\af}_{\WW'}$ and we have $\ov{\af}_{\WW'} = \ov{\af}_{\WW} \oplus \af_\perp$.
So we can represent the domain of the integral as a product with respect to the splitting $\af = \ov{\af}_{\WW'} \oplus \un{\af}_{\WW'}$:
\[ \big( \{ \ov{v} \} \times \un{\af}_\WW \big) \cap X^{\WW'}_\sigma = \big( \{ \ov{v} \} \times \af_\perp \times \un{\af}_{\WW'} \big) \cap \big( \ov{X}^{\WW'}_\sigma \times \un{X}^{\WW'}_\sigma \big) = \big( \{ \ov{v} \} \times \af_\perp \cap \ov{X}^{\WW'}_\sigma \big) \times \un{X}^{\WW'}_\sigma \]
So by Cauchy-Schwarz
\begin{multline*}
 Y^{\WW'}_\sigma \leq \bigg( \int_{\{ \ov{v} \} \times \af_\perp \cap \ov{X}^{\WW'}_\sigma} \big( \ov{\eta}^{\WW'}_{\varepsilon^{-1} \sigma} \big)^2 \bigg( \int_{\un{X}^{\WW'}_\sigma} \un{\eta}^{\WW'}_{\varepsilon^{-1} \sigma} |K_t| \bigg| \hspace{-1mm} \prod_{\alpha \in \un{\Delta}^+_{\WW'}} \hspace{-2mm}\sinh \alpha \bigg| \bigg)^2 \bigg| \hspace{-1mm} \prod_{\alpha \in \ov{\Delta}^+_{\WW'} \cap \un{\Delta}^+_\WW} \hspace{-5mm} \sinh \alpha \bigg| \bigg)^{1/2} \\
 \times \bigg( \int_{\{ \ov{v} \} \times \af_\perp \cap \ov{X}^{\WW'}_\sigma} \bigg| \hspace{-1mm} \prod_{\alpha \in \ov{\Delta}^+_{\WW'} \cap \un{\Delta}^+_\WW} \hspace{-5mm} \sinh \alpha \bigg| \bigg)^{1/2} + \sum_{\WW'' \in \partial \WW'} Y^{\WW''}_{\varepsilon^{-1} \sigma}.
\end{multline*}
The last integral can be bounded by $C e^{2 A \sigma}$ for an appropriate $A < \infty$.
The second integral (inside the first integral) can be bounded by $C |G_t|$ using (\ref{eq:Gboundsintegral}) if we define $G_t$ for $\varepsilon^{-1} \sigma$ instead of $\sigma$.
So
\[ Y^{\WW'}_\sigma \leq C e^{A \sigma} \big( B^{\WW'}_{\varepsilon^{-1} \sigma, t} \big)^{1/2} + \sum_{\WW'' \in \partial \WW'} Y^{\WW''}_{\varepsilon^{-1} \sigma}. \]

Now recall that $f_{\WW'} = \varepsilon^{\dim \WW' - \dim \WW}$ for $\WW' \subset \WW$.
Substituting in the identity above $\varepsilon f_{\WW'}\sigma$ for $\sigma$, applying it recursively and plugging it back into (\ref{eq:perc1}) yields
\[ \big( \ov{\eta}^\WW_\sigma | \%_1 | \big) (\ov{v}) \leq C \sum_{\WW' \subsetneqq \WW} e^{A \varepsilon f_{\WW'} \sigma} \big( B^{\WW'}_{f_{\WW'} \sigma, t} \big)^{1/2}. \]
So
\[ \int_{\ov{\af}_\WW} \big( \ov{\eta}^\WW_\sigma | \%_1 | \big)^2 \bigg| \prod_{\ov{\alpha} \in \ov{\Delta}^+_\WW} \sinh \ov{\alpha} \bigg| \leq C \sum_{\WW' \subsetneqq \WW} e^{2 A \varepsilon f_{\WW'} \sigma} B^{\WW'}_{f_{\WW'} \sigma, t} . \]
This yields the desired result.
\end{proof}

So combining (\ref{eq:evBW1}) with Lemma \ref{Lem:percest}, we conclude
\begin{Lemma} \label{Lem:evBW}
There are constants $C, A < \infty$ and $c > 0$ such that the following holds:
Let $\varepsilon \in (0,1)$, consider a wall $\WW \subset \CC$ and set $f_{\WW'} = \varepsilon^{\dim \WW' - \dim \WW}$ for each $\WW' \subset \WW$.

Then, under the assumption that $\Vert k_t \Vert_{L^1(M)} \leq H e^{-\lambda_0 t}$ for $t \in [0,T]$, we have for $\sigma > 10$ and times $t \in [1,T]$
\[ \tfrac12 \partial_t B^{\WW}_{\sigma, t} \leq - \lambda_{\WW} B^{\WW}_{\sigma, t} + C e^{- 2c \sigma} H^2 e^{- 2 \lambda_0 t} + C \big( B^{\WW}_{\sigma, t} \big)^{1/2} \sum_{\WW' \subsetneqq \WW} e^{\varepsilon A f_{\WW'} \sigma}\big( B^{\WW'}_{f_{\WW'} \sigma, t} \big)^{1/2}. \]
\end{Lemma}

We will come back to this evolution inequality later.
First, we estimate the evolution of $\Vert k_t \Vert_{L^1(M)}$ in terms of the $B^{\WW}_{\sigma}$.
For this, we define the quantity
\[ S_t = \int_\CC K_t(\max) \prod_\alpha \sh \alpha  \]
and observe that $S_t$ is comparable to $\Vert k_t \Vert_{L^1(M)}$, i.e.
\[ c S_t \leq \Vert k_t \Vert_{L^1(M)} \leq C S_t  \qquad \text{for all $t>0$}.\]

\begin{Lemma} \label{Lem:evK}
There are constants $C, A < \infty$ and $c > 0$ such that:
Let $\varepsilon \in (0,1)$ and set $f_\WW = \varepsilon^{\dim \WW - \dim \CC}$.
Then at any time $t > 0$ and for $\sigma > 10$ we have the estimate
\[ \partial_t S_t \leq -(\lambda_\CC - C e^{-c \sigma}) S_t + C \sum_{\WW \subsetneqq \CC} e^{\varepsilon A f_{\WW}\sigma} \big( B^{\WW}_{f_{\WW} \sigma, t} \big)^{1/2}. \]
\end{Lemma}
\begin{proof}
Recall that by Lemma \ref{Lem:KtKtcirc} (see also the proof of Theorem \ref{Thm:dectrianglerank1})
\[ (\partial_t + \lambda_\CC) \int_{\CC} K_t(\max) \prod_\alpha \sinh \alpha \leq \int_{\CC} \Big( \sum_i 2 \mu_i \frac{\cosh \alpha_i - 1}{\sinh^2 \alpha_i} \Big) K_t(\max) \prod_\alpha \sinh \alpha. \]
By Lemma \ref{Lem:regions} (3) we have $\CC = X^\CC_\sigma \subset R^\CC_\sigma \cup \bigcup_{\WW \in \partial \CC} X^\WW_\sigma$.
So by Lemma \ref{Lem:eta} (3), we conclude that outside the regions $X^\WW_\sigma, (\WW \in \partial \CC)$, the term inside the parentheses can be bounded by $C e^{- c\sigma}$.
Hence, in order to establish the Lemma, it suffices to show that for every $\WW \in \partial \CC$, we have
\[ \int_{X^\WW_\sigma} |K_t| \Big| \prod_\alpha \sinh \alpha \Big| \leq C \sum_{\WW' \subset \WW} e^{\varepsilon A f_{\WW'} \sigma} \big( B^{\WW'}_{f_{\WW'} \sigma, t} \big)^{1/2}. \]
Analogously to the proof of Lemma \ref{Lem:percest}, we set for any wall $\WW \subsetneqq \CC$ (not only for codimension $1$ walls)
\[ Y^\WW_\sigma = \int_{X^\WW_\sigma} |K_t| \Big| \prod_\alpha \sinh \alpha \Big| \]
Now, using Lemma \ref{Lem:regions} (3) with $f = \varepsilon^{-1}$ and the splitting $X^\WW_\sigma = \ov{X}^\WW_\sigma \times \un{X}^\WW_\sigma$, we get
\begin{alignat*}{1}
Y^\WW_{\varepsilon f_\WW \sigma} &\leq \int_{X^\WW_{\varepsilon f_\WW \sigma}} \eta^\WW_{f_\WW \sigma} |K_t| \Big| \prod_\alpha \sinh \alpha \Big| + \sum_{\WW' \in \partial \WW} Y^{\WW'}_{\varepsilon f_{\WW'} \sigma} \displaybreak[0] \\
&\leq \bigg( \int_{\ov{X}^\WW_{\varepsilon f_\WW \sigma}} \big(\ov{\eta}^\WW_{f_\WW \sigma} \big)^2 \bigg( \int_{\un{X}^\WW_{\varepsilon f_\WW \sigma}} \un{\eta}^\WW_{f_\WW \sigma} |K_t| \Big| \prod_{\un{\alpha} \in \un{\Delta}^+_\WW} \sinh \un{\alpha} \Big| \bigg)^2 \Big| \prod_{\ov{\alpha} \in \ov{\Delta}^+_\WW} \sinh \ov{\alpha} \Big| \bigg)^{1/2} \\
& \qquad \times \bigg( \int_{\ov{X}^\WW_{\varepsilon f_\WW \sigma}} \Big| \prod_{\ov{\alpha} \in \ov{\Delta}^+_\WW} \sinh \ov{\alpha} \Big| \bigg)^{1/2} + \sum_{\WW' \in \partial \WW} Y^{\WW'}_{\varepsilon f_{\WW'} \sigma} \displaybreak[0] \\
& \leq C e^{\varepsilon A f_\WW \sigma} \big( B^\WW_{f_\WW \sigma, t} \big)^{1/2} + \sum_{\WW' \in \partial \WW} Y^{\WW'}_{\varepsilon f_{\WW'} \sigma}.
\end{alignat*}
Iterating this inequality yields the desired result
\end{proof}

We have now transformed our geometric problem into a problem of bounding the solutions of a system of evolution inequalities.
In the first step, we use Lemma \ref{Lem:evBW} to estimate the $B^{\WW}_{\sigma, t}$ assuming a bound on $\Vert k_t \Vert_{L^1(M)}$.

\begin{Lemma} \label{Lem:BWest}
There are constants $C, A < \infty$ and $1 > c > 0$ such that:
Assume that $\Vert K_t \Vert_{L^1(M)} \leq H e^{-\lambda_0 t}$ for $t \in [0,T]$.
Consider a wall $\WW \subset \CC$, set $\lambda^{\min}_\WW = \min_{\WW' \subset \WW} \lambda_{\WW'}$ and assume $\lambda^{\min}_\WW > \lambda_0$. \\
If $\sigma > 10$ is so large that $2 e^{-c \sigma} < \lambda^{\min}_\WW - \lambda_0$, then we have for times $t \in [1,T]$
\[ B^{\WW}_{\sigma, t} \leq C e^{- c \sigma} H^2 e^{-2 \lambda_0 t} + C e^{A \sigma} \exp(- 2(\lambda^{\min}_{\WW} - e^{- c \sigma}) t ). \]
\end{Lemma}
\begin{proof}
Let $c > 0$ be the constant from Lemma \ref{Lem:evBW}.
The constants $C,A$ will be determined in the course of the proof.

We proceed by induction over the dimension of $\WW$.
Fix $\WW \subset \CC$ and assume that the inequality is true for any $\WW' \subsetneqq \WW$.
Let $\varepsilon > 0$ be a constant whose value will be determined later and choose $(f_{\WW'})_{\WW' \subset \WW}$ according to Lemma \ref{Lem:evBW}, i.e. $f_{\WW'} = \varepsilon^{\dim \WW' - \dim \WW}$.
Then
\begin{alignat*}{1}
\tfrac12 \partial_t B^{\WW}_{\sigma, t} &\leq - \lambda_\WW B^{\WW}_{\sigma, t} + C e^{-2c \sigma} H^2 e^{-2\lambda_0 t} + C \big( B^\WW_{\sigma,t} \big)^{1/2} \sum_{\WW' \subsetneqq \WW} e^{\varepsilon A f_{\WW'} \sigma} \big( B^{\WW'}_{f_{\WW'} \sigma,t} \big)^{1/2} \\
&\leq - (\lambda_\WW - e^{-c \sigma}) B^{\WW}_{\sigma, t} + C e^{- 2c \sigma} H^2 e^{-2\lambda_0 t} + C e^{c \sigma} \sum_{\WW' \subsetneqq \WW} e^{2 \varepsilon A f_{\WW'} \sigma} B^{\WW'}_{f_{\WW'} \sigma, t} \\
&\leq  - (\lambda_\WW - e^{-c \sigma}) B^{\WW}_{\sigma, t} + C e^{-2c \sigma} H^2 e^{-2\lambda_0 t} + C \hspace{-2mm} \sum_{\WW' \subsetneqq \WW} \hspace{-2mm} e^{(2 \varepsilon A f_{\WW'} + c) \sigma} e^{- c f_{\WW'} \sigma} H^2 e^{- 2 \lambda_0 t} \\
& \qquad + C \sum_{\WW' \subsetneqq \WW} e^{(2 \varepsilon A f_{\WW'} + 1) \sigma} e^{A f_{\WW'} \sigma} \exp ( - 2(\lambda^{\min}_\WW - e^{- c \varepsilon^{-1} \sigma}) t ).
\end{alignat*}
Now choose $\varepsilon$ small enough such that $2 \varepsilon A f_{\WW'} + 1 - c f_{\WW'} \leq - 2 c$ for all $\WW' \subsetneqq \WW$ and set $A' = 2 \varepsilon A f_{\{ 0 \}} + 1 + A f_{\{ 0 \}}$.
Since $\sigma > 10$, we can find a constant $c' > 0$ such that $e^{-c \sigma} - e^{- c\varepsilon^{-1} \sigma} > c' e^{-c\sigma}$.
Applying those assumptions, we obtain the evolution inequality
\begin{multline*}
 \tfrac12 \partial_t B^{\WW}_{\sigma, t} \leq - (\lambda^{\min}_{\WW} - e^{-c \sigma}) B^{\WW}_{\sigma, t} + C e^{- 2 c \sigma} H^2 e^{-2\lambda_0 t} \\ + C e^{A' \sigma} \exp ( - 2 c' e^{-c\sigma} t -2(\lambda^{\min}_\WW - e^{-c\sigma}) t ).
\end{multline*}
So
\begin{multline*}
 \tfrac12 \partial_t \big( \exp( 2(\lambda^{\min}_{\WW} - e^{-c\sigma})t) B^{\WW}_{\sigma, t} \big) \\ 
 \leq C e^{- 2 c \sigma} H^2 \exp( 2(\lambda^{\min}_{\WW} - \lambda_0 - e^{-c\sigma}) t) + C e^{A' \sigma} \exp( - 2 c' e^{-c\sigma} t ).
\end{multline*}
Integrating this inequality and using the fact that $2(\lambda^{\min}-\lambda_0 - e^{-c\sigma}) > 2 e^{-c \sigma}$ yields
\begin{multline*}
 \exp( 2(\lambda^{\min}_{\WW} - e^{-c\sigma})t) B^{\WW}_{\sigma, t} \leq C B^{\WW}_{\sigma, 1} \\ + C e^{- c \sigma} H^2 \exp( 2(\lambda_\WW^{\min} - \lambda_0 - e^{-c\sigma})t ) + C e^{A' \sigma + c \sigma}
\end{multline*}
and hence the desired result.
\end{proof}

We can finally combine Lemmas \ref{Lem:evK} and \ref{Lem:BWest} to prove Theorem \ref{Thm:dectrianglegenrank}.

\begin{proof}[Proof of Theorem \ref{Thm:dectrianglegenrank}]
For small times, the theorem follows from Proposition \ref{Prop:CLY}.

In order to see the lower bound of Theorem \ref{Thm:dectrianglegenrank}, we proceed as in the proof of Theorem \ref{Thm:dectrianglerank1}: 
Recall that $\lambda_\CC$ is the largest eigenvalue of the operator $- \triangle_\CC = S_\CC = S_{par}$ acting on the finite dimensional vector space of all $P = A N$ invariant (i.e. parabolically invariant) sections of $E$.
Let $f \in C^\infty (M; E)$ be a section corresponding to $\lambda_\CC$, i.e. $\triangle f = - \lambda_\CC f$.
So its convolution with the heat kernel satisfies $f * k_t = e^{-\lambda_\CC t} f$.
It follows that the $L^1$-norm of $e^{\lambda_\CC t} k_t$ is uniformly bounded from below.

We will now establish the upper bound of Theorem \ref{Thm:dectrianglegenrank}.
Assume for the moment that $\lambda_0$ is an arbitrary constant satisfying $\lambda_0 < \lambda_1 = \min_{\WW \subsetneqq \CC} \lambda_\WW$ and set $H_t = \sup_{t' \in [0,t]} S_{t'} e^{\lambda_0 t'}$.
Then by Lemma \ref{Lem:evK} and Lemma \ref{Lem:BWest} as long as $2e^{-c \sigma} <  \lambda_1 - \lambda_0$ and $\sigma > 10$, we have for $t \geq 1$
\begin{multline*}
 \partial_t S_t \leq - (\lambda_\CC - C e^{-c\sigma}) S_t \\ + C \sum_{\WW \subsetneqq \CC} e^{\varepsilon A f_\WW \sigma} \Big( e^{-c f_{\WW} \sigma/2} H_t e^{-\lambda_0 t} + e^{A f_\WW \sigma/2} \exp (- (\lambda^{\min}_{\WW} - e^{-c f_{\WW} \sigma} ) t ) \Big).
\end{multline*}
So, if $\varepsilon$ is chosen small enough as to ensure $\varepsilon A f_\WW - c f_\WW / 2 < -c$ for all $\WW \subsetneqq \CC$, we obtain for $A' = A f_\WW / 2$
\begin{equation} \label{eq:finalevolution}
 \partial_t \big( S_t e^{\lambda_\CC t} \big) \leq C e^{- c\sigma} H_t e^{(\lambda_\CC - \lambda_0)t} + C e^{A' \sigma} \exp ( - (\lambda_1 - \lambda_\CC - e^{-c\sigma}) t ).
\end{equation}

Now consider first the case $\lambda_1 > \lambda_\CC$ and set $\lambda_0 = \lambda_\CC = \min_{\WW \subset \CC} \lambda_\WW$.
Choose $\delta > 0$ small enough such that $A' \delta - \lambda_1 + \lambda_\CC < - 2\delta$ and set $\sigma = \delta t$.
Then for large $t$ we can assume $e^{-c \sigma} < \delta$ and $2 e^{-c \sigma} < \lambda_1 - \lambda_0$ and we get
\[ \partial_t \big( S_t e^{\lambda_0 t} \big) \leq C e^{-c\delta t} H_t + C e^{-\delta t} \leq C e^{-c\delta t} (H_t + 1). \]
So whenever $H_t \leq 2 S_t e^{\lambda_0 t}$, we find
\[ \partial_t \log (S_t e^{\lambda_0 t} + 1) \leq 2 C e^{-c\delta t}. \]
Since the right hand side is integrable for $t \to \infty$, we conclude that $S_t e^{\lambda_0 t}$ stays bounded.

Consider now the case $\lambda_1 \leq \lambda_\CC$.
Plugging $\lambda_0 = \lambda_1 - 2 e^{-c\sigma}$ into (\ref{eq:finalevolution}) yields
\[ \partial_t \big( S_t e^{\lambda_0 t} \big) \leq  C e^{-c\sigma} H_t + C e^{A' \sigma} \exp ( (e^{-c\sigma} - 2e^{- c \sigma}) t ) \leq C e^{-c\sigma} H_t + C e^{A' \sigma}. \]
So whenever $H_t \leq 2 S_t e^{\lambda_0 t}$, we find
\[ \partial_t \big( S_t e^{\lambda_0 t} \big) \leq C e^{-c\sigma} \big( S_t e^{\lambda_0 t} \big) + C e^{A' \sigma}. \]
By Gronwall's Lemma, we conclude
\[ S_t e^{\lambda_0 t} \leq C \exp ( C e^{-c\sigma} t + (A'+1) \sigma ). \]
Choosing $\sigma = c^{-1} \log (t+2)$ yields the desired result.
\end{proof}

\section{Analysis of the Einstein operator}\label{sec:Einstop}
\subsection{Introduction}
In this section, we will apply the results from section \ref{sec:heatkernel} to the linearized Ricci deTurck equation $\partial_t h_t = - L h_t$, where $L = - \triangle - 2 \Rm$ is the Einstein operator (see \ref{subsec:RdTflow}).
Let $E = \Sym_2 T^*$ be the bundle of symmetric bilinear forms.
Observe that the zero order term $2\Rm$ is a fiberwise self-adjoint endomorphism on $E$ and hence it can be diagonalized with eigenvalues $\varphi_1, \ldots, \varphi_m$ with respect to a splitting $E = E_1 \oplus \ldots \oplus E_m$ of vector bundles.

We can integrate the extra term $2\Rm$ into Theorem \ref{Thm:dectrianglerank1} in the following way:
Redefine $\lambda_L$ to be the smallest eigenvalue of the operator $S_{par} = - \triangle - 2 \Rm : V_{par} \to V_{par}$ acting on parabolically invariant sections and $\lambda_B$ as the optimal constant $\lambda$ for Bochner formulas $-\triangle - 2\Rm = D^* D + \lambda$.
Those new constants $\lambda_L$ and $\lambda_B$ are then just the old constants minus $\varphi_i$ on each $E_i$.
Now, redefine $(k_t)_{t > 0} \in C^\infty( M; E) \otimes E_0^*$ to be the heat kernel for the operator $\partial_t + L$, i.e. $\partial_t k_t = - L k_t$.
Obviously, $k_t$ is just the old heat kernel with some extra exponential $\varphi_i$-decay on the $E_i$-component.
So with these redefinitions, Theorem \ref{Thm:dectrianglerank1} stays valid in its original reading: For $\lambda_0 = \min \{ \lambda_L, \lambda_B \}$, we have $\Vert k_t \Vert_{L^1(M)} \leq C e^{-\lambda_0 t}$.
The same is true for Theorem \ref{Thm:Asymptriangle}.

Analogously, we can integrate the $2\Rm$-term into Theorem \ref{Thm:dectrianglegenrank}:
This time, we have to redefine the constants $(\lambda_\WW)_{\WW \subset \CC}$ introduced in subsection \ref{subsec:heatkerIntro} to include the zero order term.
In order to do this, we replace the Laplace operator $\triangle$ by $-L$ in the paragraph preceding equation (\ref{eq:triangleWW}), i.e. we set $- L \widehat{f} = \widehat{f'}$.
This changes the definition of $\triangle_\WW$ and $- S_\WW$ by an extra $2\Rm$ summand and hence gives us a new $\lambda_\WW$.
Now Theorem \ref{Thm:dectrianglegenrank} continues to hold for the redefined heat kernel.

So in order to estimate the $L^1$-decay rate of $k_t$, we need to get a good bound on the redefined constants $\lambda_\WW$.
In this section we will solely be concerned with the analysis of these constants.
Our result will be:

\begin{Proposition} \label{Prop:lambdaWWpos}
Assume that $M$ is a symmetric space of noncompact type that is Einstein.
Consider the Einstein operator $L = -\triangle - 2\Rm$ acting on the vector bundle $E = \Sym_2 T^*$ of symmetric bilinear forms over $M$. 
Let $(\lambda_\WW)_{\WW \subset \CC}$ be the constants associated to $M$, $E$ and $L$ as redefined above. Then
\begin{enumerate}[(i)]
\item $\lambda_\WW \geq 0$ for all walls $\WW \subset \CC$ and hence $\lambda_0 = \min_{\WW \subset \CC} \lambda_\WW \geq 0$.
\item We have $\lambda_0 > 0$ if and only if $M$ does not contain any hyperbolic or complex hyperbolic factor in its de Rham decomposition.
\item If $M = \IH^n$ for $n \geq 3$ or $M = \IC \IH^{2n}$ for $n \geq 2$, then $\lambda_L = \lambda_\CC = 0$ and $\lambda_B = \lambda_{\{ 0 \}} > 0$.
\item If $M= \IH^2$, then $\lambda_L = \lambda_\CC > 0$ and $\lambda_B = \lambda_{\{ 0 \}} = 0$.
\end{enumerate}
\end{Proposition}
Hence in the last case in which $M$ does not contain a hyperbolic or complex hyperbolic factor, the $L^1$-norm of the heat kernel $k_t$ associated to $\partial_t + L$ is exponentially decaying as $t \to \infty$.
If $M = \IH^n$ ($n \geq 3$) or $\IC \IH^{2n}$ ($n \geq 2$), then $\Vert k_t \Vert_{L^1(M)}$ stays bounded and by Theorem \ref{Thm:Asymptriangle}, we have the bound $|k_t| < C (\vol B_r(p_0))^{-1}$ for all $t$.

This section is organized as follows:
In subsection \ref{subsec:lambdaWWprelim}, we recall the important identities and carry out some of the basic calculations.
The reader who is only interested in the rank $1$ case, will find an estimate on $\lambda_B$ in subsection \ref{subsec:lambdaWWBochner}.
For an estimate on $\lambda_L$ the reader can immediately jump to subsections \ref{subsec:af} through \ref{subsec:lambdaWWconclusion} (where he or she can always replace the index $\un{m}$ by $m$ and leave out any term with index $\ov{l}$).
In order to understand the higher rank case, subsection \ref{subsec:lambdaWWBochner} will still be important since it discusses a Bochner formula that will later be applied to cross-section $\ov{M}_\WW$ of $M$.
In subsection \ref{subsec:lambdaWWblockform}, we will find that the problem of estimating $\lambda_\WW$ reduces to an estimate on three vector bundles $E_{\ov{\pp} \un{\pp}}$, $E_{\Sym_2 \ov{\pp}}$ and $E_{\Sym_2 \un{\pp}}$.
The estimates on $E_{\ov{\pp} \un{\pp}}$ and $E_{\Sym_2 \ov{\pp}}$ will be carried out in subsections \ref{subsec:ovpunp} and \ref{subsec:Sym2ovp}.
The estimate on $E_{\Sym_2 \un{\pp}}$ is the most difficult one and will be carried out in subsections \ref{subsec:af} and \ref{subsec:afperp}.
During this discussion, a possible nullspace arises which we will then analyze in subsection \ref{subsec:nullspace}.
Finally, subsection \ref{subsec:lambdaWWconclusion} contains the proof of Proposition \ref{Prop:lambdaWWpos}.

\subsection{Preliminary calculations} \label{subsec:lambdaWWprelim}
Fix a wall $\WW \subset \CC$ and consider the splitting $\Delta^+ = \ov{\Delta}^+_\WW \dotcup \un{\Delta}^+_\WW$.
As explained in subsection \ref{subsec:crosssec}, we obtain orthogonal splittings $\af = \ov{\af}_\WW \oplus \un{\af}_\WW$, $\pp = \ov{\pp}_\WW \oplus \un{\pp}_\WW$ and $\mathfrak{k} = \ov{\mathfrak{k}}_\WW \oplus \un{\mathfrak{k}}_\WW$.
Here
\[ \ov{\pp}_\WW = \ov{\af}_\WW \oplus \bigoplus_{\ov{\alpha} \in \ov{\Delta}^+_\WW} \pp_{\ov{\alpha}} \qquad \text{and} \qquad \un{\pp}_\WW = \un{\af}_\WW \oplus \bigoplus_{\un{\alpha} \in \un{\Delta}^+_\WW} \pp_{\un{\alpha}} . \]
Moreover, we set
\[ \ov{\mathfrak{k}}_\WW = [ \ov{\pp}_\WW, \ov{\pp}_\WW] \qquad \text{and} \qquad \un{\mathfrak{k}}_\WW = \bigoplus_{\un{\alpha} \in \un{\Delta}^+_\WW} \mathfrak{k}_{\un{\alpha}} \]
We remark that $\un{\mathfrak{k}}_\WW$ is not a Lie algebra and in general $\mathfrak{k} \not= \ov{\mathfrak{k}}_\WW \oplus \un{\mathfrak{k}}_\WW$.
In the following, we will often make use of the fact that $\ov{\mathfrak{g}}_\WW = \ov{\pp}_\WW \oplus \ov{\mathfrak{k}}_\WW$ is a Lie algebra and that moreover
\begin{equation} \label{eq:ovun}
 [\ov{\pp}_\WW, \un{\pp}_\WW], [\ov{\mathfrak{k}}_\WW, \un{\mathfrak{k}}_\WW] \subset \un{\mathfrak{k}}_\WW, \qquad \text{and} \qquad [\un{\pp}_\WW, \ov{\mathfrak{k}}_\WW], [\ov{\pp}_\WW, \un{\mathfrak{k}}_\WW] \subset \un{\mathfrak{\pp}}_\WW.
\end{equation}

From now on, we will leave out the index $\WW$.
Recall the orthonormal systems $k_1, \ldots, k_{n-r} \in \mathfrak{k}$ and $p_1, \ldots, p_{n-r} \in \mathfrak{p}$ from subsection \ref{subsec:infstruc}.
They split into systems $\{ k_{\ov{i}} \}, \{ p_{\ov{i}} \}$ and $\{ k_{\un{i}} \}, \{ p_{\un{i}} \}$ corresponding to roots $\alpha_{\ov{i}} \in \ov{\Delta}^+$ and $\alpha_{\un{i}} \in \un{\Delta}^+$, respectively.
In the following, we will denote by $e_1, \ldots, e_n \in \pp$ an arbitrary orthonormal basis of $\pp$ that obeys the splitting $\pp = \ov{\pp} \oplus \un{\pp}$, i.e. the index set $\{ i \}$ splits into $\{ \ov{i} \}$ and $\{ \un{i} \}$ such that $\{ e_{\ov{i}} \}$ is an orthonormal basis for $\ov{\pp}$ and $\{ e_{\un{i}} \}$ one for $\un{\pp}$.
Note that $e_1, \ldots, e_n$ does not need to contain the vectors $p_1, \ldots, p_{n-r}$.

Let $E = \Sym_2 T^*$ be the vector bundle of symmetric bilinear forms.
We will identify $T^* \cong T$.
At the basepoint $p_0 \in M$, we can identify $T \cong \pp$ and hence $E_0 = \Sym_2 \pp$.
The splitting $\pp = \ov{\pp} \oplus \un{\pp}$ induces a splitting $\Sym_2 \pp = \Sym_2 \ov{\pp} \oplus \Sym_2 \un{\pp} \oplus \ov{\pp} \un{\pp}$.
Observe that this splitting comes from a splitting 
\begin{equation} \label{eq:splittingE}
 E_\WW = E_{\Sym_2 \ov{\pp}} \oplus E_{\Sym_2 \un{\pp}} \oplus E_{\ov{\pp} \un{\pp}}
\end{equation}
over the whole space $\ov{M}_\WW$.
We will denote the elements of $\Sym_2 \pp$ by $v \cdot w = w \cdot v$ for $v, w \in \pp$ and set $\langle v \cdot w, v' \cdot w' \rangle = \frac12 \langle v, v' \rangle \langle w, w' \rangle + \frac12 \langle v, w' \rangle \langle w, v' \rangle$.
Hence $\{ \sqrt{2} e_i \cdot e_j, \; e_k \cdot e_k \; : \; i < j \}$ is an orthonormal basis for $\Sym_2 \pp$.

From (\ref{eq:curvature}), we obtain that $\Rm (v \cdot w) = \sum_l e_l \cdot R(e_l, v)w = - \sum_l e_l \cdot [[ e_l, v], w]$.
Hence using (\ref{eq:SWW}), we can compute that for $v \cdot w \in \Sym_2 \pp$
\begin{multline*}
 S_\WW(v \cdot w) = - \sum_{\un{m}} \big( [ k_{\un{m}}, [ k_{\un{m}}, v ]] \cdot w + v \cdot [ k_{\un{m}}, [k_{\un{m}}, w]] + 2 [k_{\un{m}}, v] \cdot [k_{\un{m}}, w] \big) \\
 + 2 \sum_l e_l \cdot [[ e_l, v], w].
\end{multline*}
Pairing this with $v' \cdot w' \in \Sym_2 \pp$ yields
\begin{alignat}{2}
2 \langle S_\WW (v \cdot w), v' \cdot w' \rangle &= \sum_{\un{m}} \Big( \big\langle [k_{\un{m}}, v ], [k_{\un{m}}, v'] \big\rangle \langle w, w' \rangle + \big\langle[ k_{\un{m}}, v ], [ k_{\un{m}}, w'] \big\rangle \langle w, v' \rangle \notag \\
&\qquad\; + \langle v, v' \rangle \big\langle[ k_{\un{m}}, w ], [ k_{\un{m}}, w'] \big\rangle + \langle v, w' \rangle \big\langle[ k_{\un{m}}, w ], [ k_{\un{m}}, v'] \big\rangle \notag \\
&\qquad\; - 2 \big\langle [ k_{\un{m}}, v], v' \big\rangle \big\langle [ k_{\un{m}}, w], w' \big\rangle
- 2 \big\langle [ k_{\un{m}}, v], w' \big\rangle \big\langle [ k_{\un{m}}, w], v' \big\rangle \Big) \notag \\
 &\quad + 2 \big\langle [[v', v], w], w' \big\rangle + 2 \big\langle [[ w', v], w], v' \big\rangle  \label{eq:SWWincoo}
 \end{alignat}
 
\subsection{A Bochner formula} \label{subsec:lambdaWWBochner}
We will now derive a Bochner formula for $L$ on $M$ and hence get a lower bound for $\lambda_{\{ 0 \}} = \lambda_B$.
In the higher rank case, this Bochner formula will be applied to cross-sections $\ov{M}_\WW$ of $M$ in subsection \ref{subsec:Sym2ovp}.
So in order to allow for this further application, we will not require $M$ to be Einstein in this subsection.

Recall the definition of the divergence operator
\[ \DIV : C^\infty(M; \Sym_2 T^*) \longrightarrow C^\infty(M;T^*), \quad h_{ij} \longmapsto - \sum_i \nabla_i h_{ij} \]
and define the exterior derivative with coefficients in $T^*$
\[ d : C^\infty(M; \Sym_2 T^*) \longrightarrow C^\infty(M; \Lambda_2 T^* \otimes T^*), \quad h_{ij} \longmapsto \nabla_i h_{jk} - \nabla_j h_{ik}. \]
Their formal adjoints are
\[  \DIV^* : C^\infty(M; T^*) \longrightarrow C^\infty(M; \Sym_2 T^*), \quad \gamma_i \longmapsto \tfrac12 ( \nabla_i \gamma_j + \nabla_j \gamma_i ) \]
and
\[ d^* : C^\infty(M; \Lambda_2 T^* \otimes T^*) \longrightarrow C^\infty(M ; \Sym_2 T^*), \quad \gamma_{ijk} \longmapsto - \tfrac12( \nabla_k \gamma_{kij} + \nabla_k \gamma_{kji} ). \]
We can then calculate that
\begin{equation} \label{eq:Bochnerdivd}
 (Lh)_{ij} = (\DIV^* \DIV + d^* d)h_{ij} - \Rm (h)_{ij} - \tfrac12 \sum_k \big( \Ric_{ik} h_{kj} + h_{ik} \Ric_{kj} \big).
\end{equation}
In the following Lemma we will show that the zero order term is nonnegative and most often even positive definite.
Hence, by choosing $D = \DIV + d : C^\infty(M;E) \to C^\infty(M; T^* \oplus \Lambda_2 T^* \otimes T^*)$, we conclude that $\lambda_{\{0\}} = \lambda_B$ is nonnegative resp. positive (compare with (\ref{eq:BochnerlambdaB}) and (\ref{eq:BochnetriangleWW})). 

\begin{Lemma} \label{Lem:Bochnerformula}
Let $\Rm$ and $\Ric$ be the Riemannian and Ricci curvature at a point of a symmetric space $M$ of noncompact type and let $T$ be the tangent space at that point.
Then the operator
\[ A: \Sym_2 T^* \longrightarrow \Sym_2 T^*, \qquad h_{ij} \mapsto - \Rm (h)_{ij} - \tfrac12 \sum_k \big( \Ric_{ik} h_{kj} + h_{jk} \Ric_{ki} \big) \]
is self-adjoint and nonnegative definite. \\
Moreover, if we consider the splitting of the tangent space $T = T_1 \oplus \ldots \oplus T_m$ associated to the de Rham decomposition $M = M_1 \times \ldots \times M_m$ and assume that $M_1, \ldots, M_{m'}$ are the only $\IH^2$-factors, then the nullspace of $A$ is
\[ \big\{ h = h_1 + \ldots + h_{m'} \;\; : \;\; h_i \in \Sym_2 T_i^*, \; \tr h_i = 0 \big\}. \]
Hence, if $M$ does not contain any $\IH^2$-factor, then $A$ is positive definite.
\end{Lemma}

\begin{proof}
Assume first that $M = M' \times M''$ is reducible and let $T = T' \oplus T''$ be the corresponding splitting.
Choose an orthonormal basis $e_1, \ldots, e_n$ of $T$ which obeys this splitting.
We first show that $A$ preserves the induced splitting $\Sym_2 T^* = \Sym_2 (T')^* \oplus \Sym_2 (T'')^* \oplus (T')^* (T'')^*$:
If $h \in \Sym_2 (T')^*$, then $\Rm(h)_{ij} = \sum_{s,t} R_{istj} h_{st}$ is only nonzero if $e_i, e_j \in T'$, since $R_{istj}$ is only nonzero if either all indices $i, s, t, j$ belong to $T'$ or to $T''$.
Furthermore since $\Ric_T = \Ric_{T'} + \Ric_{T''}$, we see that $\sum_k \Ric_{ik} h_{kj}$ is only nonzero if $e_i, e_j \in T'$.
So $A(h) \in \Sym_2 (T')^*$.
Analogously, we see that $A$ maps $\Sym_2 (T'')^*$ into itself and by self-adjointness it also has to preserve $(T')^* (T'')^*$.

Next, we show that $A$ is positive definite on $(T')^*(T'')^*$.
Let $h \in (T')^* (T'')^*$ and observe that $\Rm (h) = 0$.
Then $\langle A(h), h \rangle = - \sum_{i,j,k} \Ric_{ik} h_{kj} h_{ij} > 0$ if $h \not= 0$.
So, we can restrict our proof to the case in which $M$ is irreducible. 

Let $h \in \Sym_2 T^*$ and choose an orthonormal basis $e_1, \ldots, e_n$ for which $h$ is diagonal, i.e. $h = \sum_{i=1}^n \lambda_i e_i^* \otimes e_i^*$.
Observe, that since $M$ is of noncompact type, the sectional curvatures $K_{ij} = \langle R(e_i, e_j) e_j, e_i \rangle \leq 0$.
We can compute that
\[ \big\langle \Rm (h), h \big\rangle = \sum_{i,i',j,j'} R_{ijj'i'} h_{ii'} h_{jj'} = \sum_{i,j} K_{ij} \lambda_i \lambda_j. \]
Hence, since $\Ric_{ii} = \sum_j K_{ij}$
\[ \big\langle A(h), h \big\rangle = - \Big( \sum_{i,j} K_{ij} \lambda_i \lambda_j + \tfrac12 K_{ij} \lambda_j^2 + \tfrac12 K_{ij} \lambda_i^2 \Big) = - \sum_{i,j} \tfrac12 K_{ij} (\lambda_i + \lambda_j)^2 \geq 0. \]
Assume now that $h$ lies in the nullspace.
Then we must have $K_{ij} = 0$ whenever $\lambda_i \not= - \lambda_j$.
We can split $T = T'_1 \oplus \ldots \oplus T'_{m''}$ such that $T'_k$ is spanned by all $e_k$ for which $|\lambda_k|$ is a given constant.
Then for $e_i$ and $e_j$ belonging to different $T'_k$, we have $K_{ij} = 0$ and since the Riemannian curvature even has the property of having nonnegative curvature \emph{operator}, we conclude that all sectional curvature between different $T'_k$ vanish and hence $T'_1 \oplus \ldots \oplus T'_{m''}$ corresponds to a geometric splitting $M = M_1 \times \ldots \times M_{m''}$.
Since we assumed $M$ to be irreducible, we find that $m'' = 1$ and hence $h$ only has two eigenvalues $-\lambda$ and $\lambda$.

Let $T = T_- \oplus T_+$ be the orthogonal splitting of eigenspaces of $h$.
As before, we conclude that the sectional curvatures on $T_-$ and $T_+$ all vanish and hence those subspaces correspond to flats in $M$ or abelian subspace $\af_-, \af_+ \subset \pp$.
Now since the rank $r$ of $M$ is equal to the number of simple roots and hence is not larger than $n-r$, we conclude $2r \leq n$.
But since $\dim \af_- + \dim \af_+ = n$, we must have $2r = n$ and $\dim \af_- = \dim \af_+ = r$.
In this case, all positive roots of $M$ have to be simple.
So for any two distinct positive roots $\alpha_i, \alpha_j \in \Delta^+$ we have
\begin{multline*}
 \langle \alpha^\#_i, \alpha^\#_j \rangle = \big\langle [x_i, y_i], [x_j, y_j] \big\rangle = - \big\langle [x_i, [x_j, y_j]], y_i \big\rangle \\
 = \big\langle [x_j, [y_j, x_i]], y_i \big\rangle + \big\langle [y_j, [x_i, x_j]], y_i \big\rangle 
 = \big\langle [x_i, y_j], [x_j, y_i] \big\rangle + \big\langle [x_i, x_j], [y_i, y_j] \big\rangle = 0,
\end{multline*}
since $\alpha_i - \alpha_j$, $\alpha_j - \alpha_i$ and $\alpha_i + \alpha_j$ cannot be roots.
Hence all roots are orthogonal to each other.
It follows that $M = \IH^2 \times \ldots \times \IH^2$ and by irreducibility $M = \IH^2$.
\end{proof}

\subsection{Block form of the Einstein operator} \label{subsec:lambdaWWblockform}
From now on assume that $M$ is Einstein.
The splitting $\pp = \ov{\pp} \oplus \un{\pp}$ induces a splitting $\Sym_2 \pp = \Sym_2 \ov{\pp} \oplus \Sym_2 \un{\pp} \oplus \ov{\pp} \un{\pp}$.
We will find that $S_\WW : \Sym_2 \pp \to \Sym_2 \pp$ preserves this splitting, i.e. the differential operator $- \ov{\triangle} + S_\WW$ acting on $C^\infty(\ov{M}_\WW; E_\WW)$ preserves the splitting (\ref{eq:splittingE}) if we view it as a bundle endomorphism.
Hence we can analyze the operator $- \ov{\triangle} + S_\WW$ on each component separately.
This will be carried out in the following subsections.

\begin{Lemma} \label{Lem:simpleovpunp}
The map $\pp \to \pp$, $v \mapsto - \sum_{\un{m}} [k_{\un{m}}, [k_{\un{m}}, v]]$ is self-adjoint and preserves the splitting $\pp = \ov{\pp} \oplus \un{\pp}$. \\
In particular, $\sum_{\un{m}} \big\langle [k_{\un{m}}, v], [k_{\un{m}}, w] \big\rangle = 0$ for any $v \in \ov{\pp}$ and $w \in \un{\pp}$.
\end{Lemma}
\begin{proof}
Let $e_1, \ldots, e_n \in \pp$ be an orthonormal basis which respects the splitting $\pp = \ov{\pp} \oplus \un{\pp}$.
We find for $v \in \ov{\pp}$ and $w \in \un{\pp}$
\begin{multline}
 - \sum_{\un{m}} \big\langle [k_{\un{m}}, [k_{\un{m}}, v ]], w \big\rangle = \sum_{\un{m}} \big\langle [k_{\un{m}}, v], [k_{\un{m}}, w] \big\rangle  \\
 = \sum_{\un{m}, l} \big\langle [k_{\un{m}}, v], e_{l} \big\rangle \big\langle [k_{\un{m}}, w], e_{l} \big\rangle
 = \sum_{\un{m}, l} \big\langle [v, e_l], k_{\un{m}} \big\rangle \big\langle [w, e_l], k_{\un{m}} \big\rangle \label{eq:kmvkmw}
\end{multline}
Now recall from (\ref{eq:ovun}) that if the index $l$ is of type $\ov{l}$, we have $[w, e_{\ov{l}}] \in \un{\mathfrak{k}}$ and if it is of type $\un{l}$, we have $[v, e_{\un{l}}] \in \un{\mathfrak{k}}$.
So one of the two expressions $[v, e_l]$ and $[w, e_l]$ is always contained in $\un{\mathfrak{k}}$ and we conclude that (\ref{eq:kmvkmw}) is equal to
\[ - \sum_l \big\langle [v, e_l], [w, e_l] \big\rangle = - \Ric(v,w) = - \frac{R}n \langle v, w \rangle = 0. \qedhere \]
\end{proof}

\begin{Lemma} \label{Lem:partovpunp}
For every wall $\WW \subset \CC$, the splitting $\pp = \ov{\pp} \oplus \un{\pp}$ induces a splitting $\Sym_2 \pp = \Sym_2 \ov{\pp} \oplus \Sym_2 \un{\pp} \oplus \ov{\pp} \un{\pp}$.
The operator $S_\WW : \Sym_2 \pp \to \Sym_2 \pp$ is self-adjoint and preserves this splitting.
\end{Lemma}
\begin{proof}
Let $v, w, v', w' \in \pp$.
We conclude from (\ref{eq:SWWincoo}) that
\begin{alignat*}{1}
2 \langle S_\WW (v \cdot w), v' \cdot w' \rangle
 = &\sum_{\un{m}} \Big( \big\langle[k_{\un{m}}, v ], [k_{\un{m}}, v'] \big\rangle \langle w, w' \rangle + \big\langle[ k_{\un{m}}, v ], [ k_{\un{m}}, w'] \big\rangle \langle w, v' \rangle \\
& \qquad + \langle v, v' \rangle \big\langle[ k_{\un{m}}, w ], [ k_{\un{m}}, w'] \big\rangle + \langle v, w' \rangle \big\langle[ k_{\un{m}}, w ], [ k_{\un{m}}, v'] \big\rangle  \Big) \\
 - 2 & \sum_{\un{m}} \big\langle [v,v'], k_{\un{m}} \big\rangle \big\langle [w, w'], k_{\un{m}} \big\rangle - 2\big\langle [v, v'], [w, w'] \big\rangle \\
- 2 & \sum_{\un{m}} \big\langle [ v, w'], k_{\un{m}} \big\rangle \big\langle [ w, v'], k_{\un{m}} \big\rangle
- 2\big\langle [w, v'], [v, w'] \big\rangle .
 \end{alignat*}
The fact that $S_\WW$ is self-adjoint follows directly from this expression.
Observe that the last two lines in this formula are equal to
\begin{equation} \label{eq:frakkov}
 -2 \big\langle \proj_{\un{\mathfrak{k}}^\perp} [v,v'], [w,w'] \big\rangle - 2 \big\langle \proj_{\un{\mathfrak{k}}^\perp} [w,v'], [v,w'] \big\rangle,
\end{equation}
where $\un{\mathfrak{k}}^\perp$ is the orthogonal complement of $\un{\mathfrak{k}}$ in $\mathfrak{k}$.

Now assume that $v, w \in \ov{\pp}$ and $v', w' \in \un{\pp}$, i.e. $v \cdot w \in \Sym_2 \ov{\pp}$ and $v' \cdot w' \in \Sym_2 \un{\pp}$.
Then $[v,v'], [w,v'] \in \un{\mathfrak{k}}$, so expression (\ref{eq:frakkov}) vanishes.
Since $\langle w, w' \rangle = \langle w, v' \rangle = \langle v, v' \rangle = \langle v, w' \rangle = 0$, we conclude $\langle S_\WW (v \cdot w), v' \cdot w' \rangle = 0$.

Secondly, assume that $v, w, v' \in \ov{\pp}$ and $w' \in \un{\pp}$, i.e. $v \cdot w \in \Sym_2 \ov{\pp}$ and $v' \cdot w' \in \ov{\pp}\un{\pp}$.
Then $[w, w'], [v, w'] \in \un{\mathfrak{k}}$ and (\ref{eq:frakkov}) vanishes again.
Moreover, $\langle w, w' \rangle = \langle v, w' \rangle = 0$ and by Lemma \ref{Lem:simpleovpunp} we conclude $\sum_{\un{m}} \langle [k_{\un{m}}, v], [k_{\un{m}}, w'] \rangle = \sum_{\un{m}} \langle [k_{\un{m}}, w], [k_{\un{m}}, w'] \rangle = 0$.
So $\langle S_\WW (v \cdot w), v' \cdot w' \rangle = 0$.

Finally, assume that $v \in \ov{\pp}$ and $w, v', w' \in \un{\pp}$, i.e. $v \cdot w \in \ov{\pp} \un{\pp}$ and $v' \cdot w' \in \Sym_2 \un{\pp}$.
Then (\ref{eq:frakkov}) vanishes again since $[v, v'], [v, w'] \in \un{\mathfrak{k}}$ and by Lemma \ref{Lem:simpleovpunp} as well as $\langle v, v' \rangle = \langle v, w' \rangle = 0$, we conclude $\langle S_\WW (v \cdot w), v' \cdot w' \rangle = 0$.
\end{proof}

\subsection{The Einstein operator on $\overline{\mathfrak{p}} \text{\underline{$\mathfrak{p}$}}$} \label{subsec:ovpunp}
We now analyze the operator $- \ov{\triangle} + S_\WW$ on $E_{\ov{\pp}\un{\pp}}$.
Here, we will only need the trivial Bochner formula $- \ov{\triangle} = \nabla^* \nabla$.
Hence it suffices to analyze $S_\WW$ acting on $\ov{\pp} \un{\pp}$.

\begin{Lemma} \label{Lem:ovpunp}
For every wall $\WW \subset \CC$, the restricted operator $S_\WW : \ov{\pp} \un{\pp} \to \ov{\pp} \un{\pp}$ is positive definite.
\end{Lemma}
\begin{proof}
Let $v, v' \in \ov{\pp}$ and $w, w' \in \un{\pp}$.
Using the calculations from the proof of Lemma \ref{Lem:partovpunp}, the fact that $[v,v'] \in \ov{\mathfrak{k}}$ and $[w,v'] \in \un{\mathfrak{k}}$ and hence $\langle \proj_{\ov{\mathfrak{k}}} [w, v'], [v, w'] \rangle = 0$, we conclude
\begin{multline*}
2 \langle S_\WW (v \cdot w), v' \cdot w' \rangle
 = \sum_{\un{m}} \Big( \big\langle[k_{\un{m}}, v ], [k_{\un{m}}, v'] \big\rangle \langle w, w' \rangle 
+ \langle v, v' \rangle \big\langle[ k_{\un{m}}, w ], [ k_{\un{m}}, w'] \big\rangle  \Big) \\ 
- 2\big\langle [v, v'], [w, w'] \big\rangle .
 \end{multline*}
Let $e_1, \ldots, e_n$ be an orthonormal basis of $\pp$ which respects the splitting $\pp = \ov{\pp} \oplus \un{\pp}$ and express $h = \sum_{\ov{i}, \un{j}} h_{\ov{i} \un{j}} e_{\ov{i}} \cdot e_{\un{j}} \in \ov{\pp} \un{\pp}$.
Then summing over all free indices yields
\begin{multline*}
2 \langle S_\WW h, h \rangle = h_{\ov{i} \un{j}} h_{\ov{i}' \un{j}} \big\langle [k_{\un{m}}, e_{\ov{i}}], [k_{\un{m}}, e_{\ov{i}'}] \big\rangle + h_{\ov{i} \un{j}} h_{\ov{i} \un{j}'} \big\langle [k_{\un{m}}, e_{\un{j}}], [k_{\un{m}}, e_{\un{j}'} ] \big\rangle \\
- 2 h_{\ov{i} \un{j}} h_{\ov{i}' \un{j}'} \big\langle [e_{\ov{i}}, e_{\ov{i}'} ], [e_{\un{j}}, e_{\un{j}'}] \big\rangle .
\end{multline*}
We can rewrite the last coefficient as
\begin{alignat*}{1}
 \big\langle [e_{\ov{i}}, e_{\ov{i}'}], [e_{\un{j}}, e_{\un{j}'}] \big\rangle 
 &= - \big\langle [e_{\ov{i}}, [ e_{\un{j}}, e_{\un{j}'} ]], e_{\ov{i}'} \big\rangle
 = - \big\langle [e_{\un{j}}, [ e_{\ov{i}}, e_{\un{j}'}]], e_{\ov{i}'} \big\rangle + \big\langle [ e_{\un{j}'}, [ e_{\ov{i}}, e_{\un{j}} ]], e_{\ov{i}'} \big\rangle \displaybreak[0] \\
 &\hspace{-5mm} = \big\langle [ e_{\ov{i}}, e_{\un{j}'}], [e_{\un{j}}, e_{\ov{i}'}] \big\rangle + \big\langle [e_{\ov{i}}, e_{\un{j}} ], [e_{\ov{i}'}, e_{\un{j}'}] \big\rangle \displaybreak[0] \\
 &\hspace{-5mm}= - \sum_{\un{m}} \big\langle [e_{\ov{i}}, e_{\un{j}'}], k_{\un{m}} \big\rangle \big\langle [e_{\un{j}}, e_{\ov{i}'} ], k_{\un{m}} \big\rangle + \big\langle [e_{\ov{i}}, e_{\un{j}} ], [e_{\ov{i}'}, e_{\un{j}'}] \big\rangle \displaybreak[0] \\
 &\hspace{-5mm}=  \tfrac12 \sum_{\un{m}} \big\langle [k_{\un{m}}, e_{\ov{i}}], e_{\un{j}'} \big\rangle \big\langle [k_{\un{m}}, e_{\ov{i}'}], e_{\un{j}} \big\rangle + \tfrac12 \sum_{\un{m}} \big\langle [k_{\un{m}}, e_{\un{j}}], e_{\ov{i}'} \big\rangle \big\langle [k_{\un{m}}, e_{\un{j}'}], e_{\ov{i}}\big\rangle \\
 &\qquad\qquad\qquad\qquad\qquad\qquad\qquad\qquad\qquad+ \big\langle [e_{\ov{i}}, e_{\un{j}} ], [e_{\ov{i}'}, e_{\un{j}'}] \big\rangle
\end{alignat*}
Hence
\begin{multline*}
2 \langle S_\WW h, h \rangle = h_{\ov{i} \un{j}} h_{\ov{i}' \un{j}} \big\langle [k_{\un{m}}, e_{\ov{i}}], e_{\ov{j}'} \big\rangle \big\langle [k_{\un{m}}, e_{\ov{i}'}], e_{\ov{j}'} \big\rangle +  h_{\ov{i} \un{j}} h_{\ov{i}' \un{j}} \big\langle [k_{\un{m}}, e_{\ov{i}}], e_{\un{j}'} \big\rangle \big\langle [k_{\un{m}}, e_{\ov{i}'}], e_{\un{j}'} \big\rangle \\
+ h_{\ov{i} \un{j}} h_{\ov{i} \un{j}'} \big\langle [k_{\un{m}}, e_{\un{j}}], e_{\ov{i}'} \big\rangle \big\langle [k_{\un{m}}, e_{\un{j}'}], e_{\ov{i}'} \big\rangle + h_{\ov{i} \un{j}} h_{\ov{i} \un{j}'} \big\langle [k_{\un{m}}, e_{\un{j}}], e_{\un{i}'} \big\rangle \big\langle [k_{\un{m}}, e_{\un{j}'}], e_{\un{i}'} \big\rangle \\
- h_{\ov{i} \un{j}} h_{\ov{i}' \un{j}'} \big\langle [k_{\un{m}}, e_{\ov{i}}], e_{\un{j}'} \big\rangle \big\langle [k_{\un{m}}, e_{\ov{i}'}], e_{\un{j}} \big\rangle
- h_{\ov{i} \un{j}} h_{\ov{i}' \un{j}'} \big\langle [k_{\un{m}}, e_{\un{j}}], e_{\ov{i}'} \big\rangle \big\langle [k_{\un{m}}, e_{\un{j}'}], e_{\ov{i}}\big\rangle \\
- 2 h_{\ov{i} \un{j}} h_{\ov{i}' \un{j}'} \big\langle [e_{\ov{i}}, e_{\un{j}} ], [e_{\ov{i}'}, e_{\un{j}'}] \big\rangle
\end{multline*}
Since $[k_{\un{m}}, e_{\ov{i}}] \in \un{\pp}$, the first term vanishes and we can regroup the expression as follows:
\begin{multline*}
= \tfrac12 \sum_{\un{m}, \un{j}, \un{j}'}  \Big( \sum_{\ov{i}} h_{\ov{i} \un{j}} \big\langle [k_{\un{m}}, e_{\ov{i}}], e_{\un{j}'} \big\rangle - \sum_{\ov{i}'} h_{\ov{i}' \un{j}'} \big\langle [k_{\un{m}}, e_{\ov{i}'}], e_{\un{j}} \big\rangle \Big)^2 \\
+ \tfrac12 \sum_{\un{m}, \ov{i}, \ov{i}'} \Big( \sum_{\un{j}} h_{\ov{i} \un{j}} \big\langle [k_{\un{m}}, e_{\un{j}}], e_{\ov{i}'} \big\rangle - \sum_{\un{j}'} h_{\ov{i}' \un{j}'} \big\langle [k_{\un{m}}, e_{\un{j}'}], e_{\ov{i}} \big\rangle \Big)^2 \\
+ \sum_{\ov{i}, \un{i}', \un{m}} \Big( \sum_{\un{j}} h_{\ov{i}  \un{j}} \langle [k_{\un{m}}, e_{\un{j}}], e_{\un{i}'} \rangle \Big)^2
- 2 \Big| \sum_{\ov{i}, \un{j}} h_{\ov{i} \un{j}} [e_{\ov{i}}, e_{\un{j}}] \Big|^2 \geq 0
\end{multline*}
Observe that this expression is nonnegative since the Killing form is negative definite on $\mathfrak{k}$.
If the expression is zero, then all the squared terms have to vanish, in particular
\[ 0 = \langle [k_{\un{m}}, H_{\ov{i}}], e_{\un{i}'} \rangle = \langle [k_{\un{m}}, e_{\un{i}'}], H_{\ov{i}} \rangle \]
for all $\ov{i}, \ov{i}'$  and $\un{m}$ where $H_{\ov{i}} = \sum_{\un{j}} h_{\ov{i} \un{j}} e_{\un{j}} \in \un{\pp}$.
By the argument in the next paragraph we have $\un{\pp} \subset [ \un{\mathfrak{k}}, \un{\pp} ]$.
This implies that $H_{\ov{i}} = 0$ for all $\ov{i}$ and hence $h = 0$, establishing the Lemma.

It remains to prove $\un{\pp} \subset [ \un{\mathfrak{k}}, \un{\pp} ]$.
First observe that by (\ref{eq:xiyi}) we have $[k_{\un{m}}, p_{\un{m}}] = \alpha^\#_{\un{m}}$ and hence $\un{\af} \subset [\un{\mathfrak{k}}, \un{\pp}]$.
Moreover, for any index $\un{m}$ there is a $v \in \un{\af}$ with $\alpha_{\un{m}}(v) \not= 0$ and we have $[k_{\un{m}}, v] = - \alpha_{\un{m}} (v) p_{\un{m}}$, so $p_{\un{m}} \in [\un{\mathfrak{k}}, \un{\pp}]$ what establishes the claim.
\end{proof}

\subsection{The Einstein operator on $\Sym_2 \ov{\pp}$} \label{subsec:Sym2ovp}
We will now analyze the operator on $-\ov{\triangle} + S_\WW$ on $E_{\Sym_2 \ov{\pp}}$.
Let $\ov{\Rm}$ be the Riemannian curvature operating on $\ov{M}_\WW$ acting on symmetric bilinear forms $h \in \Sym_2 \ov{T}^* \cong E_{\Sym_2 \ov{\pp}}$.
We now make use of the same Bochner formula as in (\ref{eq:Bochnerdivd}), but this time on $\ov{M}_\WW$, to conclude that
\begin{multline*}
 (-\ov{\triangle} h + S_\WW (h))_{ij} = (\ov{\DIV}^* \ov{\DIV} + \ov{d}^* \ov{d}) h_{ij} \\  + \ov{\Rm}(h)_{ij}
 - \tfrac12 \sum_k \big( \ov{\Ric}_{ik} h_{kj} + h_{ik} \ov{\Ric}_{kj} \big) + S_\WW (h)_{ij}.
\end{multline*}
Here $\ov{\DIV}$, $\ov{d}$ and $\ov{\Ric}$ denote the corresponding operators and tensors on $\ov{M}_\WW$.
It remains to analyze the last line.
This analysis can be carried out on $\Sym_2 \ov{\pp}$.

\begin{Lemma} \label{Lem:Sym2ovp}
The operator $B : \Sym_2 \ov{p} \to \Sym_2 \ov{p}$,
\[ h \mapsto \ov{\Rm}(h)_{ij} - \tfrac12 \sum_k \big(  \ov{\Ric}_{ik} h_{kj} + h_{ik} \ov{\Ric}_{kj} \big) + S_\WW(h)_{ij} \]
is nonnegative definite.
If $M$ does not contain any $\IH^2$-factor in its deRham decomposition, then $B$ is even positive definite.
\end{Lemma}
\begin{proof}
For every $h \in \Sym_2 \ov{\pp}$, we have
\[ \big\langle B (h), h \big\rangle = \big\langle \ov{A} (h), h \big\rangle + \big\langle S_\WW(h) + 2 \ov{\Rm}(h), h \big\rangle, \]
where $\ov{A} : \Sym_2 \ov{\pp} \to \Sym_2 \ov{\pp}$ is the expression from Lemma \ref{Lem:Bochnerformula} on $\ov{M}_\WW$.
By the same Lemma we know that $\ov{A}$ is nonnegative definite and we will now show that $S_\WW + 2 \ov{\Rm}$ is nonnegative definite as well.

Let $v, w, v', w' \in \ov{\pp}$.
Using the calculation from the proof of Lemma \ref{Lem:partovpunp}, we conclude
\begin{multline*}
2 \langle S_\WW (v \cdot w), v' \cdot w' \rangle
 = \sum_{\un{m}} \Big( \big\langle[k_{\un{m}}, v ], [k_{\un{m}}, v'] \big\rangle \langle w, w' \rangle + \big\langle[ k_{\un{m}}, v ], [ k_{\un{m}}, w'] \big\rangle \langle w, v' \rangle \\
+ \langle v, v' \rangle \big\langle[ k_{\un{m}}, w ], [ k_{\un{m}}, w'] \big\rangle + \langle v, w' \rangle \big\langle[ k_{\un{m}}, w ], [ k_{\un{m}}, v'] \big\rangle  \Big)
- 4 \big\langle \ov{\Rm}(v \cdot w), v' \cdot w' \big\rangle
\end{multline*}
Hence, for any $h = \sum_{\ov{i}, \ov{j}} h_{\ov{i} \ov{j}} e_{\ov{i}} \cdot e_{\ov{j}}$ with $h_{\ov{i} \ov{j}} = h_{\ov{j} \ov{i}}$ we find
\begin{multline*}
\big\langle S_\WW (h) + 2 \ov{\Rm}(h), h \big\rangle = 2 \sum_{\ov{i}, \ov{i}', \ov{j}} h_{\ov{i} \ov{j}} h_{\ov{i}' \ov{j}} \big\langle [k_{\un{m}}, e_{\ov{i}}], [k_{\un{m}}, e_{\ov{i}'}] \big\rangle
= 2 \sum_{\ov{j}, \un{m}} \Big| \Big[ k_{\un{m}}, \sum_{\ov{i}} h_{\ov{i} \ov{j}} e_{\ov{i}} \Big] \Big|^2
\end{multline*}
This establishes the nonnegativity of $B$.

Assume now that $h$ lies in the nullspace of $B$.
Hence, it lies in the nullspace of $A$ and for all $\ov{j}$ and $\un{m}$ we have $[k_{\un{m}}, \sum_{\ov{i}} h_{\ov{i} \ov{j}} e_{\ov{i}} ] = 0$.
By Lemma \ref{Lem:Bochnerformula}, we have $h = h_1 + \ldots + h_{m'}$ corresponding to a splitting $\ov{M}_\WW = \IH^2 \times \ldots \times \IH^2 \times M_{m'+1} \times \ldots \times M_m$ and each $h_k$ is traceless.
Without loss of generality, we can assume that all the $h_k$ are nonzero and hence the vectors $\sum_{\ov{i}} h_{\ov{i} \ov{j}} e_{\ov{i}}$ span a subspace $\ov{\pp}' \subset \ov{\pp}$ that corresponds to the tangent space of the $\IH^2 \times \ldots \times \IH^2$ factor.
We have $[ \un{\mathfrak{k}}, \ov{\pp}'] = 0$ and for every $e \in \ov{\pp'}$ pointing in the direction of one of the $\IH^2$-factors, we have $[k_{\ov{i}}, e] = 0$ for all but one $\ov{i}$, which corresponds to this $\IH^2$-factor.
This implies that this $\IH^2$-factor is already an $\IH^2$-factor of $M$, contradicting the assumption of the second part of the Lemma.
\end{proof}

\subsection{The Einstein operator on $\Sym_2\text{\un{$\pp$}}$---Parts involving $\text{\un{$\af$}}$} \label{subsec:af}
In the following three sections, we will analyze the operator $-\ov{\triangle} + S_\WW$ on $E_{\un{\pp}}$.
We will use the trivial Bochner formula $-\ov{\triangle} = \nabla^* \nabla$ on $M_\WW$ and we will show that $S_\WW$ is nonnegative definite on $\Sym_2 \un{\pp}$ and we will characterize the nullspace.

First observe that we have the splitting $\un{\pp} = \un{\af} \oplus \un{\af}^\perp$ where $\un{\af}^\perp = \sum_{\un{\alpha}} \pp_{\un{\alpha}}$, which induces a splitting $\Sym_2 \un{\pp} = \Sym_2 \un{\af} \oplus \un{\af} \cdot \un{\af}^\perp \oplus \Sym_2 \un{\af}^\perp$.

\begin{Lemma} \label{Lem:partsont}
For every wall $\WW \subset \CC$, the restricted operator $S_\WW : \Sym_2 \un{\pp} \to \Sym_2 \un{\pp}$ preserves the splitting $\Sym_2 \un{\pp} = \Sym_2 \un{\af} \oplus \un{\af} \cdot \un{\af}^\perp \oplus \Sym_2 \un{\af}^\perp$ and $S_\WW$ is positive definite on $\Sym_2 \un{\af} \oplus \un{\af} \cdot \un{\af}^\perp$.
\end{Lemma}
\begin{proof}
Let $v, w \in \af$.
We will need the following identity:
\[ \langle v, w \rangle = \tr [v,[w, \cdot]] = \sum_{l=1}^{n-r} \big\langle [v,[w,p_l]], p_l \big\rangle - \sum_{l=1}^{n-r} \big\langle [ v, [w, k_l]], k_l \big\rangle = 2 \sum_{l=1}^{n-r} \alpha_l(v) \alpha_l(w). \]
Hence $\sum_{l=1}^{n-r} \alpha_l^\# \alpha_l(v) = \frac12 v$.
Recall also that by (\ref{eq:xiyi}) $[k_l, p_l] = - [x_l, y_l] = \alpha_l^\#$.

Let now $v, w \in \un{\af}$, i.e. $v \cdot w \in \Sym_2 \un{\af}$.
Choose an orthonormal basis $e_1, \ldots, e_r$ of $\af$ and consider the orthonormal basis $e_1, \ldots, e_r, p_1, \ldots, p_{n-r}$ of $\pp$.
Then, since $\alpha_{\ov{l}}(v) = \alpha_{\ov{l}}(w) = 0$ and $[p_{\ov{l}}, v] = 0$
\begin{multline*}
 S_\WW (v \cdot w) = \sum_{\un{m}} \big( \alpha_{\un{m}} (v) [k_{\un{m}}, p_{\un{m}}] \cdot w + \alpha_{\un{m}} (w) v \cdot [k_{\un{m}}, p_{\un{m}}] - 2 \alpha_{\un{m}} (v) \alpha_{\un{m}} (w) p_{\un{m}} \cdot p_{\un{m}} \big) \\  - 2 \sum_l \alpha_l(v) p_l \cdot [k_l, w] 
= \sum_{\un{m}} \big( \alpha_{\un{m}} (v) \alpha_{\un{m}}^\# \cdot w + \alpha_{\un{m}} (w) v \cdot \alpha_{\un{m}}^\# \big) = v \cdot w.
\end{multline*}
So $S_\WW$ is positive definite on $\Sym_2 \un{\af}$.

Now assume that $v \in \un{\af}$ and $w \in \un{\af}^\perp$, i.e. $v \cdot w \in \un{\af} \cdot \un{\af}^\perp$.
Then
\begin{multline*} 
S_\WW (v \cdot w) = \sum_{\un{m}} \big( \alpha_{\un{m}} (v) [k_{\un{m}}, p_{\un{m}}] \cdot w - v \cdot [k_{\un{m}}, [k_{\un{m}}, w]] + 2 \alpha_{\un{m}} (v) p_{\un{m}} \cdot [k_{\un{m}} , w] \big) \\ - 2 \sum_l \alpha_l(v) p_l \cdot [k_l, w]
= \tfrac12 v \cdot w - \sum_{\un{m}} v \cdot [k_{\un{m}}, [k_{\un{m}}, w]]
\end{multline*}
Hence $S_\WW$ maps the space $\un{\af} \cdot \un{\af}^\perp$ to itself and it can be expressed as a tensor product of the identity on $\un{\af}$ and the map
\[ \un{\af}^\perp \to \un{\af}^\perp, \qquad w \mapsto \tfrac12 w - \sum_{\un{m}} [k_{\un{m}}, [k_{\un{m}}, w]] \]
This map is positive definite, and so is the tensor product.
\end{proof}

\subsection{The Einstein operator on $\Sym_2\text{\un{$\pp$}}$---The part $\Sym_2 \text{\un{$\af$}}^\perp$} \label{subsec:afperp}
It remains to analyze the operator $S_\WW$ on $\Sym_2 \un{\af}^\perp$.
This case is the most complicated one since we have to deal with the richer nilpotent structure on $\un{\nn}$.
We will find out that $S_\WW$ is nonnegative definite on this space and that the nullspace corresponds exactly to certain deformations of $\un{\nn}$.
In the next subsection, we will then show that in many cases such deformations do not exist and hence $S_\WW$ is positive definite.

As a first step it will be essential to express the operator $S_\WW$ in terms of the nilpotent structure on $\nn$.
In order to do this, we first need to discuss how we can recover the complete structure of the Lie-algebra $\mathfrak{g}$ from the nilponent structure on $\nn$ and the roots $\alpha_l$.
Recall that $\nn$ is spanned by the basis vectors $x_1, \ldots, x_{n-r}$ which are orthonormal with respect to the scalar product $(\cdot, \cdot) = - \langle \cdot, \sigma \cdot \rangle$.
We define the symbol $\tTT{i}{j}{k}$ by the following identity:
\[ [x_i, x_j] = \sum_k \TT{i}{j}{k} x_k. \]
Then $\tTT{i}{j}{k}$ is a $(1,2)$-tensor on $\nn$ in the indices $i$, $j$, $k$ and it is antisymmetric in $i$ and $j$.
Since $\nn$ is nilpotent, we know that $\tTT{i}{l}{l} = 0$ for all $i$ and $l$, so in particular
\begin{equation} \label{eq:TTid1} \sum_l  \TT{i}{l}{l} = \tr \TT{i}{\cdot}{\cdot} = 0. \end{equation}
Moreover, by nilpotency we know that $\tTT{i}{j}{k}$ or $\tTT{i'}{k}{j}$ cannot both be nonzero.
Hence
\begin{equation} \sum_{j, k} \TT{i}{j}{k} \TT{i'}{k}{j} = 0. \label{eq:TTid2} \end{equation}
Observe that equations (\ref{eq:TTid1}) and (\ref{eq:TTid2}) are tensorial, i.e. they also stay true if we change the orthonormal basis $x_1, \ldots, x_{n-r}$.

The symbol $\tTT{i}{j}{k}$ contains all the information on the structure of the nilpotent Lie group $\mathfrak{n}$.
Using this information and the roots $\alpha_1, \ldots, \alpha_{n-r} \in \af^*$, we will now reconstruct the structure of the Lie algebra $\mathfrak{g}$.
First, we analyze terms of the form $[x_i, y_j]$.
We have for any $l$:
\[ \big\langle [x_i, y_j], y_l \big\rangle =  \big\langle [y_j, y_l], x_i \big\rangle =  \big\langle [x_j, x_l ] , y_i \big\rangle =  -\TT{j}{l}{i}. \]
So if $\pr_{\mathfrak{n}}$ denotes the projection on $\mathfrak{n}$, we have
\[ \pr_{\mathfrak{n}}([x_i, y_j]) =   \sum_l \TT{j}{l}{i} x_l. \]
Furthermore, we can compute that
\[ \big\langle [x_i, y_j], x_l \big\rangle = - \big\langle [x_i, x_l], y_j \big\rangle = \TT{i}{l}{j}. \]
Hence
\[ \pr_{\nn^-}([x_i, y_j]) = -\sum_l \TT{i}{l}{j} y_l. \]
Finally, for any $v \in \af$,
\[ \big\langle [x_i, y_j], v \big\rangle = \big\langle [v, x_i], y_j \big\rangle = \alpha_i(v) \big\langle x_i, y_j \big\rangle = - \delta_{ij} \alpha_i (v). \]
So
\[ \pr_{\af}([x_i, y_j]) = - \delta_{ij} \alpha_i^\#. \]
We can thus write down the projection of $[x_i, y_j]$ onto the space $\mathfrak{k}_0^\perp = \af \oplus \nn \oplus \nn^-$:
\begin{multline} 
\pr_{\mathfrak{k}_0^\perp} ([x_i, y_j]) = \sum_l \TT{j}{l}{i} x_l -\sum_l \TT{i}{l}{j} y_l - \delta_{ij} \alpha_i^\# \\
= \frac1{\sqrt{2}} \sum_l \Big[ \TT{j}{l}{i} - \TT{i}{l}{j} \Big] k_l + \frac1{\sqrt{2}} \sum_l \Big[ \TT{j}{l}{i} + \TT{i}{l}{j} \Big] p_l - \delta_{ij} \alpha_i^\# \label{eq:xiyiformula}
\end{multline}

Understanding the part of $[x_i, y_j]$ in $\mathfrak{k}_0$ is more difficult, because we did not introduce an orthonormal basis on this space.
Hence, the best we can do here is to determine the scalar product of two such terms.
\begin{alignat*}{1}
\big\langle [x_i, y_j], [x_{i'}, y_{j'} ] \big\rangle &= - \big\langle [[x_i, y_j], y_{j'}], x_{i'} \big\rangle = - \big\langle [[x_i, y_{j'}], y_j ], x_{i'} \big\rangle - \big\langle [x_i, [y_j, y_{j'}]], x_{i'} \big\rangle \displaybreak[0] \\
&\hspace{-8mm} = \big\langle [x_i, y_{j'}], [x_{i'}, y_j] \big\rangle + \big\langle [x_i, x_{i'}], [y_j, y_{j'}] \big\rangle \displaybreak[0] \\
&\hspace{-8mm} = \big\langle [x_i, y_{j'}], [x_{i'}, y_j] \big\rangle - \sum_l \TT{i}{i'}{l} \TT{j}{j'}{l} \displaybreak[0] \\
&\hspace{-8mm} = \big\langle [x_i, y_{j'}], \sigma [x_{i'}, y_j] \big\rangle + 2 \big\langle [x_i, y_{j'} ] , \pr_{\mathfrak{p}}([x_{i'}, y_j]) \big\rangle - \sum_l \TT{i}{i'}{l} \TT{j}{j'}{l} \displaybreak[0] \\
&\hspace{-8mm} = - \big\langle [x_i, y_{j'}], [x_j, y_{i'}] \big\rangle + \sum_l \Big\{ \TT{j'}{l}{i} + \TT{i}{l}{j'} \Big\} \Big\{ \TT{j}{l}{i'} + \TT{i'}{l}{j} \Big\} \\ 
&\hspace{40mm} -  \sum_l \TT{i}{i'}{l} \TT{j}{j'}{l} + 2 \delta_{ij'} \delta_{i'j} \big\langle \alpha_i^\#, \alpha_{i'}^\# \big\rangle.
\end{alignat*}
We will now repeat this process twice while permuting $j \to j' \to i' \to j$.
\begin{multline*}
\big\langle [x_i, y_{j'}], [x_j, y_{i'}] \big\rangle = - \big\langle [x_i, y_{i'}], [x_{j'}, y_j] \big\rangle + \sum_l \Big\{ \TT{i'}{l}{i} + \TT{i}{l}{i'} \Big\} \Big\{ \TT{j'}{l}{j} + \TT{j}{l}{j'} \Big\} \\
- \sum_l \TT{i}{j}{l} \TT{j'}{i'}{l} + 2 \delta_{i i'} \delta_{j j'} \big\langle \alpha_i^\#, \alpha_j^\# \big\rangle,
\end{multline*}
\begin{multline*}
\big\langle [x_i, y_{i'}], [x_{j'}, y_j] \big\rangle = - \big\langle [x_i, y_j], [x_{i'}, y_{j'}] \big\rangle + \sum_l \Big\{ \TT{j}{l}{i} + \TT{i}{l}{j} \Big\} \Big\{ \TT{i'}{l}{j'} + \TT{j'}{l}{i'} \Big\} \\
- \sum_l \TT{i}{j'}{l} \TT{i'}{j}{l} + 2 \delta_{i j} \delta_{i' j'} \big\langle \alpha_i^\#, \alpha_{i'}^\# \big\rangle.
\end{multline*}
So if we add the first and third equation and subtract the second one, we obtain
\begin{multline*}
\big\langle [x_i, y_j], [x_{i'}, y_{j'}] \big\rangle = \frac12 \sum_l \left[ \Big\{ \TT{j'}{l}{i} + \TT{i}{l}{j'} \Big\} \Big\{ \TT{j}{l}{i'} + \TT{i'}{l}{j} \Big\} \right. \\ \displaybreak[1]
\left. - \Big\{ \TT{i'}{l}{i} + \TT{i}{l}{i'} \Big\} \Big\{ \TT{j'}{l}{j} + \TT{j}{l}{j'} \Big\} + \Big\{ \TT{j}{l}{i} + \TT{i}{l}{j} \Big\} \Big\{ \TT{i'}{l}{j'} + \TT{j'}{l}{i'} \Big\} \right] \\ \displaybreak[2]
- \frac12 \sum_l \left[\TT{i}{i'}{l} \TT{j}{j'}{l} + \TT{i}{j}{l} \TT{i'}{j'}{l} + \TT{i}{j'}{l} \TT{i'}{j}{l} \right] \\ \displaybreak[1]
+ \delta_{i j'} \delta_{i' j} \big\langle \alpha_i^\#, \alpha_j^\# \big\rangle - \delta_{i i'} \delta_{j j'} \big\langle \alpha_i^\#, \alpha_j^\# \big\rangle + \delta_{ij} \delta_{i' j'} \big\langle \alpha_i^\#, \alpha_{i'}^\# \big\rangle.
\end{multline*}
For any $i$, $i'$, $j$, $j'$ set
\[ A_{iji'j'} = \delta_{i j'} \delta_{i' j} \big\langle \alpha_i^\#, \alpha_j^\# \big\rangle - \delta_{i i'} \delta_{j j'} \big\langle \alpha_i^\#, \alpha_j^\# \big\rangle + \delta_{ij} \delta_{i' j'} \big\langle \alpha_i^\#, \alpha_{i'}^\# \big\rangle \]
and interpret $A_{iji'j'}$ as a $(0,4)$-tensor on $\mathfrak{n}$.

We will now calculate the curvature.
In order to simplify calculations later on, we will for the moment assume that $x_1, \ldots, x_{n-r}$ is \emph{any} orthonormal basis of $\nn$ with respect to the scalar product $( \cdot, \cdot)$, i.e. the $x_i$ do not satisfy the grading of $\nn$ anymore and there is no root associated to them.
We furthermore set $y_i = \sigma x_i$, $k_i = \frac1{\sqrt{2}} (x_i + y_i)$ and $p_i = \frac1{\sqrt{2}} (x_i - y_i)$.
Observe that we can still use the identities above as long as they are tensorial.

We now calculate the sectional curvature on the plane $\spann \{ p_a, p_b \}$ using (\ref{eq:curvature}):
\begin{alignat*}{1}
 4R_{abba} &= 4\big\langle [p_a, p_b], [p_a, p_b] \big\rangle
 = \big\langle [x_a - y_a, x_b - y_b], [x_a- y_a, x_b - y_b] \big\rangle \displaybreak[2] \\
 &= 2 \big\langle [x_a, x_b], [y_a, y_b] \big\rangle - 4 \big\langle [x_a, x_b], [x_a, y_b] \big\rangle - 4 \big\langle [x_a, x_b], [y_a, x_b] \big\rangle \displaybreak[2] \\
 &\hspace{60mm} + 2 \big\langle [x_a, y_b], [x_a, y_b] \big\rangle + 2 \big\langle [x_a, y_b], [y_a, x_b] \big\rangle \displaybreak[2] \\
 &= - 2 \sum_l \TT{a}{b}{l}^2 - 4 \sum_l \TT{a}{b}{l} \big\langle x_l, [x_a, y_b]\big\rangle - 4 \TT{a}{b}{l} \big\langle x_l, [y_a, x_b]\big\rangle \displaybreak[2] \\
 &\hspace{5mm} + \sum_l \left[  \Big\{ \TT{b}{l}{a} + \TT{a}{l}{b} \Big\} \Big\{ \TT{b}{l}{a} + \TT{a}{l}{b} \Big\} \right. \displaybreak[1] \\
&\hspace{3mm} \left.  - \Big\{ \TT{a}{l}{a} + \TT{a}{l}{a} \Big\} \Big\{ \TT{b}{l}{b} + \TT{b}{l}{b} \Big\}  + \Big\{ \TT{b}{l}{a} + \TT{a}{l}{b} \Big\} \Big\{ \TT{a}{l}{b} + \TT{b}{l}{a} \Big\} \right] \displaybreak[0] \\
&\hspace{25mm} -  \sum_l \left[\TT{a}{a}{l} \TT{b}{b}{l} + \TT{a}{b}{l} \TT{a}{b}{l} + \TT{a}{b}{l} \TT{a}{b}{l} \right] + 2A_{abab} \displaybreak[1] \\
&\hspace{5mm} - \sum_l \left[ \Big\{ \TT{a}{l}{a} + \TT{a}{l}{a} \Big\} \Big\{ \TT{b}{l}{b} + \TT{b}{l}{b} \Big\} \right. \displaybreak[0] \\
&\hspace{3mm} \left. - \Big\{ \TT{b}{l}{a} + \TT{a}{l}{b} \Big\} \Big\{ \TT{a}{l}{b} + \TT{b}{l}{a} \Big\} + \Big\{ \TT{b}{l}{a} + \TT{a}{l}{b} \Big\} \Big\{ \TT{b}{l}{a} + \TT{a}{l}{b} \Big\}  \right] \displaybreak[1] \\
&\hspace{25mm} + \sum_l \left[\TT{a}{b}{l} \TT{b}{a}{l} + \TT{a}{b}{l} \TT{b}{a}{l} + \TT{a}{a}{l} \TT{b}{b}{l} \right] - 2A_{abba} \displaybreak[2] \\
&= - 6 \sum_l \TT{a}{b}{l}^2 + 2 \sum_l \Big\{\TT{a}{l}{b} + \TT{b}{l}{a} \Big\}^2 - 8 \sum_l \TT{a}{l}{a} \TT{b}{l}{b} \displaybreak[1] \\
 &\hspace{30mm} - 4 \sum_l \TT{a}{b}{l} \TT{a}{l}{b} + 4 \sum_l \TT{a}{b}{l} \TT{b}{l}{a} + 2A_{abab} - 2A_{abba} .
\end{alignat*}
We will also need
\begin{alignat*}{1}
4 \big\langle [k_m, p_a], [k_m, p_a] \big\rangle
&= \big\langle [x_m + y_m, x_a - y_a], [x_m + y_m, x_a - y_a] \big\rangle \displaybreak[0] \\
&\hspace{-30mm} = - 2\big\langle [x_m, x_a],[y_m, y_a] \big\rangle - 4 \big\langle [x_m, x_a], [x_m, y_a] \big\rangle + 4 \big\langle [x_m, x_a], [y_m, x_a]\big\rangle  \\
&\hspace{32mm} + 2 \big\langle [x_m, y_a], [x_a, y_m] \big\rangle  + 2 \big\langle [x_m, y_a],[x_m, y_a] \big\rangle \displaybreak[1] \\
&\hspace{-30mm} = \sum_{l} \biggr( 2 \TT{m}{a}{l}^2 - 4 \TT{m}{a}{l} \big\langle x_l, [x_m, y_a] \big\rangle + 4 \TT{m}{a}{l} \big\langle x_l, [y_m, x_a] \big\rangle\displaybreak[1] \\
& \hspace{-28mm} + \left[  \Big\{ \TT{m}{l}{m} + \TT{m}{l}{m} \Big\} \Big\{ \TT{a}{l}{a} + \TT{a}{l}{a} \Big\} - \Big\{ \TT{a}{l}{m} + \TT{m}{l}{a} \Big\} \Big\{ \TT{m}{l}{a} + \TT{a}{l}{m} \Big\}  \right. \displaybreak[1] \\
&\hspace{40mm} \left.+ \Big\{ \TT{a}{l}{m} + \TT{m}{l}{a} \Big\} \Big\{ \TT{a}{l}{m} + \TT{m}{l}{a} \Big\}  \right] \displaybreak[1] \\
& - \left[\TT{m}{a}{l} \TT{a}{m}{l} + \TT{m}{a}{l} \TT{a}{m}{l} + \TT{m}{m}{l} \TT{a}{a}{l} \right]  \displaybreak[1] \\
& \hspace{-28mm} + \left[  \Big\{ \TT{a}{l}{m} + \TT{m}{l}{a} \Big\} \Big\{ \TT{a}{l}{m} + \TT{m}{l}{a} \Big\} - \Big\{ \TT{m}{l}{m} + \TT{m}{l}{m} \Big\} \Big\{ \TT{a}{l}{a} + \TT{a}{l}{a} \Big\} \right. \displaybreak[1] \\ 
& \hspace{40mm} \left. + \Big\{ \TT{a}{l}{m} + \TT{m}{l}{a} \Big\} \Big\{ \TT{m}{l}{a} + \TT{a}{l}{m} \Big\} \right] \displaybreak[0] \\ 
& \hspace{16mm} -  \left[\TT{m}{m}{l} \TT{a}{a}{l} + \TT{m}{a}{l} \TT{m}{a}{l} + \TT{m}{a}{l} \TT{m}{a}{l} \right]  \biggr) \displaybreak[1]\\
&\hspace{70mm} + 2 A_{maam} + 2 A_{mama}  \displaybreak[2] \\
& \hspace{-30mm} = \sum_{l} \biggr( 2 \TT{m}{a}{l}^2 + 2 \Big\{ \TT{a}{l}{m} + \TT{m}{l}{a} \Big\}^2 - 4 \TT{m}{a}{l} \TT{m}{l}{a} - 4 \TT{m}{a}{l} \TT{a}{l}{m} \biggr) \\
&\hspace{70mm}+ 2 A_{maam} + 2 A_{mama} . \displaybreak[2] 
\end{alignat*}

We can use these calculations to express the operator $S_\WW$ on $\Sym_2 \un{\af}^\perp$.
Consider a symmetric bilinear form $h \in \Sym_2 \un{\af}^\perp$.
Let $\{ p_{\un{a}} \}$ be an orthonormal basis of $\un{\af}^\perp$ that diagonalizes $h$ and that obeys the splitting $\un{\pp} = \un{\af} \oplus \un{\af}^\perp$.
On $\bigoplus_{\ov{\alpha}} \pp_{\ov{\alpha}}$ we can just choose the standard orthonormal basis $\{ p_{\ov{a}} \}$.
Corresponding to this new orthonormal basis $p_1, \ldots, p_{n-r}$ of $\bigoplus_\alpha \pp_\alpha$ there is then an orthonormal basis $x_1, \ldots, x_{n-r}$ of $\nn$ such that for $y_a = \sigma x_a$, we have $p_a = \frac1{\sqrt{2}} (x_a - y_a)$.
Moreover, we set $k_a = \frac1{\sqrt{2}} (x_a+y_a)$.
Observe that $x_1, \ldots, x_{n-r}$ respects the splitting $\nn = \ov{\nn} \oplus \un{\nn}$, i.e. $\{ x_{\ov{a}} \}$ is a basis for $\ov{\nn}$ and $\{ x_{\un{a}} \}$ one for $\un{\nn}$.

Let $\{ \lambda_{\un{a}} \}$ be the eigenvalues of $h$, i.e. $h = \sum_{\un{a}} \lambda_{\un{a}} p_{\un{a}} \cdot p_{\un{a}}$.
We can compute
\begin{alignat}{1}
 4 \langle \Rm (h), h \rangle &= \sum_{\un{a},\un{b},\un{a}',\un{b}'} 4R_{\un{a}\un{b}\un{b}'\un{a}'} h_{\un{a} \un{a}'} h_{\un{b} \un{b}'} = \sum_{\un{a},\un{b}} 4 R_{\un{a}\un{b}\un{b}\un{a}} \lambda_{\un{a}} \lambda_{\un{b}} \displaybreak[2] \notag \\
&= \sum_{\un{a}, \un{b}, l} \biggr( - 6  \TT{\un{a}}{\un{b}}{l}^2 + 2 \Big\{\TT{\un{a}}{l}{\un{b}} + \TT{\un{b}}{l}{\un{a}} \Big\}^2 - 8 \TT{\un{a}}{l}{\un{a}} \TT{\un{b}}{l}{\un{b}} \displaybreak[0]  \notag \\
 & \qquad - 4 \TT{\un{a}}{\un{b}}{l} \TT{\un{a}}{l}{\un{b}} + 4 \TT{\un{a}}{\un{b}}{l} \TT{\un{b}}{l}{\un{a}} \biggr) \lambda_{\un{a}} \lambda_{\un{b}} +  2 \sum_{\un{a},\un{b}} (A_{\un{a}\un{b}\un{a}\un{b}} - A_{\un{a}\un{b}\un{b}\un{a}}) \lambda_{\un{a}} \lambda_{\un{b}} \label{eq:ingredient2forSW}
\end{alignat}
and
\begin{alignat}{1}
& 4 \sum_{\un{a},\un{m}} \big\langle [k_{\un{m}}, [k_{\un{m}}, p_{\un{a}}]], p_{\un{a}} \big\rangle \lambda_{\un{a}}^2 + 4 \sum_{\un{a},\un{b},\un{m}} \big\langle [k_{\un{m}}, p_{\un{a}}], p_{\un{b}} \big\rangle^2 \lambda_{\un{a}} \lambda_{\un{b}} \displaybreak[2] \notag \\
& = - \sum_{\un{a}, \un{m}, l} \left( 2 \TT{\un{m}}{\un{a}}{l}^2 + 2 \Big\{ \TT{\un{a}}{l}{\un{m}} + \TT{\un{m}}{l}{\un{a}} \Big\}^2 - 4 \TT{\un{m}}{\un{a}}{l} \TT{\un{m}}{l}{\un{a}} - 4 \TT{\un{m}}{\un{a}}{l} \TT{\un{a}}{l}{\un{m}} \right) \lambda_{\un{a}}^2 \displaybreak[1] \notag \\
&\qquad\qquad - 2 \sum_{\un{a}, \un{m}} (A_{\un{m}\un{a}\un{a}\un{m}} + A_{\un{m}\un{a}\un{m}\un{a}}) \lambda_{\un{a}}^2 + 2 \sum_{\un{a},\un{b},\un{m}} \Big\{ \TT{\un{m}}{\un{a}}{\un{b}} - \TT{\un{m}}{\un{b}}{\un{a}} - \TT{\un{a}}{\un{b}}{\un{m}} \Big\}^2 \lambda_{\un{a}} \lambda_{\un{b}} \displaybreak[2] \notag \\
& = - 2 \sum_{\un{x},\un{y},\un{z}} (\lambda_{\un{x}}^2 + \lambda_{\un{y}}^2 + \lambda_{\un{z}}^2 ) \TT{\un{x}}{\un{y}}{\un{z}}^2 - 2 \sum_{\un{a}, \un{m}, \ov{l}} \Big\{ \TT{\un{a}}{\ov{l}}{\un{m}}+ \TT{\un{m}}{\ov{l}}{\un{a}} \Big\}^2 \lambda_{\un{a}}^2 \displaybreak[1] \notag \\ 
& \qquad\qquad + 2 \sum_{\un{a},\un{b}, \un{m}} \Big\{ \TT{\un{m}}{\un{a}}{\un{b}} - \TT{\un{m}}{\un{b}}{\un{a}} - \TT{\un{a}}{\un{b}}{\un{m}} \Big\}^2 \lambda_{\un{a}} \lambda_{\un{b}}
- 2 \sum_{\un{a}, \un{m}} (A_{\un{m}\un{a}\un{a}\un{m}} + A_{\un{m}\un{a}\un{m}\un{a}}) \lambda_{\un{a}}^2. \label{eq:ingredient1forSW}
\end{alignat}
Here, we have used the following two identities:
First, by (\ref{eq:TTid2}) and the fact that $\tTT{\un{a}}{l}{\ov{m}} = 0$
\[ \sum_{\un{m}, l} \TT{\un{m}}{\un{a}}{l} \TT{\un{a}}{l}{\un{m}} = \sum_{m,l} \TT{m}{\un{a}}{l} \TT{\un{a}}{l}{m} = 0 \]
and secondly by exchange of $\un{l}$ and $\un{m}$
\[ \sum_{\un{m}, \un{l}} \Big\{ \TT{\un{a}}{\un{l}}{\un{m}} \TT{\un{m}}{\un{l}}{\un{a}} - \TT{\un{m}}{\un{a}}{\un{l}} \TT{\un{m}}{\un{l}}{\un{a}} \Big\} = 0. \]

We can finally express the operator $S_\WW$.
For this note that the right hand side of (\ref{eq:SWWincoo}) is equal to the negative of the sum of the quantities of (\ref{eq:ingredient1forSW}) and (\ref{eq:ingredient2forSW}). 
\begin{alignat*}{1}
2 \langle S_\WW h, h \rangle
&= 2\sum_{\un{x},\un{y},\un{z}} (\lambda_{\un{x}} + \lambda_{\un{y}} - \lambda_{\un{z}})^2 \TT{\un{x}}{\un{y}}{\un{z}}^2 + 8 \sum_{\un{a},\un{b}, l} \TT{\un{a}}{l}{\un{a}} \TT{\un{b}}{l}{\un{b}} \lambda_{\un{a}} \lambda_{\un{b}} \\ 
&\qquad + 2 \sum_{\un{a}, \un{m}, \ov{l}} \Big\{ \TT{\un{a}}{\ov{l}}{\un{m}} + \TT{\un{m}}{\ov{l}}{\un{a}} \Big\}^2 \lambda_{\un{a}}^2
 - 2 \sum_{\un{a}, \un{b}, \ov{l}} \Big\{ \TT{\un{a}}{\ov{l}}{\un{b}} + \TT{\un{b}}{\ov{l}}{\un{a}} \Big\}^2 \lambda_{\un{a}} \lambda_{\un{b}} \displaybreak[0] \\
& \hspace{20mm} - 2 \sum_{\un{a},\un{b}}(A_{\un{a}\un{b}\un{a}\un{b}} - A_{\un{a}\un{b}\un{b}\un{a}} ) \lambda_{\un{a}} \lambda_{\un{b}} + 2 \sum_{\un{a}, \un{m}} (A_{\un{m}\un{a}\un{a}\un{m}} + A_{\un{m}\un{a}\un{m}\un{a}})\lambda_{\un{a}}^2 \displaybreak[2] \displaybreak[2] \\
& = 2\sum_{\un{x},\un{y},\un{z}} (\lambda_{\un{x}} + \lambda_{\un{y}} - \lambda_{\un{z}})^2 \TT{\un{x}}{\un{y}}{\un{z}}^2 + 8 \sum_l \Big\{ \sum_{\un{a}} \lambda_{\un{a}} \TT{\un{a}}{l}{\un{a}} \Big\}^2 \displaybreak[1] \\
& \hspace{10mm} + \sum_{\un{a}, \un{b}, \ov{l}} \Big\{ \TT{\un{a}}{\ov{l}}{\un{b}} + \TT{\un{b}}{\ov{l}}{\un{a}} \Big\}^2 (\lambda_{\un{a}} - \lambda_{\un{b}} )^2 \\
&\hspace{20mm} - 2 \sum_{\un{a},\un{b}}(A_{\un{a}\un{b}\un{a}\un{b}} - A_{\un{a}\un{b}\un{b}\un{a}} ) \lambda_{\un{a}} \lambda_{\un{b}} + 2 \sum_{\un{a}, \un{m}} (A_{\un{m}\un{a}\un{a}\un{m}} + A_{\un{m}\un{a}\un{m}\un{a}})\lambda_{\un{a}}^2.
\end{alignat*}
We will now show that the last line is always nonnegative.
This will then establish the nonnegativity of $S_\WW$.
To carry out the calculation, we rewrite the last line in tensorial form:
\[ - 2 \sum_{\un{a},\un{a}',\un{b},\un{b}'} (A_{\un{a}\un{b}\un{a}'\un{b}'} - A_{\un{a}\un{b}\un{b}'\un{a}'}) h_{\un{a}\un{a}'} h_{\un{b}\un{b}'} + 2 \sum_{\un{a}, \un{a}', \un{l}, \un{m}} (A_{\un{m}\un{a}\un{a}'\un{m}} + A_{\un{m}\un{a}\un{m}\un{a}'}) h_{\un{a}\un{l}} h_{\un{l}\un{a}'} \]
and we return to the original orthonormal basis $x_1, \ldots, x_{n-r}$, which obeys the nilpotent grading of $\mathfrak{n}$.
Then the expression above becomes
\begin{multline*}
 -2 \sum_{\un{a},\un{b}} ( h_{\un{a}\un{b}} h_{\un{b}\un{a}} - h_{\un{a}\un{a}} h_{\un{b}\un{b}} + h_{\un{a}\un{b}} h_{\un{a}\un{b}} - h_{\un{a}\un{a}} h_{\un{b}\un{b}} + h_{\un{a}\un{b}} h_{\un{b}\un{a}} - h_{\un{a}\un{b}} h_{\un{a}\un{b}} ) \big\langle \alpha_{\un{a}}^\#, \alpha_{\un{b}}^\# \big\rangle + 4 \sum_{\un{a},\un{l}} h_{\un{a}\un{l}}^2 | \alpha_{\un{a}}^\# |^2 \\
 = 4 \Big| \sum_{\un{a}} h_{\un{a}\un{a}} \alpha_{\un{a}}^\# \Big|^2 + 2 \sum_{\un{a},\un{b}} h_{\un{a}\un{b}}^2 | \alpha_{\un{a}}^\# - \alpha_{\un{b}}^\# |^2.
\end{multline*}

So $S_\WW$ is indeed nonnegative definite on $\Sym_2 \un{\af}^\perp$ and the nullspace consists exactly of those $h = \sum_{ab} h_{ab} p_a \cdot p_b$ (for $h_{ab} = h_{ba}$ and the original orthonormal basis $x_1, \ldots, x_{n-r}$, which respects the grading of $\nn$) that satisfy the following \emph{five} identities (\ref{eq:null1})-(\ref{eq:null5}) below
\begin{alignat}{1}
\sum_{\un{i}} h_{\un{a} \un{i}} \TT{\un{i}}{\un{b}}{\un{c}} + \sum_{\un{i}} h_{\un{b}\un{i}} \TT{\un{a}}{\un{i}}{\un{c}} &= \sum_{\un{i}} h_{\un{c}\un{i}} \TT{\un{a}}{\un{b}}{\un{i}} \label{eq:null1} \\
\sum_{\un{i}} h_{\un{i}\un{i}} \alpha_{\un{i}} &= 0 \label{eq:null2} \\
h_{\un{a}\un{b}} &= 0 \qquad \text{if $\alpha_{\un{a}} \not= \alpha_{\un{b}}$} \label{eq:null3} \\
\sum_{\un{i},\un{j}} h_{\un{i}\un{j}} \TT{\un{i}}{a}{\un{j}} &= 0 \label{eq:null4}
\end{alignat}
Observe that condition (\ref{eq:null3}) implies that if $h$ lies in the nullspace of $S_\WW$, then it has block form with respect to the splitting $\bigoplus_{\un{\alpha}} \pp_{\un{\alpha}}$ and hence we can find an orthonormal basis $x_1, \ldots, x_{n-r}$ which both respects the nilpotent grading of $\nn$ and for which the associated orthonormal basis $p_1, \ldots, p_{n-r}$ diagonalizes $h$.
In this basis, we see that identity (\ref{eq:null4}) is redundant.
In the said basis, the fifth identity characterizing the nullspace is
\begin{equation} \label{eq:null5} h_{\un{a} \un{a}} = h_{\un{b}\un{b}} \qquad \text{if} \qquad \TT{\un{a}}{\ov{l}}{\un{b}} + \TT{\un{b}}{\ov{l}}{\un{a}} \not= 0 \qquad \text{for some $\ov{l}$}. \end{equation}

\subsection{Analysis of the nullspace} \label{subsec:nullspace}
Let $h \in \Sym_2 \un{\af}^\perp_\WW$ be a symmetric bilinear form on $\un{\af}^\perp_\WW$.
In the last subsection we found that $h$ lies in the nullspace $\mathcal{N}_\WW \subset \Sym_2 \un{\af}^\perp$ of $S_\WW$ if and only if it satisfies the identities (\ref{eq:null1})-(\ref{eq:null5}).
Set $\af^\perp = \un{\af}^\perp_\CC = \bigoplus_\alpha \pp_\alpha$ and let $\mathcal{N} = \mathcal{N}_\CC \subset \Sym_2 \af^\perp$ be the nullspace corresponding to $\CC$.
In other words, $\mathcal{N}$ is the space of bilinear forms $h \in \Sym_2 \af^\perp$ satisfying
\begin{alignat}{1}
\sum_i h_{ai} \TT{i}{b}{c} + \sum_i h_{bi} \TT{a}{i}{c} &= \sum_i h_{ci} \TT{a}{b}{i} \label{eq:newnull1} \\
\sum_i h_{ii} \alpha_i &= 0 \label{eq:newnull2} \\
h_{ab}& = 0 \qquad \text{if $\alpha_a \not= \alpha_b$} \label{eq:newnull3} 
\end{alignat}
(Recall that identity (\ref{eq:null4}) is redundant.)
In other words, we can say that $\mathcal{N}$ is the space of all $h \in \Sym_2 \af^\perp$ that are in block form with respect to the splitting $\af^\perp = \bigoplus_\alpha \pp_\alpha$ and that satisfy the following two identities if $h = \sum_{a=1}^{n-r} \lambda_a p_a \cdot p_a$ for an orthonormal basis $p_1, \ldots, p_{n-r}$ for which the associated $x_1, \ldots, x_{n-r}$ respect the nilpotent grading of $\nn$:
\begin{alignat}{1}
\lambda_a + \lambda_b &= \lambda_c \qquad \text{if} \qquad \TT{a}{b}{c} \not= 0 \label{eq:newnewnull1} \displaybreak[0] \\
\sum_i \lambda_i \alpha_i &= 0 \label{eq:newnewnull2}
\end{alignat}

\begin{Lemma} \label{Lem:NWWN}
For every wall $\WW \subset \CC$ consider the embedding $\Sym_2 \un{\af}^\perp_\WW \subset \Sym_2 \af^\perp$.
Then $\mathcal{N}_\WW = \mathcal{N} \cap \Sym_2 \un{\af}^\perp_\WW$.
\end{Lemma}
\begin{proof}
Let $h \in \mathcal{N}_\WW$ and interpret $h$ as a symmetric bilinear form on $\af^\perp$.
Then identy (\ref{eq:newnewnull2}) is obviously satisfied as well as (\ref{eq:newnewnull1}) in the case in which $a$, $b$, $c$ are of type $\un{a}$, $\un{b}$, $\un{c}$.
If $c$ is of type $\ov{c}$, then $a$ and $b$ must be of $\ov{a}$ and $\ov{b}$ to guarantee $\tTT{a}{b}{c} \not= 0$, but in this case both sides vanish.
If $a$ is of type $\ov{a}$ and $c$ of type $\un{c}$, then $b$ must be of type $\un{b}$.
So since one of the expressions $\tTT{\un{b}}{\ov{a}}{\un{c}}$ and $\tTT{\un{c}}{\ov{a}}{\un{b}}$ has to vanish, we can use (\ref{eq:null5}) to conclude (\ref{eq:newnewnull1}).
The same is true reversing the roles of $a$ and $b$.

Let now on the other hand $h \in \mathcal{N} \cap \Sym_2 \un{\af}^\perp_\WW$ and consider the diagonalizing basis $p_1, \ldots, p_{n-r}$.
Since (\ref{eq:null1})-(\ref{eq:null3}) are trivially satisfied and (\ref{eq:null4}) is redundant, we only need to establish (\ref{eq:null5}).
This follows from identity (\ref{eq:newnewnull1}) for $\ov{b} = \ov{l}$.
\end{proof}

We will now analyze the nullspace $\mathcal{N}$.
By the following Lemma, we can restrict our analysis to \emph{irreducible} symmetric spaces.
\begin{Lemma} \label{Lem:deRham}
Assume that $\mathfrak{g}$ has a de Rham decomposition $\mathfrak{g}_1 \oplus \ldots \oplus \mathfrak{g}_m$.
Then $\mathcal{N} = \mathcal{N}_1 \oplus \ldots \oplus \mathcal{N}_m$ where $\mathcal{N}_i$ is the nullspace corresponding to $\mathfrak{g}_i$.
\end{Lemma}
\begin{proof}
By (\ref{eq:newnull3}) every $h \in \mathcal{N}$ takes block form with respect to the (coarse) splitting $\af^\perp = \af^\perp_1 \oplus \ldots \af^\perp_m$ coming from the de Rham decomposition.
The other direction is clear.
\end{proof}

\begin{Lemma} \label{Lem:lambdaalambdab}
Consider an $h \in \mathcal{N}$ and choose a diagonalizing orthonormal basis $x_1, \ldots, x_{n-r}$ as above. \\
Assume that for two indices $a, b$ we have $\alpha_a = \alpha_b = \alpha_0$ and that there is a representation
\begin{equation} \label{eq:xirep} x_a = \sum_{u,v} A_{uv} [x_u, x_v] + \sum_{u,v} B_{uv} [x_u, y_v] \end{equation}
such that $B_{uv} = 0$ whenever $\alpha_v = \alpha_0$.
Then $\lambda_a = \lambda_b$
\end{Lemma}
\begin{proof}
Using the subspaces
\begin{alignat*}{1}
 \mathfrak{g}_{\alpha, \lambda} &= \spann \{ x_i \;\; : \;\; \alpha_i = \alpha, \lambda_i = \lambda \} \\
 \mathfrak{g}_{-\alpha, -\lambda} &= \spann \{ y_i \;\; : \;\; \alpha_i = \alpha, \lambda_i = \lambda \} = \sigma \mathfrak{g}_{\alpha, \lambda},
\end{alignat*}
we obtain refined splittings
\[ \nn = \bigoplus_{\alpha \in \Delta^+, \lambda} \mathfrak{g}_{\alpha, \lambda} \qquad \text{and} \qquad \nn^- = \bigoplus_{\alpha \in \Delta^+, \lambda} \mathfrak{g}_{-\alpha, -\lambda}. \]
By (\ref{eq:newnewnull1}) and (\ref{eq:xiyiformula}), we conclude
\begin{alignat*}{1}
\big[ \mathfrak{g}_{\alpha, \lambda}, \mathfrak{g}_{\alpha', \lambda'} \big] &\subset \mathfrak{g}_{\alpha + \alpha', \lambda + \lambda'} \\
\big[ \mathfrak{g}_{\alpha, \lambda}, \mathfrak{g}_{-\alpha', -\lambda'} \big] &\subset \mathfrak{g}_{\alpha - \alpha', \lambda - \lambda'} \qquad \text{if $\alpha \not= \alpha'$}.
\end{alignat*}
Consider the representation (\ref{eq:xirep}) of $x_a$ and observe that $[x_u, x_v] \in \mathfrak{g}_{\alpha_u + \alpha_v, \lambda_u + \lambda_v}$ and $[x_u, y_v] \in \mathfrak{g}_{\alpha_u - \alpha_v, \lambda_u - \lambda_v}$.
So if we set $A_{uv} = 0$ whenever $\alpha_u + \alpha_v \not= \alpha_0$ or $\lambda_u + \lambda_v \not= \lambda_a$ as well as $B_{uv} = 0$ whenever $\alpha_u - \alpha_v \not= \alpha_0$ or $\lambda_u - \lambda_v \not= \lambda_a$, representation (\ref{eq:xirep}) continues to hold.
We will assume this property from now on.

Now observe that by (\ref{eq:xiyiformula}), we have $[x_a, y_b] \in \mathfrak{k}_0$ and hence $[x_a, y_b] = \sigma [x_a, y_b] = [y_a, x_b]$.
We can therefore compute using (\ref{eq:xiyi})
\[ [[x_a, y_b], x_a] = [[y_a, x_b], x_a]  = [[x_a, x_b], y_a] + [[y_a, x_a], x_b] = [[x_a, x_b], y_a] + |\alpha_0^\#|^2 x_b \]
By looking at the right hand side, we conclude that $[[x_a, y_b], x_a] \in \mathfrak{g}_{\alpha_0, \lambda_b}$.
Moreover, this expression does not vanish, since taking the scalar product with $y_b$ yields
\[ \big\langle [[x_a, y_b], x_a], y_b \big\rangle = \big\langle [x_a, x_b], [y_a, y_b] \big\rangle - |\alpha_0^\#|^2 < 0. \]

We will now show that also $[[x_a, y_b], x_a] \in \mathfrak{g}_{\alpha_0, 2 \lambda_a - \lambda_b}$.
This will then imply $\lambda_a = \lambda_b$.
Using the representation (\ref{eq:xirep}), we find
\begin{alignat*}{1}
[[x_a,y_b],x_a] = &\sum_{u,v} A_{uv} [[[x_u,x_v],y_b],x_a] + \sum_{u,v} B_{uv} [[[x_u,y_v],y_b],x_a] \\
= &\sum_{u,v} A_{uv} \big( [[[y_b, x_v], x_u], x_a] + [[[x_u, y_b], x_v], x_a] \big) \\
&+ \sum_{u,v} B_{uv} \big( [[[y_b,y_v],x_u], x_a] + [[[x_u, y_b], y_v], x_a] \big) \displaybreak[1] \\
 = & \sum_{u,v} A_{uv} \big( [[x_a,x_u],[y_b,x_v]] + [[[y_b,x_v],x_a], x_u] \\
 & \hspace{50mm} + [[x_a, x_v], [x_u, y_b]] + [[[x_u,y_b],x_a],x_v] \big) \\
 & +  \sum_{u,v} B_{uv} \big( [[x_a, x_u],[y_b, y_v]] + [[[y_b,y_v],x_a],x_u] \displaybreak[1] \\
 & \hspace{50mm} + [[x_a,y_v],[x_u, y_b]] + [[[x_u,y_b],x_a],y_v] \big)
\end{alignat*}
This implies $[[x_a, y_b], x_a] \in \mathfrak{g}_{\alpha_0, 2 \lambda_a - \lambda_b}$ since none of the successive Lie brackets lie in $\mathfrak{g}_0$.
Note here that for the seventh term, we have used the property that $B_{uv} = 0$ if $\alpha_v = \alpha_0$.
\end{proof}

\begin{Lemma} \label{Lem:hasrep}
Assume that $\mathfrak{g}$ is the Lie algebra of an irreducible symmetric space.
If its rank is greater than $1$, then every $x_a$ has a representation (\ref{eq:xirep}). \\
If its rank is equal to $1$, then $\Delta = \{ - \alpha', 0, \alpha' \}$ or $\Delta = \{ - 2\alpha' , -\alpha', 0, \alpha', 2 \alpha' \}$ and every $x_a \in \mathfrak{g}_{2\alpha'}$ has a representation (\ref{eq:xirep}).
\end{Lemma}
\begin{proof}
Set $\alpha_0 = \alpha_a$ and consider the following subspace of $\mathfrak{g}_{\alpha_0}$:
\[ V = \Big\{ \sum_{u,v} A_{uv} [x_u, x_v] + \sum_{u,v} B_{uv} [x_u, y_v] \;\; : \;\; \text{$B_{uv} = 0$ if $\alpha_v = \alpha_0$} \Big\} \cap \mathfrak{g}_{\alpha_0}. \]
Assume that $V \not= \mathfrak{g}_{\alpha_0}$.
Then there is an $x \in \mathfrak{g}_{\alpha_0}$ such that for $y = \sigma x$ we have
\[ \big\langle [x_u, x_v], y \big\rangle = 0 \quad \text{for all $u, v$} \qquad \text{and} \qquad \big\langle [x_u,y_v], y \big\rangle = 0 \quad \text{if $\alpha_v \not= \alpha_0$}. \]
This implies that 
\begin{alignat*}{1}
 [x_u, y] \quad &\text{has no component in $\mathfrak{g}_{-\alpha_v}$ for all $u, v$,} \\
 [x_u, y] \quad &\text{has no component in $\mathfrak{g}_{\alpha_v}$ if $\alpha_v \not= \alpha_0$,} \\
 [y_v, y] \quad &\text{has no component in $\mathfrak{g}_{-\alpha_u}$ if $\alpha_v \not= \alpha_0$.}
\end{alignat*}
Hence, we conclude that
\[ [\mathfrak{g}_\beta, y] = 0 \qquad \text{if $\beta \in \Delta \setminus \{ - 2\alpha_0, -\alpha_0, 0, \alpha_0, 2\alpha_0 \}$}. \]
Applying $\sigma$ yields $[\mathfrak{g}_\beta, x] = 0$ for the same $\beta$'s.
So we also have 
\[ 0 = [\mathfrak{g}_\beta, [y,x]] = [\mathfrak{g}_\beta, \alpha^\#]. \]
This implies that $\langle \alpha^\#, \beta^\# \rangle = 0$ for all $\beta \in \Delta \setminus \{ - 2\alpha_0 , -\alpha_0 , 0, \alpha_0, 2\alpha_0 \}$.
In the higher rank case this contradicts the irreducibility of $\mathfrak{g}$.

In the rank $1$ case, the Lemma follows from the fact that $\mathfrak{g}_{2\alpha'} = [\mathfrak{g}_{\alpha'}, \mathfrak{g}_{\alpha'}]$.
\end{proof}

We will now completely analyze the case in which $\mathfrak{g}$ is the Lie algebra of a rank $1$ symmetric space $M$.
The only possibilities here are real, complex, quaternionic and octonionic hyperbolic space:
\[ \IR \IH^n, \quad \IC \IH^{2n}, \quad \IH \IH^{4n}, \quad \IO \IH^{16} \]
where $n \geq 2$.
The symbols $\IH$ and $\IO$ denote the division algebras of the quaternions and the octonions.
We left out the spaces $\IC \IH^2$ and $\IH \IH^4$ since they are isometric to $\IR \IH^2$ resp. $\IR \IH^4$.
Observe that octonionic hyperbolic space only exists in dimension $16$.

Obviously, $\af$ has dimension $1$.
The set of positive roots $\Delta^+$ consists of a single root $\alpha$ in the real case and two roots $\alpha, 2\alpha$ in the other cases.
We can model the algebraic structure of the root spaces in the following way (see e.g. \cite[\S 19]{Mos}): 
Let $\IK = \IR, \IC, \IH$ or $\IO$ depending on which space we look at.
Denote by $\im \IK = \{ v \in \IK \; : \; \ov{v} = - v \}$ the imaginary subspace.
Observe that $\dim \im \IK = \dim \IK - 1$.
In the case $\IK = \IO$ let $n = 2$.
We have the identifications
\begin{equation} \label{eq:KImK} \mathfrak{g}_{\alpha} = \IK^{n-1}, \qquad \mathfrak{g}_{2\alpha} = \im \IK. \end{equation}
For $v, w \in \mathfrak{g}_{\alpha} = \IK^{n-1}$ set
\[ (v,w) = \ov{v}_1 w_1 + \ldots + \ov{v}_{n-1} w_{n-1}. \]
Then we can describe the Lie algebra structure on $\mathfrak{n} = \mathfrak{g}_{\alpha} \oplus \mathfrak{g}_{2\alpha}$ by
\[ [v,w] = 2 \im (v,w). \]

\begin{Lemma} \label{Lem:Nrank1}
If $\mathfrak{g}$ is the Lie algebra of a rank $1$ symmetric space $M$, then we can describe the nullspace $\mathcal{N}$ as follows $(n \geq 2)$
\begin{enumerate}[(1)]
\item If $M = \IR \IH^n$, then $\af^\perp \cong \IR^{n-1}$ and $\mathcal{N} = \{ h \in \Sym_2 \IR^{n-1} \;\; : \;\; \tr h = 0 \}$.
\item If $M = \IC \IH^{2n}$, then $\af^\perp \cong \IC^{n-1} \oplus \IR$.
View $\IC^{n-1}$ as $\IR$-vector space and let the endomorphism $J$ denote multiplication by $i$. \\
Then $\mathcal{N} = \{ h \in \Sym_2 \IC^{n-1} \;\; : \;\;  J h + h J = 0 \}$.
\item If $M = \IH \IH^{4n}$ or $M = \IO \IH^{16}$, then $\mathcal{N} = \{ 0 \}$.
\end{enumerate}
\end{Lemma}
\begin{proof}
Let $\IK = \IR, \IC, \IH$ or $\IO$ and recall the identifications (\ref{eq:KImK}).
We will furthermore identify $\nn$ with $\af^\perp$ via the map $x \mapsto \frac1{\sqrt{2}}( x - \sigma x)$ and sometimes view symmetric bilinear forms on $\af^\perp$ as endomorphisms on $\nn$.
Then a symmetric bilinear form $h$ lies in $\mathcal{N}$ if and only if
\begin{enumerate}[(i)]
\item $h = h^0 + h^1$ for symmetric bilinear forms $h^0$, $h^1$ on $\mathfrak{g}_\alpha = \IK^{n-1}$ resp. $\mathfrak{g}_{2\alpha} = \im \IK$ (see (\ref{eq:newnull3})).
\item For all $v, w \in \IK^{n-1}$, we have $[h^0(v),w] + [v, h^0(w)] = h^1([v,w])$ (see (\ref{eq:newnull1})).
\item $\tr h^0 + 2 \tr h^1 = 0$ (see \ref{eq:newnull2}).
\end{enumerate}
Hence, the case $\IK = \IR$ is settled.

Now assume that $\im \IK$ is nontrivial.
By Lemmas \ref{Lem:lambdaalambdab} and \ref{Lem:hasrep}, we conclude that $h^1 = \lambda \id$ for some $\lambda$.
Let $x_1, \ldots, x_{d-1}$ be an orthonormal basis of $\mathfrak{g}_{2\alpha} = \im \IK$ and set $y_i = x_i$.
For each $i = 1, \ldots, d-1$ define an endomorphism $J_i \in \End \mathfrak{g}_\alpha$ by
\[ \big\langle J_i v, w \big\rangle = \big\langle \im (v, w), y_i \big\rangle. \]
The $J_i$ are antisymmetric and satisfy $J_i^2 = -1$ (in the case $\IK = \IC$, we have $J_1 = J$).
So condition (ii) reads
\[ J_i h^0 + h^0 J_i = \lambda J_i \qquad \text{for all $i = 1, \ldots, d-1$.} \]
We conclude that $2 \tr h^0 =  d (n-1) \lambda$ and hence by condition (iii) it follows that $0 = \tr h^0 + 2 \tr h^1 = \frac12 d (n-1) \lambda + (d-1) \lambda = (\frac12 d n + \frac12 d - 1) \lambda$ and thus $\lambda = 0$.
This establishes the case $\IK = \IC$.

In the case $\IK = \IH$ we can choose $x_1, x_2, x_3 \in \im \IH$ such that $x_1 x_2 =  x_3$ and hence $J_1 J_2 = J_3$.
Then
\[ h^0 J_3 = h^0 J_1 J_2 = - J_1 h^0 J_2 = J_1 J_2 h^0 = J_3 h^0. \]
Together with $J_3 h^0 + h^0 J_3 = 0$ this yields $h^0 = 0$.

Finally using a multiplication table, we see that in the case $\IK = \IO$ we can choose $x_1, \ldots, x_7 \in \im \IO$ such that $J_1 J_2 J_3 J_4 J_5 J_6 = J_7$.
Hence with
\[ h^0 J_7 = h^0 (J_1 \cdots J_6) = - J_1 h^0 (J_2 \cdots J_6) = \ldots = (J_1 \cdots J_6) h^0 = J_7 h^0 \]
and $h^0 J_7 + J_7 h^0 = 0$, we conclude $h^0 = 0$.
\end{proof}

\subsection{Conclusion} \label{subsec:lambdaWWconclusion}
We can finally give a proof of Proposition \ref{Prop:lambdaWWpos}:
\begin{proof}[Proof of Proposition \ref{Prop:lambdaWWpos}]
The nonnegativity of $\lambda_\WW$ follows from Lemmas \ref{Lem:partovpunp}, \ref{Lem:ovpunp}, \ref{Lem:Sym2ovp}, \ref{Lem:partsont} and the calculations of subsection \ref{subsec:afperp}.
In the case $M = \IH^n$ or $M = \IC \IH^{2n}$, the Proposition follows from Lemmas \ref{Lem:Bochnerformula} and \ref{Lem:Nrank1}.

In order to show $\lambda_0 > 0$, in the case in which $M$ does not contain any hyperbolic or complex hyperbolic factor, we only need to show that $\mathcal{N}_\WW = \{ 0 \}$.
By Lemma \ref{Lem:NWWN}, it suffices to show $\mathcal{N} = \{ 0 \}$ and by Lemma \ref{Lem:deRham}, we can assume that $M$ is irreducible.
For the rank $1$ case we use Lemma \ref{Lem:Nrank1} and the case $M = \IH^2$ is clear (see Example B in subsection \ref{subsec:heatkerIntro}).

Assume now that $M$ is irreducible and of higher rank and let $h = \sum_{a=1}^{n-r} \lambda_a p_a \cdot p_a \in \mathcal{N}$.
Lemmas \ref{Lem:lambdaalambdab} and \ref{Lem:hasrep} yield that $\lambda_a = \lambda_b$ whenever $\alpha_a = \alpha_b$.
Hence the eigenvalue $\lambda_a$ depends only on the root $\alpha_a$.
Consider the simple roots $\beta_1, \ldots, \beta_r$ of $\Delta^+$ and let $\lambda'_1, \ldots, \lambda'_r$ be the corresponding eigenvalues.

We now show that $\mathfrak{n}$ is generated by the linear subspace
\[ \bigoplus_{i=1}^r \mathfrak{g}_{\beta_i}. \]
Let $\alpha_a \in \Delta^+$ be a positive root.
Since $\alpha_a$ can be expressed as a linear combination of the $\beta_i$ with \emph{positive} coefficients, there must be an $i \in \{ 1, \ldots, r \}$ such that $\langle \alpha_a^\#, \beta_i^\# \rangle > 0$ and such that $\alpha_a - \beta_i \in \Delta^+$.
Consider the vectors $x_a \in \mathfrak{g}_\alpha$, $y_a \in \mathfrak{g}_{-\alpha}$, $x_i \in \mathfrak{g}_{\beta_i}$, $y_i \in \mathfrak{g}_{-\beta_i}$ and recall that $y_a = \sigma x_a$, $y_i = \sigma x_i$ and $[x_a, y_a] = - \alpha_a^\#$, $[x_i, y_i] = - \beta_i^\#$.
Then similarly as in the proof of Lemma \ref{Lem:Bochnerformula}
\[
 0 < \langle \alpha_a^\#, \beta_i^\# \rangle = \big\langle [x_a, y_a], [x_i, y_i] \big\rangle 
 = \big\langle [x_a, y_i], [x_i, y_a] \big\rangle + \big\langle [x_a, x_i], [y_a, y_i] \big\rangle
\] 
Since $\langle [x_a, x_i], [y_a, y_i] \rangle =  \langle [x_a, x_i], \sigma [y_a, y_i] \rangle \leq 0$, we must have
\[ 0 < \big\langle [x_a, y_i], [x_i, y_a] \big\rangle = - \big\langle y_a, [x_i, [x_a, y_i]] \big\rangle. \]
So $x_a \in [\mathfrak{g}_{\alpha_a - \beta_i}, \mathfrak{g}_{\beta_i}]$.
We find $\mathfrak{g}_{\alpha_a} = [\mathfrak{g}_{\alpha_a - \beta_i}, \mathfrak{g}_{\beta_i}]$.
This proves the claim.

By (\ref{eq:newnewnull1}) we find that whenever $\alpha = \sum_{i=1}^r k_i \beta_i$ we have $\lambda_a = \sum_{i=1}^r k_i \lambda_i$.
So there is an element $v \in \af$ such that $\lambda_a = \alpha_a(v)$ for all $a$.
By (\ref{eq:newnewnull2}) and the first identity in the proof of Lemma \ref{Lem:partsont} we conclude
\[ 0 = \sum_{a=1}^{n-r} \alpha_a(v) \alpha_a^\# = \tfrac12 v \]
and hence $h = 0$.
This proves the desired result.
\end{proof}

\section{Proofs of the main theorems} \label{sec:proofs}
\subsection{Introduction} \label{sec:proofsPrelim}
In this section, we will prove the stability results Theorems \ref{Thm:mainA} and \ref{Thm:mainB}.
Consider a solution $(g_t)_{t \in [0,T)}$ to Ricci deTurck flow (\ref{eq:RdTflow}).
Recall from subsection \ref{subsec:RdTflow} that we can write the evolution equation for $h_t = g_t - \ov{g}$ as
\begin{equation} \label{eq:flowequation}
\partial_t h_t + L h_t = Q_t = R_t + \nabla^* S_t
\end{equation}
where
\[  |Q_t| \leq C (|\nabla h_t|^2 + |h_t| |\nabla^2 h_t|), \quad |R_t| \leq C | \nabla h_t|^2, \quad |S_t| \leq C |h_t| |\nabla h_t|.  \]
Let $k_t \in C^\infty( M \times M ; \Sym_2 T^*M \boxtimes (\Sym_2 T^*M)^*)$ be the kernel of the Einstein operator $L$, i.e.
\[ \partial_t k_t(\cdot, x_1) = - L k_t(\cdot, x_1) \qquad \text{and} \qquad k_t(\cdot, x_1) \xrightarrow{t \to 0} \delta_{x_1}  \id_{\Sym_2 T^*_{x_1} M}.  \]
For $(x_1,t_1) \in M \times [0,T)$ and $0 \leq t_0 < t_1$, we obtain by convolution
\begin{alignat}{1}
 h(x_1,t_1) &= \int_M k_{t_1-t_0} (x_1, x) h_{t_0}(x) dx + \int_{t_0}^{t_1} \int_M k_{t_1-t}(x_1, x) Q_t(x) dx dt \notag \\
&= \int_M k_{t_1-t_0} (x_1, x) h_{t_0}(x) dx \notag \\
&\hspace{10mm}+ \int_{t_0}^{t_1} \int_M \big( k_{t_1-t}(x_1, x) R_t(x) + \nabla k_{t_1 - t} (x_1, x) S_t (x) \big) dx dt. \label{eq:convoleq}
\end{alignat}
We will frequently make use of this identity.

In the next subsection, we prove Theorem \ref{Thm:mainA}.
In order to establish Theorem \ref{Thm:mainB}, we first derive some more precise short-time estimates in subsection \ref{subsec:sec6shorttimest}.
Then, we present a trick involving the geometry of negative sectional curvature to obtain a good estimate on the linearized equation in subsection \ref{subsec:hypgeometry}.
Finally, we prove Theorem \ref{Thm:mainB} in subsection \ref{subsec:ThmmainB}.

\subsection{Proof of Theorem \ref{Thm:mainA}}
\begin{proof}[Proof of Theorem \ref{Thm:mainA}]
Observe that by passing to its universal cover, we can always assume $M$ to be simply connected.

By Theorem \ref{Thm:dectrianglegenrank} and Proposition \ref{Prop:lambdaWWpos}, we know that there are constants $\lambda > 0$ and $C < \infty$ such that we have the following bound on the heat kernel
\[ \Vert k_t(x_1, \cdot) \Vert_{L^1(M)} \leq C e^{-\lambda t} \qquad \text{for all $x_1 \in M$ and $t > 0$}. \]

Let $\varepsilon_0 > 0$ be a small constant which we will determine in the course of the proof and define $T_{\max}$ to be the maximum over all $T$ such that Ricci deTurck flow $h_t$ starting from $h_0 = g_0 - \ov{g}$ exists on $[0,T)$ and satisfies $\Vert h_t \Vert_{L^\infty(M)} < \varepsilon_0$ everywhere.
In the following, we will show that for sufficiently small $\varepsilon_0$, we have $\Vert h_t \Vert_{L^\infty(M)} \leq C_1 \varepsilon e^{-\lambda t}$ where $\varepsilon$ is the constant that controls $\Vert h_0 \Vert_{L^\infty(M)}$.
Hence, if $\varepsilon$ is small enough, we conclude that $\Vert h_t \Vert_{L^\infty(M)} < \varepsilon_0 / 2$ on $[0,T_{\max})$ and Proposition \ref{Prop:shortex} yields  $T_{\max} = \infty$.

Now set for every $t_1 \in [0, T_{\max})$
\[ Z_{t_1} = \max_{t \in [0,t_1]} e^{\lambda t} \Vert h_t \Vert_{L^\infty(M)} \]
Again, by Proposition \ref{Prop:shortex}, we find that if we choose $\varepsilon$ small enough, we have $T_{\max} > \tau_{s.e.}$ and $Z_{\tau_{s.e.}} \leq C \varepsilon$.
By Corollary \ref{Cor:Shi}, we conclude that for sufficiently small $\varepsilon_0$, we have for all $t \in [\tau_{s.e.}, T_{\max})$
\[ \Vert \nabla^m h_t \Vert_{L^\infty(M)} \leq C_m Z_t  e^{-\lambda t}\qquad \Longrightarrow \qquad \Vert Q_t \Vert_{L^\infty(M)} \leq C Z^2_t e^{-2\lambda t}. \]
We now use (\ref{eq:convoleq}) for $t_1 > t_0 = \tau_{s.e.}$:
\[ h_{t_1}(x_1) = \int_M k_{t_1 - \tau_{s.e.}} (x_1, x) h_{\tau_{s.e.}} (x) dx + \int_{\tau_{s.e.}}^{t_1} \int_{M} k_{t_1-t}(x_1,x) Q_t(x) dx dt \]
to obtain the estimate
\[ |h_{t_1}|(x_1) \leq C \varepsilon e^{-\lambda t_1} + C Z_{t_1}^2 \int_{\tau_{s.e.}}^{t_1} e^{-\lambda (t_1-t)} e^{- 2 \lambda t} dt \leq C \varepsilon e^{-\lambda t_1} + C Z_{t_1}^2 e^{-\lambda t_1}. \]
We conclude that there is a constant $C_0 < \infty$ such that
\[ Z_{\tau_{s.e.}} \leq C_0 \varepsilon \qquad \text{and} \qquad Z_t \leq C_0 (\varepsilon + Z_t^2) \qquad \text{for all $t \in [\tau_{s.e.},T_{\max})$}. \]
Now assume $\varepsilon < (2 C_0)^{-2}$.
Observe that since $Z_t$ is continuous in $t$, either $Z_t \leq 2 C_0 \varepsilon$ holds for all times $t \in [\tau_{s.e.}, T_{\max})$ or there is a time $t \in [\tau_{s.e.}, T_{\max})$ with $Z_t = 2 C_0 \varepsilon$.
However, the latter case immediately gives a contradiction:
\[ 2 C_0 \varepsilon = Z_t \leq C_0 (\varepsilon + Z_t^2) = C_0 ( \varepsilon + 4 C_0^2 \varepsilon^2) < 2 C_0 \varepsilon. \]
This implies the claim for $C_1 = 2 C_0$.
\end{proof}

\subsection{Short-time estimates} \label{subsec:sec6shorttimest}
In this subsection, we establish some analytical facts that are needed later in the proof of Theorem \ref{Thm:mainB}.
Our main result will be Lemma \ref{Lem:h1h2splitting}, which states that a small perturbation that has a representation as in the assumption of Theorem \ref{Thm:mainB}, will continue having such a representation for a small times and there are suitable a priori derivative estimates.
We will partially make use of methods developed in \cite{KL}.

\begin{Lemma} \label{Lem:locestimate}
There is an $\varepsilon_0 > 0$ and constants $\varepsilon_m > 0$ such that:
Let $r_1 < 1$, $t_1 = r_1^2$, $x_1 \in M$ and assume that $(h_t)_{t \in [0,t_1]}$ is a solution to  (\ref{eq:flowequation}) on $B_{2r_1} (x_1)$.
Then if $|h_t| < \varepsilon_0$ everywhere,
\begin{multline*}
r_1^{-\frac12(n+2)} \Vert \nabla h \Vert_{L^2(B_{r_1}(x_1) \times [0,r_1^2])} \leq C r_1^{-1-\frac12 (n+2)} \Vert h \Vert_{L^2(B_{2r_1}(x_1) \times [0,r_1^2])} \\
 \qquad\qquad\qquad + C r_1^{-1 - \frac12 n} \Vert h_0 \Vert_{L^2(B_{2r_1}(x_1))}.
\end{multline*}
Moreover, if $|h_t| < \varepsilon_m$ everywhere, we have for all $(x,t) \in B_{r_1}(x_1) \times [\frac12 t_1, t_1]$
\[ | \nabla^m h| (x, t) \leq C_m r_1^{-m-\frac12 (n+2)} \Vert h \Vert_{L^2(B_{2r_1}(x_1) \times [0,r_1^2])}. \]
\end{Lemma}
\begin{proof}
As for the first estimate consider a cutoff function $\eta \in C^\infty(M)$ that is equal to $1$ on $B_{r_1}(x_1)$, vanishes outside $B_{2r_1}(x_1)$ and satisfies $|\nabla \eta| \leq C r_1^{-1}$.
We use $\partial_t h_t + \nabla^* \nabla h_t = R_t + \nabla^* S_t$ and $|R_t| \leq C |\nabla h_t|^2$, $|S_t| \leq C |h_t| |\nabla h_t|$ to carry out the following computation (integration will always be over $B_{2r_1}(x_1)$ and $|h_t| < \varepsilon_0$ is assumed to be sufficiently small)
\begin{alignat*}{1}
 & \tfrac12 \partial_t \int \eta^2 |h_t|^2 + \int \eta^2 |\nabla h_t|^2 \\
& \hspace{10mm} \leq \int \eta^2 |R_t| |h_t| + \int \eta^2 |S_t| |\nabla h_t| + 2 \int \eta |\nabla \eta| |h_t| |\nabla h_t| + 2 \int \eta |\nabla \eta| |S_t| |h_t| \\
& \hspace{10mm} \leq \int \eta^2 |R_t| |h_t| + \int \eta^2 |S_t|^2 + \tfrac14 \int \eta^2 |\nabla h_t|^2 + 4 \int |\nabla \eta|^2 |h_t|^2 \\ 
& \hspace{20mm} + \tfrac14 \int \eta^2 |\nabla h_t|^2 +  \int \eta^2 |S_t|^2 + \int |\nabla \eta|^2 |h_t|^2 \\
& \hspace{10mm} \leq \tfrac34 \int \eta^2 |\nabla h_t|^2 + 5 \int |\nabla \eta|^2 |h_t|^2 
\end{alignat*}
Hence
\[ \int_0^{t_1} \int_{B_{2r_1}(x_1)} \eta^2 |\nabla h_t|^2 \leq  C r_1^{-2} \int_0^{t_1} \int_{B_{2r_1}(x_1)} |h_t|^2 + 2 \int_{B_{2r_1}(x_1)} \eta^2 |h_0|^2. \]
This establishes the first inequality.

In order to prove the second inequality, we first choose constants $\rho_m = 1 + 2^{-m}$ and $\tau_m = \frac12 - 2^{-m-2}$ for $m \geq 0$.
Observe that $1 < \rho_m \leq 2$ decreases and $\frac14 \leq \tau_m < \frac12$ increases in $m$.
For the following fix $m \geq 0$ and consider a cutoff function $\eta \in C^\infty(M)$ that is equal to $1$ on $B_{\rho_{m+1} r_1}(x_1)$, vanishes outside $B_{\rho_m r_1}(x_1)$ and satisfies $|\nabla \eta| < C_m r_1^{-1}$.
If we differentiate (\ref{eq:flowequation}) $m$ times, we obtain
\[ \partial_t \big( \nabla^m h_t \big) + \nabla^* \nabla \big( \nabla^m h_t \big) = \big[ \nabla^* \nabla, \nabla^m \big] h_t + Q_t^{(m)} \]
where
\[ Q_t^{(m)} = \sum_{\substack{i_1 + \ldots + i_k = m+2, \;\;  k \geq 2, \\ i_1 \geq 0, \;\; i_2, \ldots, i_k \geq 1}} \nabla^{i_1} h_t * \ldots * \nabla^{i_k} h_t. \]
Now observe that by Corollary \ref{Cor:Shi} and the bound $|h_t| < \varepsilon_{m}$, we have $| \nabla^i h | \leq C'_i \varepsilon_{m} r_1^{-i}$ for all $i \leq m + 2$ on $M \times [\tau_m t_1, t_1]$ if $\varepsilon_{m}$ is sufficiently small.
Hence, we can bound the two extra terms as follows
\begin{alignat*}{1}
 |Q_t^{(m)}| &\leq C_m \sum_{i=0}^{m+1} r_1^{-m-2+i} | \nabla^i h_t |, \\
 \big| \big[ \nabla^* \nabla, \nabla^m \big] h_t \big| &\leq C_m \sum_{i=0}^m | \nabla^i h_t| \leq C_m \sum_{i=0}^m r_1^{-m-2+i} | \nabla^i h_t|.
\end{alignat*}
Hence similarly as before
\begin{alignat*}{1}
& \tfrac12 \partial_t \int \eta^2 | \nabla^m h_t |^2 + \int \eta^2 | \nabla^{m+1} h_t|^2 \\
& \hspace{5mm} \leq \int \eta^2 |Q_t^{(m)}| |\nabla^m h_t| + \int \eta^2 \big| [ \nabla^* \nabla, \nabla^m ] h_t \big| | \nabla^m h_t | + 2 \int \eta | \nabla \eta | |\nabla^{m+1} h_t| |\nabla^m h_t| \\
& \hspace{5mm} \leq C_m \sum_{i=0}^m r_1^{-m-2+i} \int \eta^2 | \nabla^i h_t | | \nabla^m h_t| + C_m r_1^{-1} \int \eta |\nabla^{m+1} h_t| |\nabla^m h_t|  \\
& \hspace{5mm} \leq C_m \sum_{i=0}^m r_1^{-2m-2+2i} \int_{B_{\rho_m r_1}(x_1)} | \nabla^i h_t |^2 + \tfrac12 \int \eta^2 | \nabla^{m+1} h_t|^2.
\end{alignat*}
So
\[ \partial_t \int \eta^2 | \nabla^m h_t |^2 + \int \eta^2 | \nabla^{m+1} h_t|^2 \leq C_m \sum_{i=0}^m r_1^{-2m-2+2i} \int_{B_{\rho_m r_1}(x_1)} | \nabla^i h_t |^2. \]
We now multiply this inequality by $t/t_1 - \tau_{m+\frac12}$ and integrate it first from $\tau_{m+\frac12} t_1$ to some $t' \in [\frac12 t_1, t_1]$ and then from $\tau_{m+\frac12} t_1$ to $t_1$ to find
\begin{alignat*}{1}
 \Vert \nabla^{m} h_{t'} \Vert_{L^2(B_{\rho_{m+1} r_1}(x_1))} &\leq C_m \sum_{i=0}^m r_1^{-m-1+i} \Vert \nabla^i h \Vert_{L^2(B_{\rho_i r_1}(x_1) \times [\tau_i t_1, t_1])}, \\
 \Vert \nabla^{m+1} h \Vert_{L^2(B_{\rho_{m+1} r_1}(x_1) \times [\tau_{m+1} t_1, t_1])} &\leq C_m \sum_{i=0}^m r_1^{-m-1+i} \Vert \nabla^i h \Vert_{L^2(B_{\rho_i r_1}(x_1) \times [\tau_i t_1, t_1])}.
\end{alignat*}
Hence, by induction
\[  \Vert \nabla^m h_{t'} \Vert_{L^2(B_{\rho_{m+1} r_1}(x_1))} \leq C_m r_1^{-m-1} \Vert h \Vert_{L^2(B_{2r_1}(x_1) \times [0,t_1])} \]
and Sobolev embedding for large $m$ yields the desired result.
\end{proof}

In the following let $\sigma_0^2 = \tau_{s.e.}$ where $\tau_{s.e.}$ is the constant from Proposition \ref{Prop:shortex}.
\begin{Lemma} \label{Lem:h1h2splitting}
There are constants $A_m < \infty$ and $\varepsilon_m > 0$ such that for all $a \geq 0$, $b_1, b_2 \geq 0$, $\sigma \leq \sigma_0$, $m \geq 0$ and $q \geq 2$ we have: \\
Let $(h_t)_{t \in [0, \sigma^2]}$ be a solution to (\ref{eq:flowequation}) and assume that $|h_0| < \varepsilon_m$ and $h_0 = h_0^1 + h_0^2$ with $|h_0^1|, |h_0^2| < \varepsilon_m$ and
\[ |h_0^1| \leq \frac{b_1}{r+1+a}, \qquad \bigg( \int_M |h_0^2|^q \bigg)^{1/q} \leq b_2. \]
Then, there are continuous families $(h^1_t)_{t \in [0,\sigma^2]}, (h^2_t)_{t \in [0,\sigma^2]}$ with $h_t = h^1_t + h^2_t$ such that for $m=0$ and all $t \in [0, \sigma^2]$ or $m \geq 1$ and all $t \in [\frac12 \sigma_0^2, \sigma^2]$ we have $|\nabla^m h^1_t|, |\nabla^m h^2_t| < A_m \varepsilon_m$ and
\[ |\nabla^m h^1_t| \leq \frac{A_m b_1}{r+1+a}, \qquad \bigg( \int_M |\nabla^m h^2_t|^q \bigg)^{1/q} \leq A_m b_2. \]
Moreover, for all $t \in [\frac12 \sigma_0^2, \sigma^2]$ we have $| \nabla^m h^2_t | \leq A_m b_2$.
\end{Lemma}
\begin{proof}
Let $B_1, B_2$ be positive numbers, which will be determined later and assume that all $\varepsilon_m$ are bounded by some constant $\varepsilon_0 > 0$.
By Proposition \ref{Prop:shortex}, we have $|h_t| < C \varepsilon_0$ on $M \times [0, \sigma^2]$ for some $C$.
Hence, there is some \emph{constant} $w \leq C \varepsilon_0$ such that 
\begin{equation} \label{eq:uvw}
|h_t| \leq u + v + w
\end{equation}
for all $t \in [0, \sigma^2]$, where $u, v \in C^\infty(M)$ are nonnegative scalar functions with
\begin{equation} \label{eq:uestvest}
 u = \frac{B_1 b_1}{r+1+a}, \qquad \Vert v \Vert_{L^q(M)} \leq B_2 b_2.
\end{equation}
Suppose that $w$ is chosen almost minimal with this property.
In the following we will show that we can rechoose $v$ such that (\ref{eq:uvw}) even holds for $\frac12 w$.
Hence, by induction it holds for $w = 0$.

Consider some $0 < r_1 \leq \sigma$, set $t_1 = r_1^2$ and for any $x \in M$
\begin{multline*}
 H_{r_1} (x) := r_1^{- \frac1q(n+2)} \Vert h \Vert_{L^q(B_{2r_1}(x) \times [0,t_1])} + r_1^{-\frac1q n} \Vert h_0 \Vert_{L^q(B_{2r_1}(x))} \\
 \leq C u(x) + C w + 2 r_1^{-\frac{n}q } \Vert v \Vert_{L^q(B_{2r_1}(x))}.
\end{multline*}
Observe, that also $H_{r_1}(x) \leq C \varepsilon_0$.
By Lemma \ref{Lem:locestimate} we have the following estimates (for $(x',t') \in B_{r_1}(x) \times [\frac12 t_1, t_1]$ and $m$ not too large)
\begin{alignat*}{1}
 r_1^{-\frac12 (n+2)} \Vert \nabla h \Vert_{L^2(B_{r_1}(x) \times [0, t_1])} &\leq C r_1^{-1-\frac12 (n+2)} \Vert h \Vert_{L^2(B_{2r_1}(x) \times [0,t_1])} \\
 & \qquad + C r_1^{-1 - \frac12 n} \Vert h_0 \Vert_{L^2 ( B_{2 r_1} (x))}  
  \leq C r_1^{-1} H_{r_1}(x), \\
 | \nabla^m h | (x',t') &\leq C_m r_1^{-m-\frac12 (n+2)} \Vert h \Vert_{L^2(B_{2r_1}(x) \times [0,t_1])} \leq C_m r_1^{-m} H_{r_1}(x) . 
\end{alignat*}
From this we obtain estimates on $R = \nabla h * \nabla h$ (again for $(x', t') \in B_{r_1}(x) \times [\frac12 t_1, t_1]$)
\begin{alignat*}{1}
 r_1^{-(n+2)} \Vert R \Vert_{L^1(B_{r_1}(x) \times [0,t_1])} &\leq C r_1^{ - 2} H^2_{r_1}(x) \leq C r_1^{-2} \varepsilon_0 H_{r_1}(x), \\
 |R|(x',t') &\leq C r_1^{ - 2} H^2_{r_1}(x) \leq C r_1^{-2} \varepsilon_0 H_{r_1}(x) .
\end{alignat*}
as well as on $S = h * \nabla h$
\begin{alignat*}{1}
 r_1^{-(n+2)} \Vert S \Vert_{L^1(B_{r_1}(x) \times [0,t_1])} \leq r_1^{-\frac12(n+2)} \Vert S \Vert_{L^2(B_{r_1}(x) \times [0,t_1])} &\leq C r_1^{ - 1} \varepsilon_0 H_{r_1}(x). \\
 |S|(x',t') &\leq C r_1^{-1} \varepsilon_0  H_{r_1}(x).
\end{alignat*}
We now use (\ref{eq:convoleq}) for $x_1 \in M$ and $t_0 = 0$
\begin{multline*}
 h(x_1,t_1) = \int_M k_{t_1}(x_1, x) h(x,0) dx \\ 
 + \int_0^{t_1} \int_M \big( k_{t_1-t}(x_1, x) R (x, t) + \nabla k_{t_1-t}(x_1,x) S (x, t) \big) dx dt.
\end{multline*}
We can bound the first integral $\int_M$ using the fact that $|k_{t_1}|(x_1,x) \leq C \Phi_{r_1} (x_1,x)$ where $\Phi_{r_1} (x_1,x) = r_1^{-n} \exp(-\frac18 r_1^{-2} d^2(x_1,x))$ (see Proposition \ref{Prop:CLY})
\begin{multline*}
 \big| { \textstyle \int_M} \big| \leq \int_M \frac{C b_1 \Phi_{r_1} (x_1,x) dx}{r(x)+1+a} + C \int_M \Phi_{r_1} (x_1, x) |h_0^2|(x) dx \\\leq \frac{C b_1}{r(x_1) + 1 + a} + C \int_M \Phi_{r_1} (x_1, x) |h_0^2|(x) dx.
\end{multline*}
As for the second integral, we split the domain of integration $M \times [0,t_1]$ into two parts: $\Omega = B_{r_1} (x_1) \times [\frac12 t_1, t_1]$ and its complement.
For the integral over $\Omega$, we use the pointwise bounds on $R$ and $S$ as well as the fact that by Proposition \ref{Prop:CLY}
\begin{alignat*}{1}
 \int_{B_{r_1} (x_1) \times [0, \frac12 t_1]} | k_t| (x_1,x)  dx dt &\leq C r_1^2, \\
 \int_{B_{r_1} (x_1) \times [0, \frac12 t_1]} |\nabla k_t| (x_1,x)  dx dt &\leq C r_1
\end{alignat*}
to conclude
\[ \big| {\textstyle \int_{\Omega}} \big| \leq C \varepsilon_0 H_{r_1}(x_1) \leq C \varepsilon_0 u(x_1) + C \varepsilon_0 w + C \varepsilon_0 r_1^{-\frac{n}{q}} \Vert v \Vert_{L^q(B_{2r_1}(x_1))}. \]
On $M \times [0,t_1] \setminus \Omega$, we use the fact that by Proposition \ref{Prop:CLY} we have the bounds $|k_{t_1-t}|(x_1,x) < C \Phi_{r_1}(x_1,x)$ and $|\nabla k_{t_1-t}|(x_1,x) < C r_1^{-1} \Phi_{r_1} (x_1,x)$ to conclude
\[ \big| {\textstyle \int_{M \times [0,t_1] \setminus \Omega}} \big| 
 \leq C  \int_M \Phi_{r_1}(x_1, x) \bigg( \int_0^{t_1} \big( |R(x,t)| + r_1^{-1} |S (x,t) | \big) dt \bigg) dx. \]
Note that for any $x, y \in M$ with $d (x, y) < r_1$ we have
\[ \Phi_{r_1} (x_1, x) \leq C \Phi_{2 r_1} (x_1, y) . \]
So by Fubini's Theorem
\begin{alignat*}{1}
 \big| {\textstyle \int_{M \times [0,t_1] \setminus \Omega}} \big| 
 &\leq C \int_M  r_1^{-n}  \bigg( \int_{ B_{r_1} (x) } \Phi_{2 r_1}(x_1, y)  \\
 & \qquad\qquad\qquad\qquad \cdot \bigg( \int_0^{t_1} \big( |R(x,t)| + r_1^{-1} |S (x,t) | \big) dt \bigg) dy \bigg) dx \displaybreak[1] \\
 &\leq C r_1^{-n} \int_M    \bigg( \int_{ B_{r_1} (y) } \Phi_{2 r_1}(x_1, y)  \\
& \qquad\qquad\qquad\qquad \cdot \bigg( \int_0^{t_1} \big( |R(x,t)| + r_1^{-1} |S (x,t) | \big) dt \bigg) dx \bigg) dy \displaybreak[1] \\
 &\leq C r_1^{-n} \int_M   \Phi_{2r_1}(x_1, y)  \bigg( \int_{B_{r_1} (y) \times [0, t_1]}  \big( |R| + r_1^{-1} |S | \big)  \bigg) dy \\
  & \leq C \varepsilon_0 \int_M \Phi_{2 r_1}(x_1, x)  \cdot H_{r_1}(x) dx \displaybreak[1] \\
& \leq C \varepsilon_0 u(x_1) + C \varepsilon_0 w + C \varepsilon_0 r_1^{-\frac{n}{q}} \int_M \Phi_{2r_1} (x_1, x)  \Vert v \Vert_{L^q (B_{2r_1}(x))} dx.
\end{alignat*}
Hence for some $C_1$
\[ |h|(x_1,t_1) \leq \frac{C_1 (1+ \varepsilon_0 B_1) b_1}{r(x_1) + 1 + a} + C_1 \varepsilon_0 w + \tilde v_{r_1} (x_1). \]
where using $\widehat v_{r_1}(x) = r_1^{-\frac{n}{q}} \Vert v \Vert_{L^q(B_{2r_1}(x))}$
\[ \tilde v_{r_1} (x_1) = C \int_M \Phi_{2 r_1}(x_1,x) \big( |h_0^2|(x) + \varepsilon_0 \widehat v_{r_1} (x) \big) dx + C \varepsilon_0 \widehat{v}_{r_1}(x_1). \]
Set $\tilde v = \sup_{0 < r_1 < \sigma} \tilde v_{r_1}$.
Denote by $\mathcal{M}_{\sigma}$ the Hardy-Littlewood maximal operator up to scale $\sigma$, i.e. for any nonnegative function $f \in C^\infty(M)$, we set 
\[ (\mathcal{M}_{\sigma} f)(x_1) = \sup_{0< r_1 < \sigma} \frac1{\vol B_{r_1}(x_1)} \int_{B_{r_1}(x_1)} f(x) dx. \]
Using this operator, we can write
\begin{multline*}
 \int_M \Phi_{2 r_1}(x_1, x)  f(x) dx 
  = \int_{M \setminus B_{\sigma}(x_1)} \Phi_{2 r_1} (x_1, x) f(x) dx \\
  + \int_{B_{\sigma}(x_1)} (2 r_1)^{-n} e^{- \frac{\sigma^2}{32 r_1^2} } f(x) dx 
 - \int_0^{ \sigma} \frac{d}{dr'} \Big((2 r_1)^{-n} e^{-\frac{(r')^2}{32r_1^2} } \Big) \bigg( \int_{B_{r'}(x_1)} f(x) dx \bigg) dr' .
\end{multline*}
Note that for all $x \in M \setminus B_\sigma (x_1)$ we have $\Phi_{2 r_1} (x_1, x) \leq C \Phi_{2\sigma} (x_1, x)$.
We also have $(2r_1)^{-n} \exp ( - \frac{\sigma^2}{32 r_1^2} ) \leq C \sigma^{-n} \leq C \Phi_{2 \sigma} (x_1, x)$ for any $x \in B_\sigma (x_1)$.
Lastly, the derivative under the last integral sign is non-positive.
So we obtain
\begin{multline*}
 \int_M \Phi_{2r_1}(x_1, x) f(x) dx 
 \leq C \int_M \Phi_{2\sigma} (x_1, x) f(x) dx \\
\qquad\qquad\qquad\qquad\qquad - \big( \mathcal{M}_\sigma f \big)(x_1) \int_{0}^{\sigma}  \frac{d}{dr'} \Big( (2r_1)^{-n} e^{-\frac{(r')^2}{32r_1^2} } \Big)  \bigg( \int_{B_{r'}(x_1)} dx \bigg) dr' \\
 \leq  C \int_M \Phi_{2\sigma} (x_1, x)  f(x) dx + C  \big( \mathcal{M}_\sigma f \big)(x_1)  .
\end{multline*}
Hence
\begin{multline*}
 \tilde v(x_1) \leq \big(\mathcal{M}_{\sigma} ( |h_0^2| + \varepsilon_0 \widehat v_{r_1} ) \big)(x_1) \\
 + C \int_M \Phi_{2 \sigma}(x_1,x)  \big( |h_0^2|(x) + \varepsilon_0 v(x) \big) dx + C \varepsilon_0 \widehat v_{r_1} (x_1).
\end{multline*}
By the Hardy-Littlewood maximal inequality (cf \cite{SW}) and Young's inequality there is some $C_2$ (which is independent of $q$) such that
\begin{multline*}
 \Vert \tilde v \Vert_{L^q(M)} \leq C \Vert h_0^2 \Vert_{L^q(M)} + C \varepsilon_0  \Vert \widehat v_{r_1} \Vert_{L^q(M)} \\
 \leq C_2 \Vert h_0^2 \Vert_{L^q(M)} + C_2 \varepsilon_0  \Vert v \Vert_{L^q(M)} 
 \leq C_2 \big( 1 + \varepsilon_0 B_2 \big) b_2.
\end{multline*}
Now choose $B_1 = 2 C_1$ and $B_2 = 2 C_2$.
Then we can choose $\varepsilon_0$ small enough such that $C_1(1+\varepsilon_0 B_1) \leq B_1$, $C_2 (1+\varepsilon_0 B_2) \leq B_2$ and $C_0 \varepsilon_0 < \frac12$.
We find that $|h_t| \leq u + \tilde v + \tilde w$ with $\tilde w = \frac12 w$ and $\Vert \tilde v \Vert_{L^q(M)} \leq B_2 b_2$.
Iterating this argument shows that we can find $u, v \in C^\infty(M)$ satisfying (\ref{eq:uestvest}) and (\ref{eq:uvw}) for $w=0$.

Now consider such $u$ and $v$ and recall that for every $x \in M$ and $(x',t') \in B_{\sigma_0} (x) \times [\frac12 \sigma_0^2, \sigma^2]$ we have
\[ |\nabla^m h|(x',t') \leq C_m \sigma_0^{-m} H_{\sigma_0} (x) \leq C_m \sigma_0^{-m} \big( u(x)+ \Vert v \Vert_{L^q(B_{2\sigma_0}(x))} \big) \]
and recall that by Corollary \ref{Cor:Shi} we have $|\nabla^m h_t| < C_m \varepsilon_m$ for $t \in [\frac12 \sigma_0^2, \sigma^2]$.
So on $B_{\sigma_0}(x) \times [\frac12 \sigma_0^2, \sigma^2]$, we can find a splitting $h_t = h^1_t + h^2_t$ such that $|\nabla^m h_t^1|, |\nabla^m h_t^2| < C_m \varepsilon_m$ and
\[ | \nabla^m h^1_t | \leq \frac{C_m b_1}{r+a+1}, \qquad  | \nabla^m h^2_t | \leq C_m \Vert v \Vert_{L^q(B_{2\sigma_0}(x))} \leq C_m B_2 b_2. \]
Using a suitable partition of unity, we can glue those splittings together.
Since these estimates are uniform in time, we can extend this splitting to the time interval $[0, \sigma^2]$ such that the zero order bounds hold on $[0, \frac12 \sigma_0^2]$.
\end{proof}

\subsection{Hyperbolic geometry and bounds on the linear equation} \label{subsec:hypgeometry}
We need the following elementary observation.
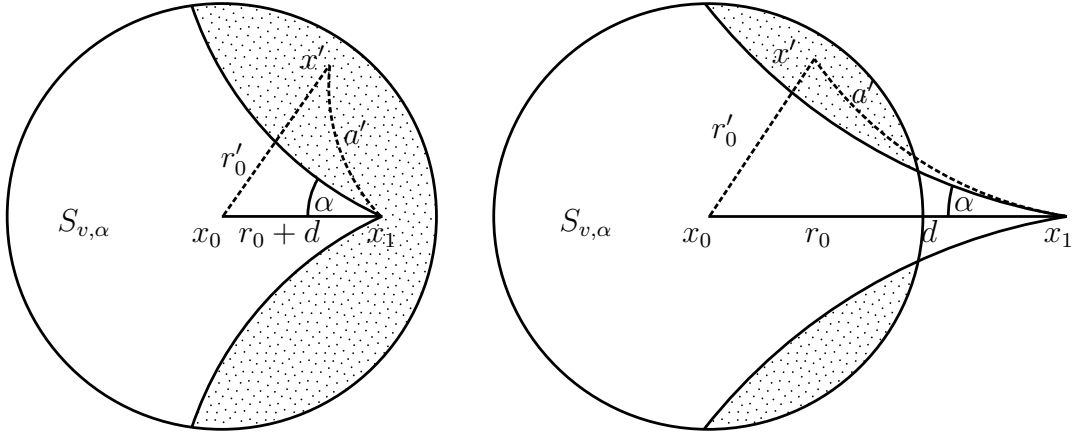
\begin{figure}[t]
\caption{The cases $d < 0$ and $d > 0$.}
\vspace{-5mm}
\begin{center}
\begin{picture}(0,0)%
\hspace{2mm}\includegraphics[width=14cm]{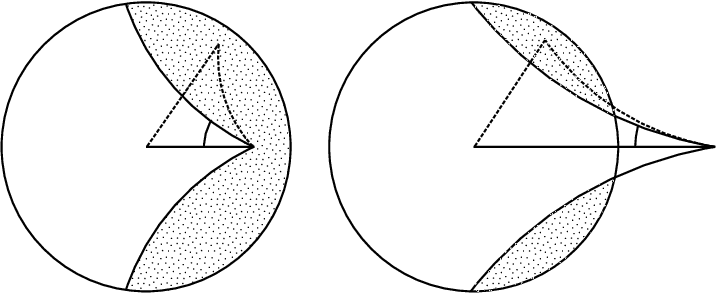}%
\end{picture}%
\setlength{\unitlength}{2863sp}%
\begingroup\makeatletter\ifx\SetFigFont\undefined%
\gdef\SetFigFont#1#2#3#4#5{%
  \reset@font\fontsize{#1}{#2pt}%
  \fontfamily{#3}\fontseries{#4}\fontshape{#5}%
  \selectfont}%
\fi\endgroup%
\begin{picture}(11195,4100)(518,-5050)
\put(2250,-3400){\makebox(0,0)[lb]{\smash{{\SetFigFont{12}{14.4}{\familydefault}{\mddefault}{\updefault}$x_0$}}}}
\put(6450,-3400){\makebox(0,0)[lb]{\smash{{\SetFigFont{12}{14.4}{\familydefault}{\mddefault}{\updefault}$x_0$}}}}
\put(3750,-3400){\makebox(0,0)[lb]{\smash{{\SetFigFont{12}{14.4}{\familydefault}{\mddefault}{\updefault}$x_1$}}}}
\put(9550,-3400){\makebox(0,0)[lb]{\smash{{\SetFigFont{12}{14.4}{\familydefault}{\mddefault}{\updefault}$x_1$}}}}
\put(7220,-1850){\makebox(0,0)[lb]{\smash{{\SetFigFont{12}{14.4}{\familydefault}{\mddefault}{\updefault}$x'$}}}}
\put(3180,-1900){\makebox(0,0)[lb]{\smash{{\SetFigFont{12}{14.4}{\familydefault}{\mddefault}{\updefault}$x'$}}}}
\put(2650,-3400){\makebox(0,0)[lb]{\smash{{\SetFigFont{12}{14.4}{\familydefault}{\mddefault}{\updefault}$r_0+d$}}}}
\put(7500,-3400){\makebox(0,0)[lb]{\smash{{\SetFigFont{12}{14.4}{\familydefault}{\mddefault}{\updefault}$r_0$}}}}
\put(8500,-3400){\makebox(0,0)[lb]{\smash{{\SetFigFont{12}{14.4}{\familydefault}{\mddefault}{\updefault}$d$}}}}
\put(8780,-3130){\makebox(0,0)[lb]{\smash{{\SetFigFont{12}{14.4}{\familydefault}{\mddefault}{\updefault}$\alpha$}}}}
\put(3300,-3130){\makebox(0,0)[lb]{\smash{{\SetFigFont{12}{14.4}{\familydefault}{\mddefault}{\updefault}$\alpha$}}}}
\put(1100,-3300){\makebox(0,0)[lb]{\smash{{\SetFigFont{12}{14.4}{\familydefault}{\mddefault}{\updefault}$S_{v,\alpha}$}}}}
\put(5400,-3300){\makebox(0,0)[lb]{\smash{{\SetFigFont{12}{14.4}{\familydefault}{\mddefault}{\updefault}$S_{v,\alpha}$}}}}
\put(2500,-2750){\makebox(0,0)[lb]{\smash{{\SetFigFont{12}{14.4}{\familydefault}{\mddefault}{\updefault}$r'_0$}}}}
\put(6700,-2500){\makebox(0,0)[lb]{\smash{{\SetFigFont{12}{14.4}{\familydefault}{\mddefault}{\updefault}$r'_0$}}}}
\put(3550,-2550){\makebox(0,0)[lb]{\smash{{\SetFigFont{12}{14.4}{\familydefault}{\mddefault}{\updefault}$a'$}}}}
\put(7900,-2200){\makebox(0,0)[lb]{\smash{{\SetFigFont{12}{14.4}{\familydefault}{\mddefault}{\updefault}$a'$}}}}
\end{picture}%
\end{center}
\end{figure}

\begin{Lemma} \label{Lem:hyperbolicgeometry}
Let $M = \IH^n$ or $\IC \IH^{2n}$.
There are constants $C < \infty$ and $\mu > 0$ such that: \\
Consider two distinct points $x_0, x_1 \in M$ and let $r_0 > 0$, $0 < \alpha < \frac{\pi}2$.
Let $v \in T_{x_1} M$ be the vector pointing towards $x_0$ and define the sector
\[ S_{v, \alpha} = \{ \exp_{x_1} (u) \; \; : \;\; u \in T_{x_1} M, \; \sphericalangle_{x_1} (u,v) \leq \alpha \}. \]
Then for $d = d(x_0, x_1) - r_0$ we have
\[ \vol \big( B_{r_0}(x_0) \setminus S_{v, \alpha} \big) \leq C e^{-\mu d} \alpha^{- 2(n-1)}. \]
\end{Lemma}
\begin{proof}
By rescaling we can assume that the sectional curvatures are $\leq -1$.
We will then show the volume estimate for $\mu = n-1$.

Choose $a$ such that $\sinh a = e^{-d} (1-\cos \alpha)^{-1} \leq C e^{-d} \alpha^{-2}$.
In the following, we will show that 
\[ B_{r_0}(x_0) \setminus S_{v, \alpha} \subset B_{a} (x_1). \]
Since $\vol B_a (x_1) \leq C (\sinh a)^{n-1}$, this will give us the desired estimate.

Consider a point $x' \in B_{r_0}(x_0) \setminus S_{v, \alpha}$.
Let $a' = d (x_1, x')$, $r_0' = d(x_0, x')$, $u \in T_{x_1} M$ such that $\exp_{x_1}(u) = x'$ and $\alpha' = \sphericalangle_{x_1} (u,v) > \alpha$.
By the triangle inequality we have $a' \geq d$.
Consider a comparison triangle $\triangle \ov{x}_0 \ov{x}_1 \ov{x}'$ for the triangle $\triangle x_0 x_1 x'$ in $\IH^2$ and let $\ov{\alpha}'$ be the angle at $\ov{x}_1$.
By triangle comparison, we have $\ov{\alpha}' \geq \alpha' \geq \alpha$ and hence by the law of cosines in $\IH^2$
\[ \cosh r_0' \geq \cosh (r_0+d) \cosh a' - \sinh (r_0+d) \sinh a' \cos \alpha. \]
Moreover, since
\[ \cosh r_0' \leq \cosh r_0 = \cosh(r_0 + d) \cosh d - \sinh (r_0+d) \sinh d, \]
we conclude
\[ \tanh (r_0+d) \big( \sinh a' \cos \alpha - \sinh d \big)  \geq \cosh a' - \cosh d. \]
Observe that since $a' \geq d$, either the right hand side is positive or $d < 0$ and hence the left hand side is positive.
So
\[ \sinh a' \cos \alpha - \sinh d \geq \cosh a' - \cosh d. \]
This implies
\[ \sinh a' (1-\cos \alpha) \leq e^{-d} \]
and hence $\sinh a' \leq \sinh a$, which establishes the claim.
\end{proof}

\begin{Lemma} \label{Lem:linearoneoverr}
Let $M = \IH^n$, $(n \geq 3)$ or $\IC \IH^{2n}$, $(n \geq 2)$ choose a basepoint $x_0 \in M$ and consider the radial distance function $r = d(\cdot, x_0)$.
For every $w > 0$ there is a constant $C = C(w) < \infty$ such that: \\
Assume that $h \in C^\infty(M;\Sym_2 T^*M)$ and that
\[ |h|(x) < \frac1{(r(x)+1+a)^w} \]
for some $a \geq 0$.
Then for all $x_1 \in M$ and $r_1 = r(x_1)$ and $t \geq 0$
\[ \int_M |k_t|(x_1,x) |h|(x)  dx  < \frac{C}{(r_1+1+a+ t)^w}. \]
\end{Lemma}
\begin{proof}
For small times $t \leq 1$, the estimate follows with the help of Proposition \ref{Prop:CLY}.
So assume that $t > 1$.

Recall $\lambda_B > 0$ from subsection \ref{subsec:heatkerIntro}.
If $r_1 + \frac{\lambda_B}\mu t \leq 1+a$, then we find by the $L^1$-boundedness of $k_t$ (cf. Theorem \ref{Thm:dectrianglerank1} and Proposition \ref{Prop:lambdaWWpos})
\[ \int_M |k_t|(x_1,x) |h|(x) dx \leq \frac{C}{(1+a)^w} \leq \frac{C'}{(r_1+ 1 + a + t)^w}. \]
Assume from now on $r_1 + \frac{\lambda_B}\mu t > 1+a$ and hence $r_2 := \frac12 r_1 - \frac12 (1+a) + \frac{\lambda_B}{2\mu} t > 0$.
We can then bound
\[ \int_{M \setminus B_{r_2} (x_0)}  |k_t|(x_1,x) |h| (x) dx \leq \frac{C}{(r_2 + 1 + a)^w} \leq \frac{C'}{(r_1 + 1 + a + t)^w} \]
and hence, it remains to bound the integral on $B_{r_2} (x_0)$.
Set $\alpha = \exp ( - \frac{\mu}{8(n-1)} r_1 - \frac{\lambda_B}{4(n-1)} t )$.
Let $v \in T_{x_1} M$ be the vector that points in the direction of $x_0$ and consider the sector $S_{v, \alpha}$.
By Lemma \ref{Lem:hyperbolicgeometry}, we have
\[ \vol ( B_{r_2} (x_0) \setminus S_{v, \alpha}) \leq C e^{\mu (r_2 - r_1)} \alpha^{-2(n-1)}. \]
So by Cauchy-Schwarz and the bound $\Vert k_t \Vert_{L^2(M)} \leq C e^{-\lambda_B t}$ (cf. (\ref{eq:M2}) in the proof of Theorem \ref{Thm:dectrianglerank1})
\begin{multline*}
 \int_{B_{r_2} (x_0) \setminus S_{v, \alpha}} |k_t|(x_1,x) |h| (x) dx \leq C e^{\frac{\mu}2 (r_2 - r_1)} \alpha^{-(n-1)} e^{-\lambda_B t} \\
= C \exp \big( - \tfrac{\mu}8 r_1 - \tfrac{\mu}4 (1+a) - \tfrac{\lambda_B}2 t \big) \leq \frac{C}{(r_1 + 1 + a + t)^w}.
\end{multline*}
In order to bound the integral on the remaining part $B_{r_2} (x_0) \cap S_{v, \alpha}$, we use the fact that $\Vert k_t \Vert_{L^1(S_{v, \alpha})} \leq C \alpha^{n-1}$ (observe that $k_t$ is spherical and that the set of angles pointing into the sector $S_{v,\alpha}$ at $x_1$ has measure $\sim \alpha^{n-1}$):
\[  \int_{B_{r_2} (x_0) \cap S_{v, \alpha}} |k_t|(x_1,x) |h| (x) dx \leq C \alpha^{n-1} \frac1{(1+a)^w} \leq \frac{C}{(r_1 + 1 + a + t)^w}. \qedhere \]
\end{proof}

\subsection{Proof of Theorem \ref{Thm:mainB}} \label{subsec:ThmmainB}
We will need the following linear estimate:

\begin{Lemma} \label{Lem:Lqdecay}
Assume that $2 \leq q < \infty$ and let $h_0 \in C^\infty(M; \Sym_2 T^*M)$ such that $\Vert h_0 \Vert_{L^q(M)} < \infty$.
Consider $h_t(x) = \int_M k_t(x,x') h_0(x') dx'$, the solution of $\partial_t h_t = - L h_t$.
Then, for $\lambda = \frac2q \lambda_B > 0$ we have
\[ \Vert h_t \Vert_{L^q(M)} \leq C e^{-\lambda t} \Vert h_0 \Vert_{L^q(M)}. \]
\end{Lemma}
\begin{proof}
By the $L^1$-boundedness of the heat kernel, the inequality is true for $q = \infty$ with $\lambda = 0$ and by the Bochner formula (cf (\ref{eq:M2})), it holds for $q = 2$ and $\lambda = \lambda_B$.
Hence, by the Marcinkiewicz interpolation theorem, it holds for any $2 \leq q < \infty$ with $\lambda = \frac2q \lambda_B$.
\end{proof}

\begin{proof}[Proof of Theorem \ref{Thm:mainB}]
Observe first that the Theorem is more general for larger $q$.
Hence, we can assume $q \geq 2$.

Let $\varepsilon_0 > 0$ be a small constant, which we will determine in the course of the proof.
As in the proof of Theorem \ref{Thm:mainA}, let $T_{\max}$ be the maximum over all $T$ such that Ricci deTurck flow $h_t$ exists on $[0,T)$ and satisfies $|h_t| < \varepsilon_0$ everywhere.
We will show that for sufficiently small $\varepsilon$ (independent of $T_{\max}$) we even have $|h_t| < \varepsilon_0/2$ and hence $T_{\max} = \infty$.
Recall that by Proposition \ref{Prop:shortex} we have $T_{\max} > \tau_{s.e.}$.

Let $\varepsilon_1, \varepsilon_2 > 0$ be constants whose value we will fix at the end of the proof.
By Lemma \ref{Lem:h1h2splitting} (applied successively to the time intervals $[0,\tau_{s.e.}], [\frac12 \tau_{s.e.}, \frac32 \tau_{s.e.}], \linebreak[1] \ldots$), we conclude that for every $t \in [0, T_{\max})$ there is a splitting $h_t = h^1_t + h^2_t$ with
\[ h_t^1 \leq \frac{\varepsilon_1 Y'_t}{r+1+t}, \qquad \Big( \int_M |h_t^2|^q \Big)^{1/q} \leq \varepsilon_2 Y'_t e^{-\frac{\lambda}2 t} \]
for some $Y'_t < \infty$ (here $\lambda$ is the constant from Lemma \ref{Lem:Lqdecay}).
For every $t \in [0, T_{\max})$ let $Y_t$ be the infimum over all possible $Y'_t$ for all splittings $h_t = h_t^1 + h_t^2$ and set $Z_t = \max_{t' \in [0,t]} Y_{t'}$.

By Lemma \ref{Lem:h1h2splitting} we have $Z_{\tau_{s.e.}} \leq C_0$.
Applying Lemma \ref{Lem:h1h2splitting} at positive times, we find that for any $t_1, t_2 \in [0,T_{\max})$, we have $Z_{t_2} \leq C_1 Z_{t_1}$ whenever $t_2 < t_1 + \tau_{s.e.}$.
Moreover, we conclude that for times $[\tau_{s.e.}, T_{\max})$ we can rechoose $h^1_t$ and $h^2_t$ piecewise continuously in time such that $h_t = h^1_t + h^2_t$ and for $m = 0,1,2$ we have $|\nabla^m h^1_t|, |\nabla^m h^2_t| < A \varepsilon_0$ and
\[ |\nabla^m h^1_t| \leq \frac{A \varepsilon_1 Z_t}{r + 1 + t}, \qquad \sup_M |\nabla^m h_t^2| + \Big( \int_M |\nabla^m h^2_t|^q \Big)^{1/q} \leq A \varepsilon_2 Z_t e^{-\frac{\lambda}2 t}. \]
Hence, since  $|Q_t| \leq C ( |\nabla h_t|^2 + |h_t| |\nabla^2 h_t| )$, we find that $Q_t = Q^1_t + Q^2_t$ with
\[ |Q^1_t| \leq \frac{C \varepsilon_1^2 Z_t^2}{(r + 1 + t)^2}, \qquad \Big( \int_M |Q^2_t|^q \Big)^{1/q} \leq C  \varepsilon_2^2 Z_t e^{-\frac{\lambda}2 t}. \]

Let now $t_1 \in [\tau_{s.e.},T_{\max})$, $x_1 \in M$ and $r_1 = r(x_1)$ and recall from (\ref{eq:convoleq}) that
\[ h_{t_1}(x_1) = \int_M k_{t_1 - \tau_{s.e.}} (x_1, x) h_{\tau_{s.e.}} (x) dx + \int_{\tau_{s.e.}}^{t_1} \int_{M} k_{t_1-t}(x_1,x) Q_t(x) dx dt \]
Hence, $h_{t_1} = \tilde h_{t_1}^1 + \tilde h_{t_1}^2$, where ($i=1,2$)
\[ \tilde h_{t_1}^i (x_1) = \int_M k_{t_1-\tau_{s.e.}} (x_1,x) h^i_{\tau_{s.e.}} (x) dx + \int_{\tau_{s.e.}}^{t_1} \int_M k_{t_1-t} (x_1,x) Q^i_t (x) dx dt. \]
We will estimate $\tilde h^1_{t_1}(x_1)$ and $\tilde h^2_{t_1} (x_1)$.

Observe first that by Lemma \ref{Lem:linearoneoverr} for $w=1$
\[ \Big| \int_M k_{t_1 - \tau_{s.e.}} (x_1, x) h^1_{\tau_{s.e.}}(x) dx \Big| \leq \frac{C \varepsilon_1}{r_1 + 1 + t_1} \]
and for $w=2$ and $t \in [\tau_{s.e.}, t_1]$
\[ \Big| \int_M k_{t_1-t}(x_1,x) Q^1_t(x) dx \Big| \leq \frac{C \varepsilon_1^2 Z_t^2}{(r_1 + 1 + t_1)^2} \leq \frac{C \varepsilon_1^2 Z_t^2}{t_1 (r_1 + 1 + t_1)}. \]
Hence
\[ |\tilde h^1_{t_1}| (x_1)  \leq \frac{C_2 ( \varepsilon_1 + \varepsilon_1^2 Z_{t_1}^2)}{r_1+1+t_1} . \]

Secondly, by Lemma \ref{Lem:Lqdecay} we find
\begin{multline*}
 \Big( \int_M |\tilde h^2_{t_1}|^q  \Big)^{1/q} \leq C \varepsilon_2 e^{- \lambda t_1} + C_3 \varepsilon_2^2 Z_{t_1} \int_{\sigma^2}^{t_1} e^{- \lambda (t_1-t)} e^{- \frac{\lambda}2 t} dt \\
\leq (C_3 \varepsilon_2 + C_4(q) \varepsilon_2^2 Z_{t_1} ) e^{-\frac{\lambda}2 t_1}.
\end{multline*}
Here $C_4(q)$ depends on $\lambda$ and hence on $q$.
By the minimality of $Y_{t_1}$ we conclude
\[ Y_{t_1} \leq \max \{ C_2 (1 + \varepsilon_1 Z_{t_1}^2), C_3 + C_4(q) \varepsilon_2 Z_{t_1} \}. \]
Let $C_5 = \max \{ C_0, C_2, C_3 \}$ and observe that
\[
 Z_{\tau_{s.e.}} \leq C_5 \qquad \text{and} \qquad Z_{t} \leq \max \{ C_5 ( 1+ \varepsilon_1 Z_{t}^2), C_5 + C_4(q) \varepsilon_2 Z_{t} \}.
\]

Now set $\varepsilon_1 = (2 C_1 C_5)^{-2}$ and $\varepsilon_2(q) = (2 C_1 C_4(q))^{-1}$.
If $Z_t \leq 2 C_5$ did not hold for all $t \in [\tau_{s.e.}, T_{\max})$, then there must be a jump, i.e. two times $t_1 < t_2$ with $t_2 - t_1 < \tau_{s.e.}$ such that $Z_{t_1} \leq 2 C_5$, but $Z_{t_2} > 2 C_5$.
By the fact that $Z_{t_2} \leq C_1 Z_{t_1} \leq 2 C_1 C_5$, we find
\[ 2 C_5 < Z_{t_2} \leq \max \{ C_5 ( 1 + \varepsilon_1 (2 C_1 C_5)^2), C_5 + C_4(q) \varepsilon_2 (2 C_1 C_5) \} = 2C_5, \]
a contradiction.
Hence, we have $Z_t \leq 2 C_5$ for all $t \in [\tau_{s.e.}, T_{\max})$ and the claim follows.
\end{proof}

\end{document}